\documentclass[reqno,english,11pt]{amsart}
\usepackage{times}

\usepackage{amsmath,amsfonts,amssymb,graphicx,amsthm,enumerate,url,mathabx}
\usepackage[noadjust]{cite}
\usepackage{stmaryrd}
\usepackage{mathrsfs,booktabs,tabularx,setspace}
\usepackage{xifthen,xcolor,tikz}
\usetikzlibrary{arrows}
\usetikzlibrary{decorations.pathmorphing,patterns,shapes,calc,decorations}
\usetikzlibrary{decorations.pathreplacing}
\usepackage{mathtools,upgreek}
\usepackage[final]{showkeys} %add in 'final' into parameter to remove showkeys

\definecolor{refkey}{gray}{.75}
\definecolor{labelkey}{gray}{.5}
\usepackage[algo2e,boxed,vlined,algoruled]{algorithm2e}

\usepackage[letterpaper]{geometry}
\newcommand{\midarrow}{\tikz \draw[-triangle 45] (0,0) -- ++(.1,0);}
\newcommand{\upangledarrow}{\tikz \draw[-triangle 45] (0,0) -- ++(.1,.05);}
\newcommand{\downangledarrow}{\tikz \draw[-triangle 45] (0,0) -- ++(.1,-.05);}

\usepackage[colorinlistoftodos]{todonotes}
\presetkeys{todonotes}{inline, color=green}{}

\usepackage{accents}
\usepackage{float}

\usepackage[colorlinks=true]{hyperref}
\colorlet{DarkGreen}{green!50!black}
\colorlet{DarkGray}{gray!60!black}
\usepackage{subcaption}

\numberwithin{equation}{section}

\newcommand{\ignore}[1]{}

\renewcommand{\epsilon}{\varepsilon}

\newcommand{\one}{\mathbf{1}}
\newcommand{\zero}{\mathbf{0}}

\usepackage{color}
 \definecolor{refkey}{gray}{.5}
 \definecolor{labelkey}{gray}{.5}
\definecolor{light}{gray}{.9}

\newtheorem{theorem}{Theorem}[section]
\newtheorem*{theorem*}{Theorem}
\newtheorem{lemma}[theorem]{Lemma}

\newtheorem{claim}[theorem]{Claim}
\newtheorem{proposition}[theorem]{Proposition}

\newtheorem{observation}[theorem]{Observation}

\newtheorem{corollary}[theorem]{Corollary}

\theoremstyle{definition}{

\newtheorem{definition}[theorem]{Definition}

\newtheorem*{definition*}{Definition}

\newtheorem{remark}[theorem]{Remark}

}

\newcommand{\plusone}{\pmb{+1}}
\newcommand{\plus}{\pmb{+}}
\newcommand{\minusone}{\pmb{-1}}
\newcommand{\minus}{\pmb{-}}

\newcommand{\cC}{\ensuremath{\mathcal C}}

\newcommand{\cE}{\ensuremath{\mathcal E}}

\newcommand{\cG}{\ensuremath{\mathcal G}}

\newcommand{\cT}{\ensuremath{\mathcal T}}

\newcommand{\Ext}{{\mathsf{Ext}}}
\newcommand{\Int}{{\mathsf{Int}}}

\newcommand{\Prrg}{{\mathbb{P}}_{\textsc{rrg}}}

\newcommand{\tmix}{t_{\textsc{mix}}}

\newcommand{\rc}{\textsc{rc}}
\newcommand{\es}{\textsc{es}}

 \renewcommand{\epsilon}{\varepsilon}

\newcommand{\tv}{{\textsc{tv}}}

\makeatletter
\newcommand{\superimpose}[2]{%
  {\ooalign{$#1\@firstoftwo#2$\cr\hfil$#1\@secondoftwo#2$\hfil\cr}}}
  
\newcommand{\sbullet}{%
  \hbox{\fontfamily{lmr}\fontsize{.4\dimexpr(\f@size pt)}{0}\selectfont\textbullet}}

\makeatother

\begin{document}

\title{Low-temperature Ising dynamics with random initializations}

\author[Reza Gheissari]{Reza Gheissari$^\dag$}
% \address{R.\ Gheissari\hfill\break
% Departments of Statistics and EECS \\ UC Berkeley }
% \email{gheissari@berkeley.edu}
\thanks{$^\dag$ Departments of Statistics and EECS, University of California, Berkeley. Email: gheissari@berkeley.edu.}

\author[Alistair Sinclair]{Alistair Sinclair$^\ddag$}
% \address{A.\ Sinclair\hfill\break
% Computer Science Division \\ UC Berkeley }
% \email{sinclair@berkeley.edu}
\thanks{$^\ddag$ Computer Science Division, University of California, Berkeley. Email: sinclair@cs.berkeley.edu.}

\maketitle

\vspace{-.5cm}
\begin{abstract}
It is well known that Glauber dynamics on spin systems typically suffer exponential slowdowns
at low temperatures.  This is due to the emergence of multiple metastable phases in the state space, separated
by narrow bottlenecks that are hard for the dynamics to cross.  It is a folklore belief that if the dynamics is
initialized from an appropriate random mixture of ground states, one for each phase, then convergence to the
Gibbs distribution should be much faster. However, such phenomena have largely evaded rigorous analysis, as most tools in the study of Markov chain mixing times are tailored to worst-case initializations.    

In this paper we develop a general framework towards establishing this conjectured behavior for the Ising model.
In the classical setting of the Ising model on
an $N$-vertex torus in $\mathbb Z^d$, our framework implies that the mixing time for the Glauber dynamics, initialized from a $\frac 12$-$\frac 12$ mixture 
of the all-plus and all-minus configurations, is $N^{1+o(1)}$ in dimension $d=2$, and at most quasi-polynomial in all dimensions $d\ge 3$, at all temperatures below the critical one. 
The key innovation
in our analysis is the introduction of the notion of ``weak spatial mixing within a phase'', a low-temperature adaptation of the
classical concept of weak spatial mixing. We show both that this new notion is strong enough to control the mixing time from the above random initialization (by relating it to the mixing time with plus boundary condition at $O(\log N)$ scales), and that it holds at all low temperatures in all dimensions. 

This framework naturally extends to much more general families of graphs. To illustrate this, we also use the same approach to establish optimal $O(N\log N)$ mixing for the Ising Glauber dynamics on random regular graphs at sufficiently low temperatures, when initialized from the same random mixture.
\end{abstract}

\vspace{-.25cm}
\section{Introduction}
It is well known that Glauber dynamics (local Markov chains) on spin systems (such as the Ising, Potts, and hard-core models, graph colorings, etc.) suffer an exponential slowdown
at low temperatures.  This is due to the emergence of multiple phases in the state space, which are separated
by narrow bottlenecks that are hard for the dynamics to cross.  Much effort has been devoted to overcoming
this obstacle, including the Swendsen-Wang dynamics~\cite{SW} (which allows large-scale, non-local moves), various dynamics on alternative representations of the spin system
(including the subgraph dynamics~\cite{JSIsing}, polymer dynamics~\cite{CGGPS} and 
random-cluster dynamics~\cite{Grimmett}) and non-dynamical
methods based on the cluster expansion~\cite{HPR-Algorithmic-Pirogov-Sinai}.  
(See Section~\ref{subsec:related} below for a summary of what is known about these methods.) 

However, there is a much more ``obvious'' solution to this problem, at least in cases where the phases and their respective 
{\it ground states\/} (maximum-likelihood configurations) are well understood: simply initialize the standard Glauber
dynamics to be in a random mixture of the ground states (one for each phase), and run it as usual.
The intuition is that, presumably, the main obstacle to rapid mixing is the slow transitions 
between phases, so we should expect it to converge rapidly ``within each phase''; since at low temperatures the overall
probability distribution on configurations is approximated by a mixture of the single-phase distributions, this should suffice for 
global convergence.  
This intuition is valid in the special case of the mean-field Ising model (i.e., on 
the complete graph)~\cite{LLP}, where the dynamics reduces to a 1D process. 
But, as we shall explain later, it is rather more subtle on more complex geometries.

In this paper, we make progress towards establishing the validity of this intuition for the Ising model on more complicated geometries. For concreteness we focus our discussion on the extensively studied setting of the $d$-dimensional integer lattice $\mathbb Z^d$. There, we show that the mixing time, started from a random ground-state
initialization, is quasi-linear in dimension $d=2$ and quasi-polynomial in all 
dimensions $d\ge 3$ throughout the low-temperature regime; this should
be contrasted with the exponential mixing time for worst-case initializations in the same regime.  The methodology we develop is useful for establishing the analogous paradigm on more general families of graphs; as an illustration, we show that the same approach yields optimal mixing time bounds for the Ising Glauber dynamics from the random ground-state initialization on random regular graphs at sufficiently low temperatures.

To state our results more precisely, we first remind the reader of the definition of the Ising model.  Given a finite
graph $G=(V,E)$, configurations of the Ising model are assignments $\sigma:V\to \{\pm 1\}$
of one of two {\it spins}, denoted $\pm 1$, to each vertex.  The probability that the model is in configuration~$\sigma$
is specified by the {\it Gibbs distribution at inverse temperature} $\beta>0$: 
\begin{equation}\label{eqn:gibbs}
   \pi(\sigma) = \frac{1}{\mathcal Z_{G,\beta}}\, \exp(-\beta|C(\sigma)|),
\end{equation}
where $C(\sigma)$ is the set of edges $\{u,v\}\in E$ connecting vertices of different spins
(i.e., edges in the cut induced by the spins)
and $\mathcal Z_{G,\beta}$ is the normalizing factor, or partition function.
Note that the distribution~\eqref{eqn:gibbs} favors configurations with fewer cut edges,
and this bias increases with the parameter~$\beta >0$. The Gibbs distribution can be seen as a symmetric mixture of its restrictions to two phases: the {\it plus\/} phase, in which the {\it magnetization\/} (excess of $+1$ over $-1$ spins) is non-negative, and the {\it minus\/} phase, in which the magnetization is non-positive. Whereas at high temperatures (small~$\beta$) this perspective is uninformative, at low temperatures (large $\beta$) the Gibbs distribution becomes bimodal, one mode for each phase, and the boundary between the phases (the set of configurations with magnetization zero) has exponentially small weight.

\begin{figure}[t]
\begin{subfigure}[b]{.32\textwidth}
\begin{tikzpicture}[scale = .5]
\node at (0,0) {\includegraphics[width = 5cm]{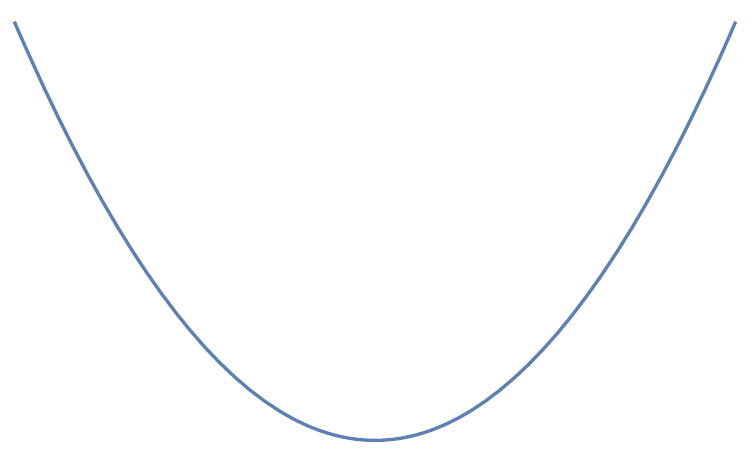}};
\draw[|-|] (-5,-2.8)--(5,-2.8);
\draw[->] (0,-2.8)--(0,3);
\shade[ball color = red!40, opacity = 1] (0,-2.525) circle (.25cm);
\draw [->] (.3,-2.5) to [out=10,in=220] (1,-2.2);
\draw [->] (-.3,-2.5) to [out=180-10,in=180-220] (-1,-2.2);
\end{tikzpicture}
\subcaption{$\beta<\beta_c$}
\end{subfigure}
\begin{subfigure}[b]{.32\textwidth}
\begin{tikzpicture}[scale = .5]
\node at (0,0) {\includegraphics[width = 5cm]{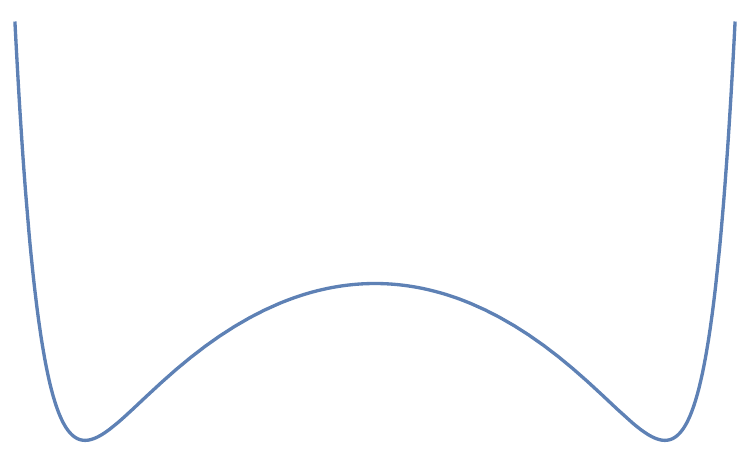}};
\draw[|-|] (-5,-2.8)--(5,-2.8);
\draw[->] (0,-2.8)--(0,3);
\shade[ball color = red!40, opacity = 1] (4.525,2.7) circle (.25cm);
\shade[ball color = red!40, opacity = 1] (-4.525,2.7) circle (.25cm);
\draw [->] (4.525,2.4) to [out=-95,in=85] (4.475,1.7);
\draw [->] (-4.525,2.4) to [out=-85,in=95] (-4.475,1.7);
\end{tikzpicture}
\subcaption{$\beta>\beta_c$}
\end{subfigure}
\begin{subfigure}[b]{.32\textwidth}
\begin{tikzpicture}[scale = .5]
\node at (0,0) {\includegraphics[width = 5cm]{Low-temp-FE.pdf}};
\draw[|-|] (-5,-2.8)--(5,-2.8);
\draw[->] (0,-2.8)--(0,3);
\draw [->] (4.525,2.4) to [out=-95,in=85] (4.475,1.7);
\shade[ball color = red!40, opacity = 1] (4.525,2.7) circle (.25cm);
\draw[rectangle, fill = gray, opacity = .4] (-5,-2.8)--(0,-2.8)--(0,3)--(-5,3)--(-5,-2.8);
\end{tikzpicture}
\subcaption{$\beta>\beta_c$}
\end{subfigure}
\caption{The normalized magnetization $\frac{1}{|V|}\sum_{v}\sigma_v$ (on the $x$-axis) plotted against $F_\beta(m)$, the negative logarithm of the probability that $\frac{1}{|V|}\sum_{v}\sigma_v=m$ (on the $y$-axis). 
\\
\\
(\textsc{a}) At high temperatures, $F_\beta(m)$ is minimized at $m=0$, and is strictly convex. The magnetization does not pose any obstacle to mixing, and the Glauber dynamics can mix rapidly from every initialization. (\textsc{b}) At low temperatures, $F_\beta(m)$ is bimodal, being minimized at $\pm m_*(\beta)$, and it takes exponentially long for the dynamics to transition from one mode to the other. However, $F_\beta(m)$ is locally convex around $\pm m_*(\beta)$, and thus the global magnetization should not be an obstacle to mixing if one starts in a symmetric mixture 
of the all $+1$ and all $-1$ configurations (though there may be other, more geometric, obstructions: see Remark~\ref{rem:worst-case-mixing-restricted-chain} and Figure~\ref{fig:restricted-chain-bottlenecks}).  
(\textsc{c}) The restricted dynamics, considered in Theorem~\ref{thm:Ising-plus-phase}, rejects any transition that would make the magnetization negative (region shaded gray). This restriction ensures that
again, the global magnetization is not an obstacle to mixing of the restricted dynamics initialized from the $+1$ configuration.}
\label{fig:magnetization-schematic}
\end{figure}

We center our discussion on the classical setting where the graph $G$ is a $d$-dimensional
torus $(\mathbb Z/n\mathbb Z)^d$ (or, equivalently, a cube of side length~$n$ in $\mathbb{Z}^d$
with ``periodic" boundaries) for $d\ge 2$; generalizations to other geometries are discussed in Section~\ref{subsec:other-geometries}.  We write $N=n^d$ for the number of vertices in~$(\mathbb Z/n\mathbb Z)^d$.
In this physically relevant setting, the separation of the two phases alluded to above is known to occur at 
a {\it critical value} $\beta=\beta_c(d)$. More precisely, for $\beta<\beta_c(d)$ the correlation
between spins $\sigma_u,\sigma_v$ goes to zero as the distance between~$u$ and~$v$ goes to infinity, and the (normalized) magnetization satisfies a central limit theorem about zero. For $\beta>\beta_c(d)$ correlations remain uniformly bounded away from zero, and the normalized
magnetization converges to a $\frac{1}{2}$-$\frac{1}{2}$ mixture of two point-distributions at~$\pm m_*(\beta)$. We refer to Figure~\ref{fig:magnetization-schematic} for a schematic of this phase transition and how it relates to mixing of the Glauber dynamics. 
We emphasize that this picture is only heuristic, as the geometry of the configurations is lost when projecting onto the magnetization; see
Remark~\ref{rem:worst-case-mixing-restricted-chain} below for a detailed
discussion of this point.

The Glauber dynamics for the Ising model is a Markov chain on configurations that is very simple to describe.
At each step, a vertex~$v\in V$ is chosen uniformly at random and the spin $\sigma_v$ is replaced by a 
random spin selected according to the correct conditional distribution given the spins on the neighbors of~$v$.
This Markov chain is ergodic and reversible w.r.t.\ the Gibbs distribution~\eqref{eqn:gibbs} and hence converges
to it.   In addition to providing a natural and easy-to-implement algorithm for sampling from the Gibbs distribution~\eqref{eqn:gibbs}, the dynamics mimics the thermodynamic
evolution of the underlying physical system, and is thus of interest in its own right.

The key quantity associated with the Glauber dynamics is its rate of convergence, or {\it mixing time}, which
measures the time until the total variation distance to the stationary distribution is small starting from a {\it worst-case\/} initial distribution.
Classical results have established that the mixing time is $\Theta(N\log N)$ throughout the 
high-temperature regime $\beta<\beta_c(d)$~\cite{MaOl1,LS-information-percolation},
and exponentially slow ($\exp(\Theta(N^{1-1/d}))$) throughout the low-temperature 
regime $\beta>\beta_c(d)$~\cite{Pisztora96,Bodineau05} (where the constants in the $\Theta$ notation here depend on $\beta$ and~$d$).
In this latter case, the slow mixing is exhibited by a bottleneck between the plus and minus phases:
with high probability, the time for the dynamics initialized from the $\plusone$ configuration to enter the minus phase 
is exponentially large. Here and throughout the paper, we will use $\plusone$ (respectively, $\minusone$) to
denote a configuration consisting of all $+1$ spins (respectively, all $-1$ spins).

Our main result proves that, if we initialize the dynamics on the torus from the $\plusone$ or $\minusone$
configurations with probability $\frac 12$ each (denoted by $\nu^{\pmb{\pm}}$), then the mixing time is quasi-linear in dimension $d=2$ and at most quasi-polynomial 
in $d\ge 3$.\footnote{By ``mixing time'' here
we mean the time~$\tmix$ until the variation distance from stationarity
is at most~$\frac{1}{4}$, as in the standard definition from a worst-case initial distribution.  In the standard setting,
this implies that the time to reach variation distance~$\varepsilon$ is $O(\tmix\log\varepsilon^{-1})$.  In our setting,
due to our {\it specific\/} initial distribution, this ``boosting'' no longer automatically holds.  However, all our mixing time bounds in this paper suffice
to achieve any variation distance $\varepsilon\ge 1/{\rm poly}(N)$.
See Section~\ref{sec:relaxation-within-phase} for details of the $\varepsilon$-dependence of the bounds in Theorem~\ref{thm:Ising-main}.}

\begin{theorem}\label{thm:Ising-main}
Fix $d\ge 2$ and any $\beta>\beta_c(d)$. The mixing time of the Glauber dynamics for the Ising model on $(\mathbb Z/n\mathbb Z)^d$ initialized from the distribution $\nu^{\pmb{\pm}}$ is 
\begin{itemize}
    \item $N^{1+o(1)}$ in dimension $d=2$;
    \item $N^{O(\log^{d-1} N)}$ in dimensions $d\ge 3$. 
\end{itemize}
\end{theorem}
Recall that the worst-case mixing time is exponential in $N^{1-1/d}$ for all $d\ge 2$ and $\beta>\beta_c(d)$; to the best of our knowledge, there was previously no sub-exponential
bound known for well-chosen initializations.

\subsection{Proof approach} To prove Theorem~\ref{thm:Ising-main}, we introduce a novel notion that we call \emph{weak spatial mixing
within a phase}, which we believe to be of independent interest.   Weak spatial mixing (WSM) is a classical notion capturing the decay of correlations in spin systems, and has been extensively used in the analysis of these spin systems and their dynamics. In the case of the Ising model on~$(\mathbb Z/n\mathbb Z)^d$, WSM says (informally) that for every vertex $v$, the influence of the spins at a distance greater than~$r$ from~$v$ on the distribution at $v$ 
decays exponentially with~$r$.  I.e., for every~$n$ and $r<n/2$, say,
\begin{equation}\label{eqn:wsmintro}
   \max_{\tau: B_{r}^c(v) \to \{\pm 1\}} \| \pi(\sigma_v\in \cdot \mid \sigma_{B_{r}^c(v)}=\tau) - \pi(\sigma_v\in \cdot ) \|_{\tv} \le Ce^{-r/C}
\end{equation}   
for some constant~$C$, where $B_{r}^c(v)$ denotes the set of vertices outside the $\ell_\infty$ ball of radius~$r$ centered at~$v$, and $\Vert\cdot\Vert_\tv$ is total variation distance.
The phase transition in the Ising model corresponds to the fact that WSM holds for all $\beta<\beta_c(d)$ and
breaks down as soon as $\beta \ge \beta_c(d)$.  To define {\it WSM  within a phase},
we let $\widehat\pi$ denote the Gibbs
distribution on $(\mathbb Z/n\mathbb Z)^d$ restricted to the plus phase (i.e., conditioned on having non-negative magnetization), and
replace~\eqref{eqn:wsmintro} by the requirement that, for every $n$ and $r< n/2$, 
\begin{equation}\label{eqn:wsmphaseintro}
   \Vert \pi(\sigma_v\in \cdot \mid\sigma_{B_{r}^c(v)} = \plusone) - \widehat\pi(\sigma_v \in \cdot)\Vert_{\tv} \le Ce^{-r/C}.
\end{equation}
(By spin-flip symmetry, if~\eqref{eqn:wsmphaseintro} holds then the analogous bound for the minus phase also holds.)
Note that~\eqref{eqn:wsmphaseintro} is a much weaker condition than~\eqref{eqn:wsmintro} as it requires correlations
to decay only in the plus phase, with $+1$ spins outside $B_r(v)$---loosely, it can be viewed as a ``monotone version"
of~\eqref{eqn:wsmintro}. 

Armed with this new notion, we prove Theorem~\ref{thm:Ising-main} in two stages.  The first stage concerns
the Glauber dynamics restricted to a single phase, say the plus phase, by ignoring all  updates that would make the magnetization
negative (so that its stationary distribution is~$\widehat\pi$).  We show that if WSM within a phase 
holds, this restricted dynamics mixes in subexponential time when initialized from its ground state. 

\begin{theorem}\label{thm:Ising-plus-phase}
If $\beta>\beta_c(d)$ and WSM within a phase holds, then the mixing time of the Ising Glauber dynamics on $(\mathbb Z/n\mathbb Z)^d$ restricted to the plus phase, initialized from $\plusone$, satisfies the bounds of Theorem~\ref{thm:Ising-main}.
\end{theorem}

\begin{remark}\label{rem:worst-case-mixing-restricted-chain}
We emphasize the delicate point that Theorem~\ref{thm:Ising-plus-phase} (and therefore also
Theorem~\ref{thm:Ising-main}) relies
crucially on having a ``nice" initialization within each phase, namely the ground-state
configuration~$\plusone$ (or $\minusone$).
In contrast to the heuristic picture suggested in Figure~\ref{fig:magnetization-schematic} and applicable for the Ising model on the complete graph~\cite{LLP}, in settings with non-trivial geometry such as $\mathbb Z^d$, the worst-case mixing time of even the restricted dynamics can be much slower.    
Indeed, consider configurations on $(\mathbb Z/n\mathbb Z)^d$ consisting of a strip of $-1$ spins of small but macroscopic width, surrounded on both sides by $+1$ spins (see
Figure~\ref{fig:restricted-chain-bottlenecks}).  At very low temperatures, the mixing time of the restricted dynamics from such an initialization is polynomially slower than $N^{1+o(1)}$ when $d=2$, and more dramatically, is expected to be exponentially slow when $d\ge 3$ (by analogy
with its solid-on-solid approximation: see~\cite{CLMST14}).
Theorem~\ref{thm:Ising-plus-phase} demonstrates that all such bottlenecks are avoided when the dynamics is started from the $\plusone$ configuration. 
\end{remark}

\begin{figure}
    \centering
    \begin{subfigure}[b]{.44\textwidth}
    \centering
\begin{tikzpicture}[scale = .8]
	\draw[fill = blue, opacity = .4] (0,0)--(5,0)--(5,5)--(0,5)--(0,0);
	\draw[white, fill = white, opacity = 1] (3.25,0)--(4,0)--(4,5)--(3.25,5)--(3.25,0);
	\draw[fill = red, opacity = .4] (3.25,0)--(4,0)--(4,5)--(3.25,5)--(3.25,0);
	\node at (1.25,2.5) {$\plusone$};
	\node at (3.6,2.5) {$\minusone$};
	\begin{scope}[every node/.style={sloped,allow upside down}];
     \draw (0,0)-- node {\midarrow} (5,0);
     \draw (0,5)-- node {\midarrow} (0,0);
     \draw (0,5)-- node {\midarrow} (5,5);
     \draw (5,5)-- node {\midarrow} (5,0);
     \end{scope}
\end{tikzpicture}
\subcaption{$d=2$}
\end{subfigure}
    \begin{subfigure}[b]{.44\textwidth}
    \centering
\begin{tikzpicture}[scale = .8]
\shade[yslant=-0.5,right color=blue!20, left color=blue!50]
(0,0) rectangle +(3,3);

\shade[yslant=-0.5,right color=red!20, left color=red!50]
(1.75,0) rectangle +(.5,3);
\shade[yslant=0.5,right color=blue!70,left color=blue!20]
(3,-3) rectangle +(3,3);
\shade[yslant=0.5,xslant=-1,bottom color=blue!20,
top color=blue!80] (6,3) rectangle +(-3,-3);
\shade[yslant=0.5, xslant = -1, bottom color=red!20,
top color=red!80]
(3,1.25) rectangle +(3,-.5);
\node at (2,3.25) {$\plusone$};
\node at (.75,1.5) {$\plusone$};
\node at (2,.75) {$\minusone$};
\node at (3.25,2.65) {$\minusone$};
\node at (4.5,.75) {$\plusone$};

	\begin{scope}[every node/.style={sloped,allow upside down}, yslant=-.5];
     \draw (0,0)-- node {\downangledarrow} (3,0);
     \draw (0,3)-- node {\midarrow} (0,0);
     \draw (0,3)-- node {\downangledarrow} (3,3);
     \draw (3,3)-- node {\midarrow} (3,0);
     \draw (3,6)-- node {\downangledarrow} (6,6);
     \end{scope}
     
     \begin{scope}[every node/.style={sloped,allow upside down}, yslant=.5];
     \draw (3,-3)-- node {\upangledarrow} (6,-3);
     \draw (3,0)-- node {\midarrow} (3,-3);
     \draw (3,0)-- node {\upangledarrow} (6,0);
     \draw (6,0)-- node {\midarrow} (6,-3);
\end{scope}

     \begin{scope}[every node/.style={sloped,allow upside down}, yslant=.5, xslant = -1];
     \draw (3,3)-- node {\upangledarrow} (6,3);
    %  \draw (3,0)-- node {\midarrow} (3,-3);
    %  \draw (3,0)-- node {\midarrow} (6,0);
    %  \draw (6,0)-- node {\midarrow} (6,-3);
\end{scope}

\end{tikzpicture}
\subcaption{$d\ge 3$}
\end{subfigure}
    \caption{Initial configurations on the torus $(\mathbb Z/n\mathbb Z)^d$ whose magnetizations exceed $m_\star(\beta)$, but from which the mixing time within the plus phase (restricted to positive magnetization) is slowed down. Left: In $d=2$, a strip of width $\epsilon n$ of $-1$ spins leads to a mixing time that is at least polynomially slower than from the $\plusone$ initialization. Right: In $d\ge 3$, a strip of width $\epsilon n$ of $-1$ spins on the torus is expected to lead to exponentially slow mixing for the restricted chain, due to the rigidity of the interface separating the $\plusone$ and $\minusone$ regions.}
    \label{fig:restricted-chain-bottlenecks}
\end{figure}
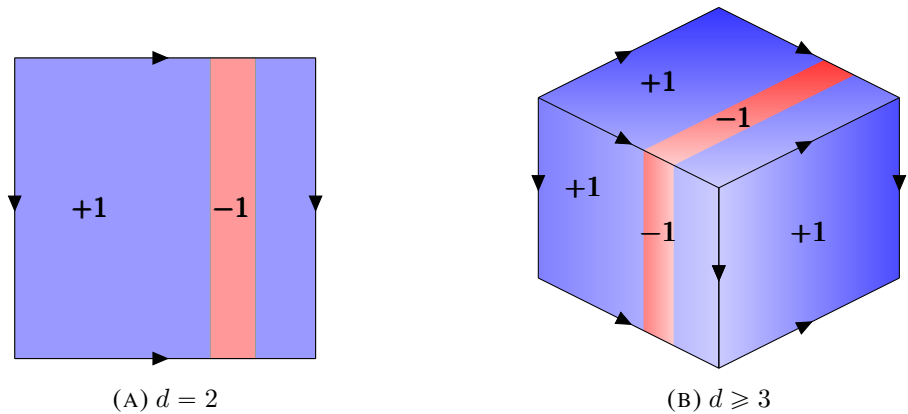

Given the conclusion of Theorem~\ref{thm:Ising-plus-phase} (and its symmetric counterpart for the minus phase), it is straightforward
to prove that WSM within a phase implies the upper bound in Theorem~\ref{thm:Ising-main}; this 
follows because $\pi$ itself is essentially a uniform mixture of its restrictions $\widehat \pi$ and $\widecheck \pi$ to the plus and minus phases, respectively.

To prove Theorem~\ref{thm:Ising-plus-phase}, we apply the WSM within a phase
property to reduce the mixing time on the torus with the $\nu^{\pmb{\pm}}$ initialization to the (worst-case) mixing time of the Ising model {\it with $\plusone$ boundary condition\/} on an exponentially smaller box of side length $m=O(\log n)$.  Plugging in known bounds for this latter mixing
time (specifically, the bound $m^{O(\log m)}$ of~\cite{LMST} for $d=2$ and the trivial
bound $\exp(O(m^{d-1}))$ for $d\ge 3$) then leads immediately to the quasi-linear and quasi-polynomial bounds stated
in Theorem~\ref{thm:Ising-main}.  (A scale reduction in a similar spirit was performed in~\cite{MaTo,LMST}, using arguments specific to dimension $d=2$, to bound the mixing time with $\plusone$ boundary condition initialized from the $\plusone$ configuration.) We emphasize that the (worst-case) mixing time with 
$\plusone$ boundary condition is a famous open question that has been extensively
studied, especially in dimension $d=2$ (see, e.g., \cite{Martinelli-SP,MaTo,LMST}).
It is widely conjectured to be polynomially
bounded for all~$d$, which if established would immediately improve our mixing time
bounds in this paper for the $\nu^{\pmb{\pm}}$ initialization to $O(N{\rm polylog}(N))$.
Indeed, any improvement in the mixing time for the $\plusone$ boundary condition
translates to an improvement in our bounds, and even a modest improvement to $\exp(O(N^{1/d}))$ would suffice to reduce the mixing time for the $\nu^{\pmb{\pm}}$ initialization to ${\rm poly}(N)$. 

The second stage, and technical bulk, of our proof of Theorem~\ref{thm:Ising-main} is to show that WSM within
a phase holds at all low temperatures.\footnote{We will actually prove a slightly more general version of~\eqref{eqn:wsmphaseintro} where $v$ is replaced by a box in $(\mathbb Z/n\mathbb Z)^d$: see Definition~\ref{def:wsm-within-a-phase}.} 

\begin{theorem}\label{thm:Ising-wsm-within-phase-intro}
For every $d\ge 2$ and every $\beta>\beta_c(d)$, the Ising model on $(\mathbb Z/n\mathbb Z)^d$ has WSM within a phase.
\end{theorem}

Our proof of Theorem~\ref{thm:Ising-wsm-within-phase-intro} starts from a recent result of Duminil-Copin, Goswami and Raoufi~\cite{DCGR20}, showing a WSM-type
property for the {\it random-cluster representation\/} of the Ising model at all low temperatures in all dimensions. (The random-cluster
representation is oblivious to the phases of the Ising model.)  In order to lift the random-cluster WSM property to our Ising WSM property within a phase, we combine careful revealing and coupling schemes with a sophisticated coarse-graining approach first introduced 
by Pisztora~\cite{Pisztora96}: coarse-graining replaces single vertices in~$\mathbb Z^d$ with boxes of large but constant size to boost a marginally super-critical model into a very super-critical one.\footnote{The coarse-graining is crucial to proving the result throughout the entire low-temperature regime; if we were only interested in proving it for sufficiently large $\beta$, the exponential tails on connected components of $-1$ spins would simplify the argument.} (See Section~\ref{sec:Ising-wsm-within-phase} for details.)

\subsection{Convergence to equilibrium in a pure phase}
A related question, especially in the statistical physics literature, is that of controlling the rate of convergence to equilibrium \emph{within a pure phase}; this is often formalized in \emph{infinite-volume}, where the Glauber dynamics is defined on all of $\mathbb Z^d$ using continuous-time updates (see Seection~\ref{sec:prelim} for a formal definition). 
It is a straightforward consequence of monotonicity (see~\cite{Martinelli-notes}) that when initialized from the $\plusone$ configuration, this dynamics converges to the (pure phase) infinite-volume plus measure $\pi_{\mathbb Z^d}^{\plus}$ as $t\to \infty$, in the following sense: for any function $F$ that is local (depends only on finitely many spins), its expectation under the dynamics initialized from $\plusone$ (denoted $X_{t,\mathbb Z^d}^{\plus}$) converges to its expectation under $\pi_{\mathbb Z^d}^{\plus}$. In a seminal paper~\cite{FisherHuse}, Fisher and Huse predicted that this rate of relaxation within the plus phase should go as $e^{-\Omega(\sqrt{t})}$. 

As a consequence of our approach, we can obtain the following bounds on the rate of this convergence. A proof is provided at the end of Section~\ref{sec:relaxation-within-phase}.

\begin{corollary}\label{cor:infinite-volume-relaxation}
    For every $d\ge 2$ and every $\beta>\beta_c(d)$, for every local function $F$, there exists $C_F$ such that 
    \begin{align}\label{eqn:infvolcor}
        \big|\mathbb E[F(X_{t,\mathbb Z^d}^{\plus})] - \pi_{\mathbb Z^d}^{\plus} [F(\sigma)]\big|\le \begin{cases}C_F \exp( - e^{\Omega(\sqrt{\log t})}) & d=2  \\ C_F \exp( - \Omega(\log t)^{1/(d-1)}) & d\ge 3 \end{cases}\,,\qquad \mbox{for all $t\ge 0$}\,.
    \end{align}
\end{corollary}
When $d=2$, the bound of~\eqref{eqn:infvolcor} already followed from~\cite[Corollary 4]{LMST}. On the other hand, when $d\ge 3$, to the best of our knowledge, no quantitative bound on this relaxation rate within a pure phase was previously known, except possibly at very low temperatures.

\subsection{Generalization to other geometries}\label{subsec:other-geometries}
Our approach to establishing Theorem~\ref{thm:Ising-main} can be exported to much more general families of graphs without substantial modification. For a more general family of graphs $(G_N)$ on $N$ vertices, our approach consists of two ingredients, as follows:
\begin{enumerate}
    \item establish the ``WSM within a phase" property at low temperatures: in more general settings, this property is defined exactly as in~\eqref{eqn:wsmphaseintro}, except that the restriction $r<n/2$ is replaced 
    by $r< c\cdot\mbox{diam}(G_N)$ for a suitable constant~$c$; and  
    \item establish a ``local" mixing time with $\plusone$ boundary condition, i.e., a uniform mixing time bound $T(r)$ on all balls of radius $r \le c\cdot  \mbox{diam}(G_N)$ with $\plusone$ boundary condition.
\end{enumerate}
Given these ingredients, our approach yields the machinery to deduce
a mixing time of $T(\epsilon_\beta \log N)$ for the chain $X^{\nu^{\pmb{\pm}}}_t$, as well as the restricted chain $\widehat X_t^{\plus}$. We do not include a formal statement of this general implication since every geometry has its own subtleties, e.g., the optimal choice of~$c$ in (1) and~(2) above; the temperatures at which WSM within a phase holds; and the quality of the local mixing time bound in~(2).

The integer lattice discussed above is an example in which the WSM within a phase property could be established throughout the entirety of the low-temperature regime, but in $d\ge 3$ the mixing time estimates with $\plusone$ boundary condition at the local scale were rather crude. 
To illustrate the wider applicability of our framework, we now present results for the Ising Glauber dynamics on random regular graphs; this presents a situation in which the temperature regime for the WSM within a phase property is sub-optimal, only covering sufficiently low temperatures, but the local mixing time bound, and therefore the resulting mixing time bound for the randomly-initialized dynamics, is optimal. 

The Ising model on random $\Delta$-regular graphs has seen significant attention in recent years. In~\cite{DeMo10} it was shown that the critical temperature $\beta_c(\Delta)$, and indeed the free energy of the model on the random $\Delta$-regular graph, aligns exactly with that of the $\Delta$-regular tree. In particular, for all $\beta$ larger than $\beta_c(\Delta)$, the (worst-case) mixing time is $\exp(\Omega(N))$~\cite{CHK19}. The following theorem establishes optimal $O(N\log N)$ mixing at sufficiently large $\beta$ from the random ground state initialization. 

\begin{theorem}\label{thm:random-graph-mixing}
Fix $\Delta\ge 3$ and $\beta$ large. With probability $1-o(1)$, a $\Delta$-regular random graph $\cG$ on $N$ vertices is such that the mixing time of the Ising Glauber dynamics on $\cG$ initialized from $\nu^{\pmb{\pm}}$ is $O(N\log N)$. 
\end{theorem}

The restriction to sufficiently large $\beta$ comes from our proof of WSM within a phase, though we would expect WSM within a phase to hold down to $\beta_c$. The local mixing time with $\plusone$ boundary condition comes by way of comparison with sharp mixing time bounds on the $\Delta$-regular tree with $\plusone$ boundary condition due to~\cite{MSW-trees-bc}. 
We conclude by remarking that the $O(N\log N)$ bound of Theorem~\ref{thm:random-graph-mixing} actually only uses the facts that the  $\Delta$-regular graph $\cG$ is locally tree-like (in a precise sense), and has positive edge expansion. On regular trees, the related result of fast convergence at sufficiently low temperatures with $\plusone$ boundary condition, when initialized from the $\plusone$ configuration, was already shown in~\cite{CaMaTree}. We also mention the work of~\cite{Bianchi}, which extends~\cite{MSW-trees-bc} to a broader class of non-amenable graphs, and could serve as the necessary ingredient for step (2) above on a larger family of expander graphs.

\subsection{Possible extensions} \label{subsec:extensions}
We briefly discuss some potential extensions of our results above, leaving them as open problems. The most immediate such extension is of course improving the quasi-polynomial bounds of Theorem~\ref{thm:Ising-main} in $d\ge 3$ to polynomial, or even $N^{1+o(1)}$ bounds. As discussed earlier, a possible, though challenging route towards this is to improve the existing worst-case mixing time bounds with $\plusone$ boundary condition, which we use as inputs at local scales: improving this even to $\exp(O(N^{1/d}))$ in dimension $d\ge 3$ would already lead to a polynomial mixing time bound for the randomly initialized dynamics on the torus.

Going beyond the Ising model, it is natural to ask the analogous question for the Potts model, a canonical generalization of the Ising model in which each vertex can be assigned one of $q$ possible spins, with the same Gibbs distribution as in~\eqref{eqn:gibbs} except that now $C(\sigma)$ denotes
the {\it $q$-way cut\/} induced by the spins; the Ising model is the case $q=2$. The natural analogue of Theorem~\ref{thm:Ising-main}
would suggest sub-exponential mixing time for the low-temperature Potts Glauber dynamics, starting from the initialization that puts equal
weight~$\frac{1}{q}$ on each of the ground states, i.e., the configurations in which all spins are the same.  
A major obstacle
to extending our current analysis to this scenario is the lack of monotonicity in the Potts model for~$q\ge 3$.

A closely related model that \emph{does} exhibit monotonicity is the {\it random-cluster model}, an extensively studied graphical representation of the Ising and Potts models
in which configurations are subgraphs weighted by their number of edges and number of connected components~\cite{Grimmett}. 
The Glauber dynamics for the random-cluster model (which updates edges rather than vertex spins) on $(\mathbb Z/n\mathbb Z)^d$ does not suffer from the exponential slowdown at low temperatures of the Ising and Potts models, but is expected
to experience such a slowdown {\it at the critical point} if $q$ is larger than some $q_c(d)$, due to a bottleneck between an ``ordered phase" (marked by the existence of a giant component) and a ``disordered phase" (where the subgraph is shattered). This slowdown was proved for~$q$ sufficiently large in~\cite{BCT,BCFKTVV}, and for all $q>q_c(2) = 4$ when $d=2$~\cite{GL1}.  We predict that this bottleneck can be overcome if the dynamics
is initialized from a mixture of the corresponding ground states, i.e., the full and empty subgraphs. (In this case, interestingly, the right mixture is not
uniform, but has weights determined by the parameters $q$ and~$d$.)
Indeed, our proof approach can be seen to generalize immediately to this setting, so the main missing piece is establishing that the appropriate analog of WSM within a phase holds.

\subsection{Related work} \label{subsec:related}
There is a vast literature on Glauber dynamics for spin systems which we do not attempt to summarize here.
We focus only on attempts to circumvent low-temperature bottlenecks---our main concern in this paper.
As mentioned earlier, the only setting in which 
sub-exponential mixing of the low-temperature Glauber dynamics from special initializations
was previously known 
is the mean-field Ising model~\cite{LLP} where 
the question reduces to a 1-dimensional argument based just on the magnetization.  
Other approaches to this problem have typically started by transforming the spin system to an alternative
representation, which does not suffer from the same bottleneck. We summarize these here, though we emphasize that our motivation in this paper is to
explore the more direct approach of using the original Glauber dynamics on the spin system itself.  

The early work of~\cite{JSIsing}, which gave a polynomial-time sampler for the Ising model on all graphs at all temperatures, was based on the so-called ``even subgraph" representation of the model. 

The Glauber dynamics for the random-cluster model has been seen as a promising route to overcoming low-temperature bottlenecks in both the Ising and Potts models, as it is oblivious to the choice of phase. For the $q=2$ case of the Ising model,~\cite{GuoJer} proved that the random-cluster dynamics mixes in polynomial time on all graphs at all temperatures. For general $q$,~\cite{BS} showed that the random-cluster dynamics on subsets of $\mathbb Z^2$ mixes in $O(N\log N)$ time at all low temperatures (see also~\cite{BGVfull} where different boundary conditions were considered). More generally, the classical proof of~\cite{MaOl1} (shown to extend to random-cluster dynamics
in~\cite{HarelSpinka}) implies fast mixing of the random-cluster dynamics on $(\mathbb Z/n\mathbb Z)^d$ whenever the random-cluster model has a WSM property; recall that this property was shown to hold at all low temperatures in all dimensions when $q=2$ in~\cite{DCGR20}.

The Swendsen-Wang dynamics~\cite{SW} exploits the Edwards-Sokal coupling~\cite{ES} of spins and edges in the Potts
and random-cluster models to update large clusters of spins in a single step, and thus jump between phases.
By virtue of a comparison technique introduced in~\cite{Ullrich-random-cluster}, the above bounds on the random-cluster
dynamics translate to bounds on the Swendsen--Wang dynamics up to a multiplicative factor of~$N$. In special settings, one can do better than this lossy comparison and obtain optimal bounds: at low temperatures this includes the complete graph~\cite{LNNP14,GSV,BS-MF},  $\mathbb Z^2$~\cite{Martinelli-SW,BCPSV} and trees~\cite{BZSV-SW-trees}.  

Finally, there has been a recent series of papers based on the polymer representation of the Ising and Potts models 
(also known as the cluster expansion).  This began with the work ~\cite{HPR-Algorithmic-Pirogov-Sinai}, which gave polynomial-time sampling algorithms for the Potts model on $(\mathbb Z/n\mathbb Z)^d$ at
sufficiently low temperatures.  Further work~\cite{BCHPT-Potts-all-temp} derives similar results at all
temperatures, provided the parameter~$q$ is sufficiently large as a function of the dimension~$d$.
In a related direction, ~\cite{CGGPS,GGS} obtain rapid mixing results for
the so-called ``polymer dynamics'' at sufficiently low temperatures on expander graphs, and deduce polynomial mixing for Glauber dynamics restricted to a small
region around the ground states (measured in terms of the size of its largest polymer).  It should be noted that such approaches are ultimately based on 
convergence of the cluster expansion and are unlikely to cover the full low-temperature regime unless~$q$ is taken sufficiently large.

\par\medskip\noindent
{\bf Outline of paper.}
In Section~\ref{sec:prelim}, we recall necessary preliminaries on the Ising model and its Glauber dynamics. In Section~\ref{sec:relaxation-within-phase}, we prove the mixing time upper bounds of Theorems~\ref{thm:Ising-main} and~\ref{thm:Ising-plus-phase} assuming WSM within a phase holds. Then, in Section~\ref{sec:Ising-wsm-within-phase}, we prove Theorem~\ref{thm:Ising-wsm-within-phase-intro}, establishing that WSM within a phase does indeed hold at all low temperatures. In Section~\ref{sec:random-graphs}, we study the same question on the random regular graph and prove Theorem~\ref{thm:random-graph-mixing}.

\subsection*{Acknowledgements}
The authors thank the anonymous referees for their careful reading of the manuscript. The authors are grateful to Allan Sly for pointing out an error in the proof of an earlier version of Theorem~\ref{thm:Ising-main}, and to Etien Lo\"ic Fari\~na L\"uth for pointing out several typos in an earlier version of the paper. The authors also thank Fabio Martinelli and Allan Sly for useful discussions. R.G.\ thanks the Miller Institute for Basic Research in Science for its support. The research of A.S.\ is supported in part by NSF grant CCF-1815328.

\section{Preliminaries and notation}\label{sec:prelim}
In this section, we recap our notation, and important preliminaries about the Ising model and monotone Markov chains that we appeal to throughout the paper. 

\subsection{Underlying geometry}
For most of the paper (excluding Section~\ref{sec:random-graphs}), we will work on rectangular subgraphs of the $d$-dimensional lattice $\mathbb Z^d$,
with vertex set 
\[
    \Lambda_m := [-m,m]^d \cap \mathbb Z^d\,,
\]
and edge set $E(\Lambda_m) = \{\{u,v\}:d(u,v)=1\}$, where $d(\cdot,\cdot)$ is the $\ell_1$ distance in $\mathbb Z^d$. 
The (inner) boundary vertices of $\Lambda_m$ are denoted $\partial \Lambda_m= \{w\in \Lambda_m: d(w,\mathbb Z^d \setminus \Lambda_m) = 1\}$. 

For a vertex $v$ in $\mathbb Z^d$, we will frequently consider its ($\ell_\infty$) ball of radius $r$, which we denote $B_r(v)=\{w\in \mathbb Z^d: d_\infty(w,v)\le r\}$
where $d_\infty(\cdot,\cdot)$ is the $\ell_\infty$ distance in $\mathbb Z^d$. 

We also consider the torus $\mathbb T_m = (\mathbb Z/(2m\mathbb Z))^d$ with 
nearest neighbor edges, where the $\ell_1$ distance is measured modulo $2m$. One naturally identifies the fundamental domain of $\mathbb T_m$ with $\Lambda_m$, but with vertices on opposite sides identified with one another as one vertex---we denote this graph $\Lambda_m^p$ indicating its equivalence to $\Lambda_m$ with periodic boundary condition. 

\subsection{The Ising model}
Recall the definition of the Ising Gibbs distribution~\eqref{eqn:gibbs} on a graph $G$ at inverse temperature $\beta>0$. We denote the set of configurations of the model by  $\Omega = \{\pm 1\}^{V(G)}$. We indicate the underlying graph by a subscript, e.g., $\pi_{G}$. (The parameter $\beta$ will always be fixed, and thus its dependency is suppressed.) We use $\pi_G[\cdot]$ to denote the corresponding expectation. 

For an Ising configuration $\sigma\in \Omega$, let $M(\sigma) = \sum_{v\in V(G)} \sigma_v$ be its {\it magnetization}. We will frequently work with the following subsets of the configuration space, as indicated in the introduction: 
\[    \widehat \Omega= \{\sigma: M(\sigma) \ge 0\}\,, \qquad \mbox{and}\qquad \widecheck \Omega = \{\sigma: M(\sigma) \le 0\}\,.
\]We also use $\widehat \pi_G$ and $\widecheck\pi_G$ to denote the conditional distributions $\pi_G (\cdot \mid \widehat \Omega)$ and $\pi_G (\cdot \mid \widecheck \Omega)$ respectively. 

\subsubsection*{Ising phase transition on $\mathbb Z^d$}
The Ising model has a famous phase transition on $\mathbb Z^d$, manifested in finite volumes as follows. For every $d\ge 2$, there exists a critical $\beta_c(d)>0$ such that (1) for all $\beta<\beta_c(d)$, the spin-spin correlation $\pi_{\mathbb T_m}(\sigma_v = \sigma_w) - \frac 12$ decays to zero exponentially with $d(v,w)$; and (2) for all $\beta>\beta_c(d)$, $\pi_{\mathbb T_m}(\sigma_v = \sigma_w) - \frac 12$ is bounded away from zero, uniformly in $m$ and $v,w\in \mathbb T_m$.  

As a result, for $\beta>\beta_c(d)$ there is no spatial mixing of the form~\eqref{eqn:wsmintro},
and $(2m)^{-d} M(\sigma)$ converges to a two-atomic distribution at $\pm m_*(\beta)$ for some $m_*(\beta)>0$. 
The works~\cite{Pisztora96,Bodineau05} established a \emph{surface-order} large deviation principle: for all $\beta>\beta_c(d)$, for all $\epsilon<m_*(\beta)$, 
\begin{align}\label{eq:surface-order-LDP}
    \pi_{\mathbb T_m}\big((2m)^{-d}M(\sigma) \in [-\epsilon,\epsilon]\big) \le Ce^{ - m^{d-1}/C}\,.
\end{align}

\subsubsection{Boundary conditions}
When considering the Ising model on subgraphs of $\mathbb Z^d$, e.g., $\Lambda_m$, we will consider the model with \emph{boundary conditions}. A boundary condition is an arbitrary configuration $\eta\in \{\pm 1\}^{\mathbb Z^d}$. The Ising measure on $\Lambda_m$ with boundary condition~$\eta$ is the  conditional distribution 
$$\pi_{\Lambda_{m+1}}\big(\sigma_{\Lambda_m}\in \cdot \mid \sigma_{\partial \Lambda_{m+1}} = \eta_{\partial \Lambda_{m+1}}\big)\,,$$
where we write $\sigma_A$ for a configuration~$\sigma$ restricted to vertices in a subset~$A$.
We use $\pi_{\Lambda_m^\eta}$ to denote this conditional distribution over $\{\pm 1\}^{\Lambda_m}$, and use $\plus$ or $\plusone$ to denote the all $+1$ configuration, and $\minus$ or $\minusone$ to denote the all $-1$ configuration.

\subsection{Markov chains and monotonicity relations}
Consider a (discrete-time) ergodic Markov chain with transition matrix $P$ on a finite state space $\Omega$, reversible with respect to a stationary distribution~$\pi$. Denote the chain initialized from $x_0\in \Omega$ by $(X_t^{x_0})_{t \in \mathbb N}$. 
It is well known that, for any~$x_0$, the distance $d_{\tv}(x_0; t) := \|\mathbb P(X_{t}^{x_0}\in \cdot) -\pi\|_\tv$ is non-increasing with~$t$ and converges to zero as $t\to\infty$.  (Here
$\Vert\cdot\Vert_\tv$ is the total variation distance (i.e., half the $\ell_1$ norm).)
The {\it $\epsilon$-mixing time from initialization~$x_0$} is defined as $\tmix^{x_0}(\epsilon) = \min\{t: d_\tv(x_0;t) \le \epsilon\}$, and by convention, we set $\tmix^{x_0} = \tmix^{x_0}(1/4)$ and $\tmix = \max_{x_0}\tmix^{x_0}$. 

\subsubsection*{Continuous-time dynamics}
It will be convenient in our analysis to work in continuous rather than discrete time.  The
{\it continuous-time Glauber dynamics\/} on a graph~$G$ is defined as follows.
Assign every vertex an i.i.d.\ rate-1 Poisson clock; when the clock at $v$ rings, replace the spin $\sigma_v$ by a random spin sampled according to the correct conditional distribution given the spins on the neighbors of~$v$.  Recall (e.g., from~\cite[Theorem~20.3]{LP}) the standard fact that \begin{align}\label{eq:continuous-time-discrete-time-comparison}
   d_{\tv}^{\textrm{cont}}(x_0 ; C t) \le d_\tv(x_0; |V(G)| t)  \le  d_{\tv}^{\textrm{cont}}(x_0 ; C^{-1} t)\,,
\end{align}
for some absolute constant $C$, where $d_\tv^{\textrm{cont}}(x_0;t)$ is the total variation distance of the continuous-time dynamics at time $t$ initialized from $x_0$; i.e., the mixing times of the discrete- and continuous-time
dynamics differ only by a factor~$O(|V(G)|)$, where $|V(G)|$ is the number of vertices in $G$.  As such, 
for the $N^{1+o(1)}$ and $O(N\log N)$ mixing time results stated in the introduction, it suffices
to prove that the continuous-time dynamics has mixing time $N^{o(1)}$ and $O(\log N)$ respectively.
Abusing notation, from this point forth, we let $(X_t^{x_0})_{t\ge 0}$, and all other Ising Glauber dynamics we consider, be the continuous-time chains.

\subsubsection*{The grand coupling}
A standard tool in the study of Markov chains is the \emph{grand coupling}, which places chains with all possible initializations in a common probability space. 

\begin{definition}\label{def:grand-coupling}
Independently to each vertex~$v\in V(G)$ we assign an infinite sequence of times $(T_1^{v}, T_2^{v},...)$ given by the rings of a rate-1 Poisson clock, and an infinite sequence
$(U_{1}^{v},U_2^v,...)$ of $\mbox{Unif}[0,1]$ random variables. The dynamics $(X_{t}^{x_0})_{t\ge 0}$ is generated as follows (where $\pi$ denotes the stationary distribution):  
\begin{enumerate}
\item Set $X_0^{x_0} =x_0$.  Let $t_1< t_2<...$ be the times in $\bigcup_{k}\bigcup_{v}\{T_k^v\}$ in increasing order (almost surely, these times are distinct).
\item For $j\ge 1$, let $(v,k)$ be the unique pair such that $t_j=T_k^v$.
Set $X_t^{x_0}= X_{t_{j-1}}^{x_0}$ for all $t\in [t_{j-1},t_j)$; then set 
$X_{t_j}^{x_0}(\Lambda_n\setminus \{v\}) = X_{t_{j-1}}^{x_0}(\Lambda_n\setminus \{v\})$ and
    \begin{align}\label{eq:update-rule}
    X_{t_j}^{x_0}(v) = \begin{cases}+1 & \mbox{if } U_{k}^{v}\le \pi(\sigma_v = +1 \mid \sigma_{\Lambda_n\setminus v} = X_{t_{j-1}}^{x_0}(\Lambda_n \setminus v))  \\ 
    -1 &  \mbox{otherwise}
    \end{cases}\,.
    \end{align}
\end{enumerate}
\end{definition}

By using the same times $(T_1,T_2,...)_{v\in \mathbb Z^d}$ and uniform random variables $(U_1^v,U_2^v,...)_{v\in \mathbb Z^d}$, and taking $\pi = \pi_{A^\eta}$, the above construction gives a grand coupling of $(X_{A^\eta,t}^{x_0})_{t\ge 0}$, the Markov chains on $A\subset \mathbb Z^d$ with boundary condition $\eta$ initialized from all possible configurations~$x_0$. It is well known (and simple to check) that in the case of
the Ising model this coupling is
{\it monotone}, in the sense that if $x_0 \ge x_0'$ and $\eta \ge \eta'$, then $X_{A^\eta,t}^{x_0}\ge X_{A^{\eta'},t}^{x_0'}$ for all $t\ge 0$. (Monotonicity
does not hold for many other models, such as the Potts model.)

\subsubsection*{The restricted Glauber dynamics}
A crucial tool in our analysis will be the Glauber dynamics 
\emph{restricted to $\widehat \Omega$} (and, symmetrically, to~$\widecheck \Omega$).
This chain, denoted $(\widehat X_t^{x_0})_{t\ge 0}$, is defined exactly as in Definition~\ref{def:grand-coupling}, except that 
if the update in~\eqref{eq:update-rule} would cause~$\sigma$ not to be 
in~$\widehat \Omega$ we set $\widehat X_{t_j}^{x_0}(v) = +1$ deterministically (i.e.,
leave $\sigma_v$ unchanged).
It is easy to check that $(\widehat X_t)$ is a Markov chain reversible w.r.t.\ $\widehat \pi$.

\subsection{Notational disclaimers}
Throughout the paper, $d\ge 2$ and $\beta>\beta_c(d)$ will be fixed and understood from context. All our results should be understood as holding uniformly over sufficiently large $n$ or $m$. The letter $C>0$ will be used frequently, always indicating the existence of a constant that is independent of $r,n,m$ etc., but that may depend on $\beta, d$ and may differ from line to line. Finally, we use big-$O$, little-$o$ and big-$\Omega$ notation in the same manner, where the hidden constants may depend on $\beta, d$. 

\section{Mixing from random initializations given WSM within a phase}\label{sec:relaxation-within-phase}
In this section, we reduce the mixing time on the torus $\Lambda_n^p$ 
from $\nu^{\pmb{\pm}}$ to the mixing time at local $O(\log n)$ scales with $\plusone$ boundary condition, under the assumption of WSM within a phase. 
Let us formalize the following minimal notion of mixing at local scales, which only asks for a bound on the rate of exponential relaxation to equilibrium given a $\plusone$ initialization;  namely, suppose there exists a non-decreasing sequence $f(m)$ such that for all $m<n$ and all $t>0$, 
\begin{align}\label{eq:assumption-scale-m}
    \max_{v\in \Lambda_n}\|\mathbb P\big(X_{B_m^{\plus}(v),t}^{\plus}(v) \in \cdot)  - \pi_{B_m^{\plus}(v)}(\sigma_v \in \cdot)\|_\tv \le e^{ - t/f(m)}\,.
\end{align}
Of course by transitivity of the torus, the left-hand side is independent of $v$, so we may drop the maximum and just consider a box of side-length $m$ centered at some fixed $v\in \Lambda_n$ with $\plusone$ boundary condition.

A (worst-case) mixing time (or inverse spectral gap) bound of $f(m)$ for the Glauber dynamics on $\Lambda_m^{\plus}$ would automatically yield~\eqref{eq:assumption-scale-m}, but we use the above formulation in case one can obtain better bounds on the exponential rate of relaxation when initialized from the $\plusone$ configuration. 

Given such a sequence $f(m)$, fix a constant $K$ sufficiently large, and define the following function: 
\begin{align}\label{eq:g(t)}
    g_n(t) = \max\Big\{m \le n : mf(m) \le t\wedge  e^{ n^{d-1}/K}\Big\}\,.
\end{align}

We aim to prove the following general theorem, from which Theorem~\ref{thm:Ising-main} will follow (see Section~\ref{subsec:Ising-main-proof}). 

\begin{theorem}\label{thm:Ising-torus-restatement}
Suppose that WSM within a phase holds and suppose further that~\eqref{eq:assumption-scale-m} holds for a non-decreasing sequence $f(m)$. Let $g_n(t)$ be as in~\eqref{eq:g(t)} for $K$ large enough. Then for all $t\ge 0$,  
\begin{align*}
    \|\mathbb P\big(X_{\Lambda_n^p,t}^{\nu^{\pmb{\pm}}}\in \cdot) - \pi_{\Lambda_n^p}\|_\tv \le Cn^d e^{ - g_n(t)/C}\,.
\end{align*}
\end{theorem}

Let us pause to interpret the above theorem, and what the quantity $g_n(t)$ will typically look like.  In order for the dynamics on the torus $\Lambda_n^p$ from the random $\nu^{\pmb{\pm}}$ initialization to be mixed, we need $Cn^d e^{ - g_n(t)/C}$ to be $o(1)$, for which in turn we need $g_n(t)$ to be $C'\log n$ for some $C'$ large. Now notice that if $f(m)$ is at least polynomial in $m$ (as it will be in any bounds we apply), then $g_n(t)$ is roughly the same as $f^{-1}(t)$.
Setting $f^{-1}(t) \approx C'\log n$, we see that $t$ needs to be approximately $f(C'\log n)$. Thus, under the assumption of WSM within a phase, the mixing time on a torus of side length~$n$ from the random initialization~$\nu^{\pmb{\pm}}$ is reduced to the mixing time on a box with $\plusone$ boundary of side length $C'\log n$. 

The key step in the proof of Theorem~\ref{thm:Ising-torus-restatement} is a
bound on the rate of convergence of the restricted chain $\widehat X_t^{\plus}$ defined in Section~\ref{sec:prelim} to the plus phase measure $\widehat \pi_{\Lambda_n^p}$. We expect this to be of independent interest. 

\begin{theorem}\label{thm:Ising-mixing-restricted-chain}
Suppose that WSM within a phase holds and suppose further that~\eqref{eq:assumption-scale-m} holds for some non-decreasing sequence $f(m)$. Let $g_n(t)$ be as in~\eqref{eq:g(t)} for $K$ large enough. Then for all $t\ge 0$,
\begin{align*}
    \|\mathbb P(\widehat X_{\Lambda_n^p,t}^{\plus}\in \cdot)-\widehat \pi_{\Lambda_n^p}\|_\tv \le Cn^d e^{ - g_n(t)/C}\,.
\end{align*}
\end{theorem}

For ease of notation, we will drop $\Lambda_n^p$ subscripts, writing $\widehat \pi = \widehat \pi_{\Lambda_n^p}$, $X_{t}^\sigma = X_{\Lambda_n^p,t}^\sigma$, and $\widehat X_t^\sigma = \widehat X_{\Lambda_n^p,t}^{\sigma}$. 

\subsection{Single-site relaxation within a phase}
The goal of this subsection is to prove the following rate of relaxation of the single-site marginals of $X_t^{\plus}$ to $\widehat \pi$. This will be the main ingredient in proving Theorem~\ref{thm:Ising-mixing-restricted-chain}. 

\begin{proposition}\label{prop:single-site-relaxation}
Suppose WSM within a phase holds, and suppose~\eqref{eq:assumption-scale-m} holds for some non-decreasing sequence $f(m)$. Let $g_n(t)$ be as in~\eqref{eq:g(t)} for $K$ large enough. Then for every $v\in \Lambda_n$, for all $t\le e^{n^{d-1}/K}$, 
\begin{align*}
    |\mathbb P(\widehat X_{t}^{\plus}(v) = +1) - \widehat \pi(\sigma_v = +1)| \le Ce^{ - g_n(t)/C}\,.
\end{align*}
\end{proposition}

We first overcome the absence of monotonicity of the restricted chain using the observation that its typical hitting time to $\partial \widehat \Omega$ is exponentially long. Throughout the paper, let us denote by $\widehat \tau^{x_0}$ the hitting time to $\partial \widehat \Omega$ of the restricted dynamics $\widehat X_{t}^{x_0}$. We use $\widehat \tau^{\widehat \pi}$ to denote the random variable given by $\widehat \tau^{x_0}$ where $x_0$ is drawn from $\widehat \pi$. By Definition~\ref{def:grand-coupling} and monotonicity, together with the definition of the restricted dynamics, one observes the following.

\begin{claim}\label{clm:monotonicity-relations}
For every $0\le t\le \widehat \tau^{\widehat \pi}$, we have  $\widehat X_{t}^{\widehat \pi} = X_{t}^{\widehat \pi} \le  X_{t}^{\plus} = \widehat X_{t}^{\plus}$ under the grand coupling. 
\end{claim}

Claim~\ref{clm:monotonicity-relations} leads to the following lower bound on $\widehat \tau^{\widehat \pi}\le \widehat \tau^{\plus}$. 

\begin{lemma}\label{lem:zero-magnetization-hitting-time}
    For every $t\ge 0$, we have $\mathbb P(\widehat \tau^{\plus}\le t) \le \mathbb P(\widehat \tau^{\widehat \pi}\le t) \le C (t\vee 1) e^{ - n^{d-1}/C}$. 
\end{lemma}

\begin{proof}
Fix the sequence of times $t_1,t_2,...$ at which some clock rings in $\Lambda_n$: note that this is distributed as a Poisson clock with rate $|\Lambda_n|$. By definition of the Glauber dynamics, even conditionally on this clock sequence, for all $t$, the law of $\widehat X_{t}^{\widehat \pi}$ is stationary, and thus distributed as $\widehat \pi$. The number of clock rings by time $t$ is at most $n^{d-1}|\Lambda_n|(t\vee 1)$ except with probability $Ce^{- n^{d-1}/C}$. Therefore, we have 
$$\mathbb P(\widehat \tau^{\widehat \pi} \le t) \le Ce^{ - n^{d-1}/C} + \sum_{i\le n^{d-1}|\Lambda_n|(t\vee 1)} \mathbb P(\widehat X_{t_i}^{\widehat \pi}\in  \partial \widehat \Omega) \le Ce^{ - n^{d-1}/C} + Cn^{2d-1}(t\vee 1) \widehat \pi(\partial \widehat \Omega)\,,$$
which implies the desired bound via~\eqref{eq:surface-order-LDP} and $\widehat \pi(\partial \widehat \Omega)\le 2 \pi(|M(\sigma)|\le 1)$. 
\end{proof}

We are now in a position to provide a proof of Proposition~\ref{prop:single-site-relaxation}. 

\begin{proof}[\textbf{\emph{Proof of Proposition~\ref{prop:single-site-relaxation}}}]
Fix $t$ and $v\in \Lambda_n$ and consider the quantity on the left-hand side of the proposition. For ease of notation, define 
\begin{align}\label{eq:P-t-f-v}
    \widehat P_t f_v(\plus) : = \mathbb P(\widehat X_t^{\plus}(v) = +1) - \widehat \pi(\sigma_v = +1),
\end{align}
so that the quantity we wish to control is $|P_t f_v(\plus)|$. 
We first derive a simple lower bound on $P_t f_v(\plus)$, namely under the grand coupling,
\begin{align*}
    \widehat P_t f_v(\plus) = \mathbb E[\mathbf 1\{\widehat X_t^{\plus}(v) = +1\} - \mathbf  1\{\widehat X_t^{\widehat \pi}(v) = +1\}] \ge - \mathbb P(\widehat \tau^{\widehat \pi}\le t)\,,
\end{align*}
since by Claim~\ref{clm:monotonicity-relations} the difference of the two indicators is non-negative on $\{\widehat \tau^{\widehat \pi}>t\}$. Thus, by Lemma~\ref{lem:zero-magnetization-hitting-time}, while $t \le e^{ n^{d-1}/K}$ for a sufficiently large $K$, we get $\widehat P_t f_v(\plus) \ge - Ce^{ - n^{d-1}/C}$. Since $g_n(t)\le n$, we also get $\widehat P_t f_v(\plus) \ge -Ce^{ - g_n(t)/C}$. 

We now derive an upper bound on $\widehat P_t f_v(\plus)$. For any $m = m(t) \le n$, by monotonicity of the Glauber dynamics, 
\begin{align*}
    \widehat P_t f_v(\plus) \le  \mathbb P(\widehat X_{t}^{\plus}(v) = +1) - \mathbb P(X_{t}^{\plus}(v)= +1)  + \mathbb P(X_{B_m^{\plus}(v),t}^{\plus}(v) = +1) - \widehat \pi(\sigma_v = +1)\,.
\end{align*}
Here recall that $X_{B_m^{\plus}(v),t}^{\plus}$ is the Glauber dynamics on $B_m^{\plus}(v)$ initialized from the $\plusone$ configuration. 
The difference of the first two terms on the right is at most $\mathbb P(\widehat \tau^{\widehat \pi}\le t)$ in absolute value by Claim~\ref{clm:monotonicity-relations}. Thus, by a triangle inequality, we obtain
\begin{align*}
   \widehat P_t f_v(\plus) \le \mathbb P(\widehat \tau^{\widehat \pi}\le t) & + |\pi_{B_m^{\plus}(v)}(\sigma_v = +1) - \widehat \pi(\sigma_v = +1)| \\
   &   + |\mathbb P(X_{B_m^{\plus}(v),t}^{\plus}(v) = +1)  - \pi_{B_m^{\plus}(v),t}(\sigma_v = +1)|\,.
    \end{align*}
The first term is at most $Ce^{ - n^{d-1}/C}$ by Lemma~\ref{lem:zero-magnetization-hitting-time} while $t\le e^{ n^{d-1}/K}$. The second term is exactly the quantity bounded by WSM within a phase, so that by~\eqref{eqn:wsmphaseintro} it is at most $Ce^{ - m/C}$. The third term is exactly the quantity controlled by the assumption in~\eqref{eq:assumption-scale-m}. Combining these three estimates, we find that while $t\le e^{ n^{d-1}/K}$, 
\begin{align*}
   \widehat P_t f_v(\plus) \le Ce^{ - n^{d-1}/C} + Ce^{ - m/C} + Ce^{-t/f(m)} \,.
\end{align*}
Now, taking $m = g_n(t)$, we see that $m\le n$ so that the first term is at most $Ce^{ - m/C}$ and $mf(m) \le t$, so that the third term is also at most $Ce^{ - m}$. Together with the earlier lower bound on $\widehat P_t f_v(\plus)$, we obtain the desired inequality. 
\end{proof}

\subsection{Proofs of Theorems~\ref{thm:Ising-torus-restatement} and \ref{thm:Ising-mixing-restricted-chain}}\label{subsec:fast-mixing-torus}
To prove these theorems
from Proposition~\ref{prop:single-site-relaxation}, it remains to reduce the total variation distance between $\widehat X_t^{\plus}$ and $X_{t}^{\nu^{{\pmb{\pm}}}}$ and their respective stationary distributions to sums of the distances of their one-point marginals. Such a bound is not true in general, but 
can be obtained in our setting using the monotonicity over sub-exponential timescales implied by Claim~\ref{clm:monotonicity-relations} and Lemma~\ref{lem:zero-magnetization-hitting-time}.

\begin{proof}[\textbf{\emph{Proof of Theorem~\ref{thm:Ising-mixing-restricted-chain}}}]
Consider the total variation distance 
\begin{align*}
    \|\mathbb P(\widehat X_{t}^{\plus}\in \cdot) - \widehat \pi\|_\tv = \|\mathbb P(\widehat X_{t}^{\plus}\in \cdot) - \mathbb P(\widehat X_{t}^{\widehat \pi}\in\cdot)\|_\tv\,.
\end{align*}
Using the standard definition of variation distance as the probability of disagreement under an \emph{optimal} coupling, we can bound the above by the following probability of disagreement under the grand coupling: 
\begin{align*}
    \mathbb P\big(\widehat X_{t}^{\plus} \ne \widehat X_{t}^{\widehat \pi}\big) & \le \mathbb E\big[ \mathbf 1\{\widehat X_{t}^{\plus} \ne \widehat X_{t}^{\widehat \pi}\} \mathbf 1\{\widehat \tau^{\widehat \pi} > t\}\big] + \mathbb P(\widehat \tau^{\widehat \pi} \le t) \le \mathbb P\big(X_{t}^{\plus} \ne X_{t}^{\widehat \pi}\big) + \mathbb P(\widehat \tau^{\widehat \pi} \le t)\,,
\end{align*}
where in the second inequality we used Claim~\ref{clm:monotonicity-relations}. By Lemma~\ref{lem:zero-magnetization-hitting-time}, the second term on the right is at most $Ce^{ - n^{d-1}/C}$ while $t \le e^{n^{d-1}/K}$.
By monotonicity of the grand coupling, the first term is at most  
\begin{align*}
    \sum_{v\in \Lambda_n^p} \Big(\mathbb P(X_{t}^{\plus}(v) & = +1)  - \mathbb P(X_{t}^{\widehat \pi}(v) = +1)\Big) \\
     & \le \sum_{v\in \Lambda_n^p} \Big(\mathbb P(X_t^{\plus}(v) = +1, \widehat \tau^{\widehat \pi} > t) - (1 - \mathbb P(X_{t}^{\widehat \pi}(v) = -1, \widehat \tau^{\widehat \pi} > t)) + 2\mathbb P(\widehat \tau^{\widehat \pi} \le t)\Big) \\
    & \le  \sum_{v\in \Lambda_n^p} \Big(\mathbb P(\widehat X_{t}^{\plus}(v) = +1) - \widehat \pi(\sigma_v =+1)  + 2\mathbb P(\widehat \tau^{\widehat \pi} \le t) \Big),
\end{align*}
where in the second inequality we again used Claim~\ref{clm:monotonicity-relations} to switch from $X_{t}^{\plus}$ to $\widehat X_t^{\plus}$ and from $X_{t}^{\widehat\pi}$ to the stationary chain $\widehat X_{t}^{\widehat \pi}$. 
Combining the above, and using Lemma~\ref{lem:zero-magnetization-hitting-time} again,
we find that for all $t\le e^{n^{d-1}/K}$,
\begin{align*}
    \mathbb P\big(\widehat X_{t}^{\plus} \ne \widehat X_{t}^{\widehat \pi}\big) \le \sum_{v\in \Lambda_n^p}  \widehat P_t f_v(\plus) + Cn^d e^{ - n^{d-1}/C}\,,
\end{align*}
where we recall the definition of $\widehat P_t f_v$ from~\eqref{eq:P-t-f-v}. 
By 
Proposition~\ref{prop:single-site-relaxation}, the first term above is at most $Cn^d e^{ - g_n(t)/C}$ for all $t\le e^{ n^{d-1}/K}$; the second term can be absorbed into this term as $g_n(t)\le n\le n^{d-1}$. The constraint on $t$ can be dropped by the fact that total variation distance of a Markov chain to its stationary measure is non-increasing, and the fact that $g_n(t) = g_n(e^{n^{d-1}/K})$ for all $t\ge e^{n^{d-1}/K}$.
\end{proof}

\noindent We now conclude the mixing time upper bound on $\Lambda_n^p$ initialized from $\nu^{\pmb{\pm}}$, given WSM within a phase. 

\begin{proof}[\textbf{\emph{Proof of Theorem~\ref{thm:Ising-torus-restatement}}}]
Consider the total variation distance 
\begin{align*}
    \|\mathbb P(X_{t}^{\nu^{{\pmb{\pm}}}} \in \cdot) - \pi\|_\tv & \le \big\|\frac 12 \mathbb P(X_{t}^{\plus}\in \cdot ) + \frac 12 \mathbb P(X_{t}^{\minus}\in \cdot) -\frac 12 \widehat \pi - \frac 12 \widecheck\pi \big\|_\tv  + \big\|\pi - \frac 12\widehat \pi  - \frac 12 \widecheck \pi\big\|_\tv  \\ 
    & \le \frac 12 \Big[ \|\mathbb P(X_{t}^{\plus}\in \cdot )- \widehat \pi\|_\tv + \|\mathbb P(X_{t}^{\minus}\in \cdot ) - \widecheck \pi\|_\tv \Big] + \pi(\partial \widehat \Omega) + \pi(\partial \widecheck \Omega)\,.
\end{align*}
The second inequality here used the triangle inequality for the first term, and the spin-flip symmetry of the Ising model, together with the definitions of $\widehat \Omega$ and $\widecheck \Omega$, for the second. The last two terms are bounded by $Ce^{ - n^{d-1}/C}$ by~\eqref{eq:surface-order-LDP}. 
The first two terms are symmetric, so we only consider one of them. By the triangle inequality, we have
\begin{align}\label{eqn:ajs4}
    \|\mathbb P(X_{t}^{\plus}\in \cdot) - \widehat \pi\|_\tv \le \|\mathbb P(X_{t}^{\plus}\in \cdot) - \mathbb P(\widehat X_{t}^{\plus}\in \cdot)\|_\tv +  \|\mathbb P(\widehat X_{t}^{\plus}\in \cdot) - \widehat \pi\|_\tv\,. 
\end{align}
By Claim~\ref{clm:monotonicity-relations} and the grand coupling, the first term in~\eqref{eqn:ajs4} is at most $\mathbb P(\widehat\tau^{\plus}\le t) \le Ce^{ - n^{d-1}/C}$ while $t\le e^{n^{d-1}/K}$. The second term is bounded by $Cn^d e^{ - g_n(t)/C}$ for $t\le e^{n^{d-1}/K}$ by Theorem~\ref{thm:Ising-mixing-restricted-chain}. All the other terms can be absorbed into this term since $g_n(t)\le n$. The constraint on $t$ can be dropped exactly as in the proof of Theorem~\ref{thm:Ising-mixing-restricted-chain}. 
\end{proof}

\subsection{Bounds obtained using known mixing time bounds with plus boundary condition}\label{subsec:Ising-main-proof}
In this subection, we use known bounds on the mixing time of Glauber dynamics on a box with $\plusone$ boundary condition of side-length $m$ to deduce the bounds of Theorem~\ref{thm:Ising-main} on the mixing time with random initialization on $(\mathbb Z/n\mathbb Z)^d$. In principle, we would like to leverage the fact that the total variation distance appearing in~\eqref{eq:assumption-scale-m} is from a $\plusone$ initialization to set up a recursion. However, unlike the distance from a worst-case initialization, the distance from the $\plusone$ initialization is not necessarily  sub-multiplicative in time; thus, we have to work with worst-case initializations to get the desired exponential decay at scale $m$.  

\begin{proof}[\textbf{\emph{Proof of Theorem~\ref{thm:Ising-main} assuming Theorem~\ref{thm:Ising-wsm-within-phase-intro}}}]
We begin with the weaker bound that holds for all $d\ge 3$. By a standard canonical paths argument (see~\cite{JS}), the mixing time of the Ising Glauber dynamics on $\Lambda_m^{\plus}$ is $\exp(O(m^{d-1}))$ at all $\beta>0$. By sub-multiplicativity of worst-case total variation distance, we get
\begin{align*}
    \max_{x_0}\|\mathbb P(X_{\Lambda_m^{\plus},t}^{x_0}\in \cdot) - \pi_{\Lambda_m^{\plus}}\|_\tv \le \exp\Big(- \frac{t}{Ce^{ Cm^{d-1}}}\Big).
\end{align*}
In particular~\eqref{eq:assumption-scale-m} is satisfied by the choice $f(m) = Ce^{C m^{d-1}}$. In turn, $g_n(t)\ge C^{-1}(\log t)^{1/(d-1)}\wedge n$ for some other constant $C$. Thus Theorem~\ref{thm:Ising-torus-restatement} implies that
\begin{align*}
    \|\mathbb P(X_{t}^{\nu^{\pmb{\pm}}}\in \cdot)  -\pi\|_\tv \le Cn^{d} \exp\big( - C^{-1} (\log t)^{1/(d-1)} \wedge n\big).
\end{align*}
This total variation distance will be $o(1)$ as soon as $t \ge \exp( C' (\log n)^{d-1})$ for a large enough $C'$.

In dimension $d=2$, it was shown in~\cite{LMST} that the mixing time of the Glauber dynamics on $\Lambda_m^{\plus}$ is at most $Cm^{C\log m}$ for all $\beta>\beta_c$. By sub-multiplicativity of worst-case total variation distance, 
\begin{align*}
    \max_{x_0}\|\mathbb P(X_{\Lambda_m^{\plus},t}^{x_0}\in \cdot) - \pi_{\Lambda_m^{\plus}}\|_\tv \le \exp\Big(- \frac{t}{Cm^{C\log m}}\Big).
\end{align*}
In particular~\eqref{eq:assumption-scale-m} is satisfied by choosing $f(m) = Cm^{C\log m}$, so that $g_n(t) \ge e^{(\log t)^{1/2}/C} \wedge n$. Thus Theorem~\ref{thm:Ising-torus-restatement} implies that 
\begin{align*}
    \|\mathbb P(X_{t}^{\nu^{\pmb{\pm}}} \in \cdot) - \pi\|_\tv \le Cn^{d} \exp\Big(- e^{(\log t)^{1/2}/C} \wedge n\Big),
\end{align*}
which will be $o(1)$ as soon as $t = (\log n)^{C' \log \log n}$ for a sufficiently large $C'$. Clearly this timescale is $N^{o(1)}$ and we recall the extra factor of $N$ comes from the switch from continuous to discrete time. 
\end{proof}

\begin{remark}
It should be clear that any improvement in the mixing time
bounds with $\plusone$ boundary condition will lead to an improvement in the
bound of Theorem~\ref{thm:Ising-main}.  In particular, as mentioned in
the introduction, it is conjectured that the mixing time with 
$\plusone$ boundary condition is in fact ${\rm poly}(m)$, which would 
imply a bound on the mixing time from the random $\nu^{\pmb{\pm}}$ initialization
of $O(N{\rm polylog}(N))$.  Indeed, it is already known, by
generalizing the proof approach of~\cite[Theorem 3.1]{Martinelli-SP} to $d\ge 3$ using bounds on interface fluctuations from~\cite{Dobrushin72a,GL-Ising-max},
that the relatively
trivial bound of $\exp(O(m^{d-1}))$ used above can be improved
to $\exp(O(m^{d-2}\log m))$, at least at all sufficiently low temperatures. This improves the bound in Theorem~\ref{thm:Ising-main} for $d\ge 3$ to $\exp( C' (\log N)^{d-2} \log \log N)$ at sufficiently low temperatures.
\end{remark}

We conclude the section by explaining how the above arguments also prove Corollary~\ref{cor:infinite-volume-relaxation}. 
\begin{proof}[\textbf{\emph{Proof of Corollary~\ref{cor:infinite-volume-relaxation} assuming Theorem~\ref{thm:Ising-wsm-within-phase-intro}}}]
Suppose that the local function $F$ depends only on the spins in the ball of radius $r$ centered at the origin, denoted by $B_r$. Then for any $F:\{\pm 1\}^{B_{r}}\to \mathbb R$, the left-hand side of~\eqref{eqn:infvolcor} reduces to
\begin{align*}
    |\mathbb E[F(X_{\mathbb Z^d,t}^{\plus})] - \pi_{\mathbb Z^d}^{\plus}[F(\sigma)]|& \le {\|F\|}_\infty \|\mathbb P(X_{\mathbb Z^d,t}^{\plus}(B_{r})\in \cdot) - \pi_{\mathbb Z^d}^{\plus}(\sigma_{B_{r}}\in \cdot) \|_\tv \\
    & \le {\|F\|}_\infty \sum_{u\in B_{r}} \big(\mathbb P(X_{\mathbb Z^d,t}^{\plus}(u)= +1) - \pi_{\mathbb Z^d}^{\plus}(\sigma_u = +1)\big) \\
    & = C_F \big(\mathbb P(X_{\mathbb Z^d,t}^{\plus}(o)= +1) - \pi_{\mathbb Z^d}^{\plus}(\sigma_o = +1)\big)\,,
\end{align*}
where $o$ is the origin, and in the last equality we used vertex transitivity of the dynamics and the infinite-volume measure, and let $C_F = {\|F\|}_\infty r^d$. 

If we now let $P_t f_o(\plus) = \mathbb P(X_{\mathbb Z^d,t}^{\plus}(o)= +1) - \pi_{\mathbb Z^d}^{\plus}(\sigma_o = +1)$ and follow the proof of Proposition~\ref{prop:single-site-relaxation}, replacing $\widehat P_t f_v(\plus)$ by $P_t f_o(\plus)$, $\widehat X_t^{\plus}$ by $X_{\mathbb Z^d,t}^{\plus}$, and $\widehat \pi$ by $\pi_{\mathbb Z^d}^{\plus}$ (in fact the proof simplifies because we are setting $n$ to infinity, and we do not have to deal with the stopping time $\widehat \tau$), we obtain 
\[
P_t f_o(\plus) \le Ce^{ - g_\infty (t)}\,, \qquad \mbox{where} \qquad g_\infty(t) = \max\{m: mf(m)\le t\}\,,
\]
for a sequence $f(m)$ satisfying~\eqref{eq:assumption-scale-m}. 
At this point, following the proof of Theorem~\ref{thm:Ising-main}, in $d=2$ we use the bound of~\cite{LMST} on the worst-case mixing time of the Glauber dynamics on $\Lambda_m^{\plus}$ to obtain $g_\infty(t) \ge e^{(\log t)^{1/2}/C}$, and in $d\ge 3$ we apply the crude $\exp(O(m^{d-1}))$ bound to obtain $g_\infty(t) \ge C^{-1}(\log t)^{1/(d-1)}$.
\end{proof}

\section{WSM within a phase for the low-temperature Ising model}\label{sec:Ising-wsm-within-phase}
In this section we prove Theorem~\ref{thm:Ising-wsm-within-phase-intro}, showing that WSM within a phase holds for the Ising model at all low temperatures in all dimensions. Recall that this is the only missing ingredient in the proof of the upper bound of our main result, Theorem~\ref{thm:Ising-main}.
We will actually prove the following stronger notion of WSM within a phase, where the marginals of the two measures are compared on balls, rather than single sites. 

\begin{definition}\label{def:wsm-within-a-phase}
Fix $d\ge 2$ and $\beta>\beta_c(d)$. We say the Ising model on $\Lambda_n$ has \emph{weak spatial mixing (WSM) within a phase} if for every $r<n$, for every $v\in \Lambda_n$, 
\begin{align*}
\|\pi_{B_r^{\plus}(v)}(\sigma_{B_{r/2}(v)}\in \cdot) - \widehat \pi_{\Lambda_n^p}(\sigma_{B_{r/2}(v)}\in \cdot ) \|_\tv \le Ce^{ - r/C}\,.
\end{align*}
\end{definition}

To prove Theorem~\ref{thm:Ising-wsm-within-phase-intro}, we will construct an explicit coupling between $\sigma\sim \pi_{B_r^{\plus}(v)}$ and $\sigma'\sim \widehat \pi$ such that $\sigma$ and $\sigma'$ agree on $B_{r/2}(v)$ except with probability $Ce^{ - r/C}$. An analogous coupling for the corresponding random-cluster representations $\omega,\omega'$ is available due to recent results in~\cite{DCGR20}. We construct a ``good" set~$\cE$ of pairs of random-cluster configurations such that, if $(\omega,\omega')\in \cE$, we can lift the coupling of $\omega,\omega'$ to a coupling of the Ising configurations $(\sigma,\sigma')$ on $B_{r/2}(v)$ via the \emph{Edwards--Sokal coupling}~\cite{ES}. We then use a powerful coarse-graining technique of~\cite{Pisztora96} to construct a coupling under which $(\omega,\omega')\in \cE$ except with probability~$Ce^{ - r/C}$.   

Sections~\ref{subsec:random-cluster-rep} and~\ref{subsec:coarse-graining} describe the necessary background on the random-cluster representation and coarse-graining techniques, respectively. In Section~\ref{subsec:coupling-rc-configurations}, we use coarse-graining to couple the random-cluster configurations $\omega,\omega'$ and define a high-probability event~$\cE$ for this coupling. In Section~\ref{subsec:edwards-sokal-pi-hat}, we give a way to generate samples from the conditional measure $\widehat \pi$ using the random-cluster representation, and in Section~\ref{subsec:proof-of-wsm} we combine these ingredients to couple the Ising configurations corresponding to $\omega,\omega'$ on $B_{r/2}(v)$ on the event~$\cE$. This results in the proof of Theorem~\ref{thm:Ising-wsm-within-phase-intro}.

\subsection{The random-cluster representation}\label{subsec:random-cluster-rep}
We first formally define the random-cluster representation of the Ising model, and the Edwards--Sokal coupling of the random-cluster model to the Ising model. For more details, and a discussion of the random-cluster model for non-integer~$q$, we refer the reader to~\cite{Grimmett}. 

For $p\in [0,1]$ and $q>0$, the random-cluster model at parameters $(p,q)$ on a (finite) graph $G = (V(G), E(G))$ is the distribution over edge-subsets $\omega \subset E(G)$, naturally identified with  $\omega\in \{0,1\}^{E(G)}$ via $\omega(e) := 1$ if and only if $e\in \omega$, given by 
\begin{align}\label{eq:rcmeasure}
\pi^{\rc}_{G}(\omega) = \frac{1}{\mathcal Z^\rc_{G,p,q}} p^{|\omega|} (1-p)^{|E(G)| - |\omega|} q^{|\mathsf{Comp}(\omega)|}\,,
\end{align}
where $|\mathsf{Comp}(\omega)|$ is the number of connected components (clusters) in the subgraph $(V(G),\omega)$. 

When $\omega(e) = 1$, we say that $e$ is \emph{wired} or \emph{open}, and when $\omega(e) = 0$, we say it is \emph{free} or \emph{closed}. If $x,y$ are in the same connected component of the subgraph $(V(G),\omega)$, we write $x\xleftrightarrow[]{\omega} y$. 

\subsubsection{Random-cluster boundary conditions}
When $G$ is a subset of $\mathbb Z^d$, e.g., $G = \Lambda_m$, it will be important for us to introduce the notion of \emph{boundary conditions} for the random-cluster model. 

\begin{definition}
A random-cluster \emph{boundary condition} $\xi$ on a subset $G \subset \mathbb Z^d$ is a partition of the (inner) boundary $\partial G = \{v\in G: d(v,\mathbb Z^d \setminus G) =1\}$ such that the vertices in each part of the partition are identified with one another. The random-cluster measure with boundary condition~$\xi$, denoted $\pi^{\xi}_{G,p,q}$, is the same as in~\eqref{eq:rcmeasure} except that $\mathsf{Comp}(\omega)$ is replaced by $\mathsf{Comp}(\omega;\xi)$, counted with this vertex identification.
     Alternatively, $\xi$ can be seen as introducing ghost ``wirings" of  vertices in the same part of the partition.
\end{definition} 

     The \emph{free} boundary condition, $\xi = \zero$, is the one whose partition of $\partial G$ consists only of singletons.
     The \emph{wired} boundary condition on $\partial G$, denoted $\xi = \one$, is that whose partition has all vertices of $\partial G$ in the same part. There is a natural stochastic order on boundary conditions, given by $\xi \le \xi'$ if $\xi$ is a refinement of~$\xi'$. The wired/free boundary condition is then maximal/minimal under this order.
 
     Finally, a class of boundary conditions that will recur are those induced by a configuration on $\mathbb Z^d \setminus G$: given a random-cluster configuration $\eta$ on $E(\mathbb Z^d) \setminus E(G)$, the boundary condition it induces on $G$ is that given by the partition where $v,w\in \partial G$ are in the same element if and only if $v\xleftrightarrow[]{\eta} w$. 

\subsubsection*{Edwards--Sokal coupling}
For $q=2$, there is a canonical way to couple the random-cluster model to the Ising model so that the random-cluster model encodes the correlation structure of the Ising model. (An analogous coupling holds for all integer~$q$, relating the random-cluster model to the Potts model.)  The ($q=2$) \emph{Edwards--Sokal} coupling $\pi^\es_G$ is the probability distribution over spin-edge pairs $(\omega,\sigma)\in \{0,1\}^{E(G)} \times \{\pm 1\}^{V(G)}$ given by:
\begin{enumerate}
    \item sampling a random-cluster configuration $\omega\sim \pi^\rc_G$ at parameters $p=1-e^{ - \beta}$ and $q=2$;
    \item independently assigning (coloring) each connected component $\cC$ of $\omega$ a random variable $\eta_\cC$ which is $\pm1$ with probability $\frac 12$-$\frac 12$, and setting $\sigma_v = \eta_\cC$ for every $v\in \cC$,. 
\end{enumerate} 
The marginal of this coupling on $\sigma$ then gives a sample from $\pi_G$ at inverse temperature~$\beta$. 

We will work with this coupling extensively, as it will give us a mechanism for boosting couplings of random-cluster configurations to couplings of Ising configurations. 
The coupling can also be used in the presence of Ising boundary conditions that are $\plusone$ (symmetrically, all $\minusone$) as follows. In step (1) above, draw $\omega\sim \pi_{G^\one}^\rc$, and in step (2) above, if $\cC({\partial G})$ is the connected component of the boundary, then deterministically set $\eta_{\cC({\partial G})} = +1$ (symmetrically, $-1$), while coloring all other clusters independently as before. This will yield an Ising configuration drawn from $\pi_{(G\cup \partial G)^{\plus}}$ (symmetrically, $\pi_{(G\cup \partial G)^{\minus}}$).

\subsubsection*{Random-cluster phase transition on $\mathbb Z^d$}
The $q=2$ random-cluster model has a phase transition matching that of the Ising model on subsets of $\mathbb Z^d$. Namely, there exists $p_c(d) = 1-e^{ - \beta_c(d)}$ such that, when $p<p_c(d)$, the probability under $\pi^\rc_{\mathbb T_m}$ that $v\stackrel{\omega}\leftrightarrow w$ decays exponentially in $d(v,w)$, while when $p>p_c(d)$, that probability stays uniformly (in $m$ and $v,w\in \mathbb T_m$) bounded away from zero.  

Pisztora~\cite{Pisztora96}, towards proving~\eqref{eq:surface-order-LDP}, obtained more refined information on the random-cluster model at $p>p_c$, showing that configurations typically have one macroscopically sized giant component, and all other components have exponential tails on their size. Recently, using inputs from the \emph{random-current representation} of the Ising model, the following weak spatial mixing property was established for the $q=2$ (Ising) random-cluster model in general dimension in~\cite{DCGR20}. 

\begin{theorem}[{\cite{DCGR20}}]\label{thm:rc-wsm}
Let $d \ge 2$, $q=2$ and $p> p_c(2,d)$. Then 
\begin{align*}
    \|\pi_{\Lambda_n^\one}^\rc(\omega(\Lambda_{n/2})\in \cdot) - \pi_{\Lambda_n^\zero}^\rc(\omega(\Lambda_{n/2}) \in \cdot)\|_\tv \le Ce^{ - n/C}\,.
\end{align*}
\end{theorem}
This result suggests, roughly, that the lack of spatial mixing in the Ising model at low temperatures comes solely from the second step in the Edwards--Sokal coupling, where the random-cluster components are assigned independent spins.  Our Theorem~\ref{thm:Ising-wsm-within-phase-intro} is a formalization of this intuition.

\subsection{Coarse graining of the random-cluster model}\label{subsec:coarse-graining}
In this subsection, we describe the coarse-graining technique of~\cite{Pisztora96}, which will be a key tool in our proof of Theorem~\ref{thm:Ising-wsm-within-phase-intro}. 

Consider a tiling of $\mathbb T_m$ by blocks ($\ell_\infty$ balls of radius~$k$).
 The coarse-graining approach assigns to each block an indicator random variable according to some property of the random-cluster configuration in the block that serves as a signature of the low-temperature regime.  By taking $k$ sufficiently large (depending on $p,d$), the probability that a block satisfies that property can be made arbitrarily close to~$1$. The resulting process can then be compared to a Bernoulli percolation process so that, when $k$ is sufficiently large, the set of blocks {\it not\/} satisfying the signature property form a sub-critical percolation process, with exponential tails.  Thus, moving to blocks serves to {\it boost\/} the spatial prevalence of the signature property.

\subsubsection{$k$-good blocks}
For a fixed $k$, let $\Lambda_m^{(k)} = \Lambda_m \cap k\mathbb Z^d$, and tile $\mathbb T_m = \Lambda_m^p$ by overlapping blocks $(B_x)_{x\in \Lambda_m^{(k)}}$ given by $B_x = B_k(x)$, i.e., the $\ell_\infty$ balls of $\mathbb T_m$ of radius $k$, centered about $x\in \Lambda_{m}^{(k)}$, when the fundamental domain of $\mathbb T_m$ is identified with $\Lambda_m$. (Without loss of generality, we will assume that $m$ is a multiple of~$k$---it will be evident that any remainder issues could otherwise be handled easily.)  The following will be our signature property for these blocks.

\begin{figure}
    \centering
    \includegraphics[width = .4\textwidth]{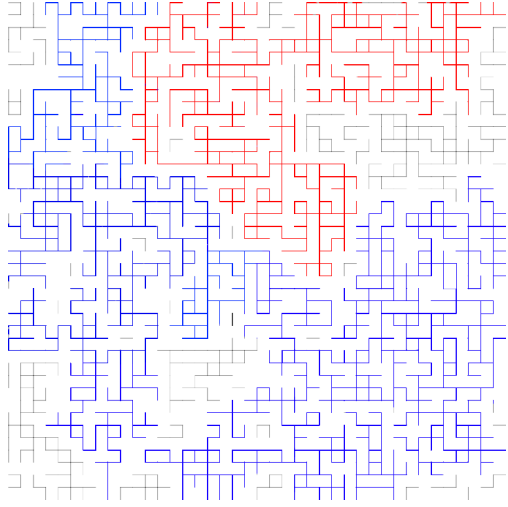}
    \caption{A $k$-bad block of a low-temperature random-cluster configuration: even though its giant component (blue) intersects all boundary sides as desired, the configuration has a second distinct component (red) of size greater than $k$.}
    \label{fig:k-good-block}
\end{figure}

\begin{definition}\label{def:good-low-temp}
When $p> p_c(d)$, a random-cluster configuration $\omega$ on a block $B_x$, is called $k$-\emph{good} if 
\begin{enumerate}
    \item There is at most one open cluster of size (number of vertices) at least $k$ in $\omega$.
    \item There exists an open cluster in $\omega$ intersecting all $2d$ sides of $\partial B_x$. 
\end{enumerate}
 Given $k$ and a random-cluster configuration $\omega$ on $\mathbb T_m$, we say $B_x$ is $k$-\emph{good} in $\omega$ if $\omega(B_x)$ is $k$-good. We call $\omega$ (resp., $B_x$) $k$-\emph{bad} if it is not $k$-good. See Figure~\ref{fig:k-good-block} for a visualization of a $k$-bad block. 
\end{definition}

The definition of \emph{good} blocks in~\cite{Pisztora96} was actually stronger than the above (cf~\cite[Theorem 3.1]{Pisztora96}); we borrow our definition from~\cite{DCGR20}. The results of~\cite{Pisztora96}, together with~\cite{Bodineau05}, directly imply the following. 

\begin{lemma}[{\cite{Pisztora96,Bodineau05}}]\label{lem:good-whp-low-temp}
For all $d\ge 2$ and $p>p_c(d)$, there exists $C>0$ such that, for every $k$, 
\begin{align}\label{eqn:PBlemma}
    \inf_{\xi} \pi^\rc_{B_{2k}^{\xi}(v)} (B_v\mbox{ is $k$-good}) \ge 1-Ce^{ - k/C}\,,
\end{align}
where the infimum is over boundary conditions $\xi$ on $B_{2k}(v)$. 
\end{lemma}

\noindent Note that the infimum in~\eqref{eqn:PBlemma} is important because the event that $\omega(B_v)$ is $k$-good is {\it not\/} monotone. 

\subsubsection{The coarse-grained percolation process}
The above notions of $k$-goodness can be used to define a coarse-graining of a configuration $\omega$ on $\mathbb T_m$, via the percolation process of $k$-\emph{good} blocks. 

\begin{definition}\label{def:k-good-coarse-graining}
Fix $q\ge 1$, $d\ge 2$ and $p>p_c(d)$. For $k$ fixed, and a configuration $\omega$ on $\Lambda_m$, define a {\it coarse-grained site percolation} $\eta(\omega) = \eta^{(k)}(\omega): \mathbb T_{m}^{(k)} \to \{0,1\}$ by setting the random variable $\eta_x(\omega)$ to be 1 (open) if $\omega(B_x)$ is $k$-good (according to Definition~\ref{def:good-low-temp}), and $0$ if it is $k$-bad. 
\end{definition}

By~\cite[Theorem 1.3]{LSS97}, on a graph $G$ of degree at most $d$, for every $r$ and $\varepsilon>0$, there exists $\delta(r,\varepsilon,d)>0$ such that every site percolation process $\zeta: V(G)\to \{0,1\}$ having 
\begin{align*}
    \min_{v\in V(G)} \mathbb P(\zeta_v = 1\mid \{\zeta_w: d_G(w,v)\ge r\}) \ge 1-\delta
\end{align*}
stochastically dominates the i.i.d.\ $\mbox{Ber}(1-\varepsilon)$ independent percolation on $G$. This implies that, by taking $k$ large enough, we can ensure that $\eta$ stochastically dominates an i.i.d.\ percolation process on $\mathbb T_m^{(k)}$ whose parameter can be taken as close to $1$ as we like. The result is in fact quantitative in relating $r, \varepsilon, d$ and $\delta$, i.e., both $\epsilon$ and $\delta$ can be taken to be exponentially decaying in~$k$.
Thus, we get the following sharper statement.

\begin{corollary}
\label{cor:stoch-domination}
Fix $d\ge 2$ and $p>p_c(d)$. There exists $C$ such that, for every $k$, if $\omega \sim \pi_{\mathbb T_m}^\rc$ and $(\tilde \eta_x)_{x\in \mathbb T_m^{(k)}}$ are i.i.d.\ $\mbox{Ber}(1-Ce^{ - k/C})$, we have the stochastic domination 
\begin{align*}
    \eta(\omega) \succeq \tilde \eta\qquad \mbox{ on }\qquad \mathbb T_m^{(k)}\,.
\end{align*}
By similar reasoning, for any $\xi$, if $\omega \sim \pi_{\Lambda_m^\xi}^{\rc}$, we have
\begin{align*}
    \eta(\omega) \succeq \widetilde \eta \qquad \mbox{ on }\qquad \Lambda_{m-2k}^{(k)}\,.
\end{align*}
\end{corollary} 

\subsubsection{Separating surfaces of good blocks}
Not only does the percolation process of $k$-\emph{good} blocks, and in particular sub-criticality of $k$-\emph{bad} blocks, control the typical connectivity structure of the random-cluster model as in its application in~\cite{Pisztora96}, it can sometimes be used to couple configurations inside \emph{separating surfaces} of the coarse-graining; this was done in e.g.,~\cite{DCGR20} for the Ising model, using a priori inputs from the random-current representation of the Ising model.

Let us formalize what we mean by a separating surface. We begin by defining clusters of the coarse-grained percolation process, both on the original coarse-grained graph $\mathbb T_m^{(k)}$ as well as under a stronger form of adjacency, which we call $\star$-adjacency (allowing diagonal adjacencies), corresponding to $\ell_\infty$ distance. 

\begin{definition}\label{def:k-adjacency}
Call two vertices $x,y \in k\mathbb Z^d$ \emph{$k$-adjacent}, denoted $x\sim_{k} y$, if they are a (graph) distance at most $k$ apart.
Call two vertices $x,y\in k\mathbb Z^d$ \emph{$k$-$\star$-adjacent}, denoted $x\sim_{k}^\star y$, if they are at $\ell_\infty$ distance at most $k$. Observe that for $x\ne y$, we have $x\sim_k^\star y$ if and only if $B_k(x) \cap B_k(y)\ne \emptyset$. 
\end{definition}

\begin{definition}
Consider a site percolation $\zeta: \Lambda_m^{(k)} \to \{0,1\}$ for some $\Lambda^{(k)}\subset k\mathbb Z^d$. 
An {\it (open) $k$-cluster of $\zeta$} is a maximal $k$-connected (via $k$-adjacency) component of the vertices $\{v\in \Lambda^{(k)}: \zeta_v = 1\}$. An {\it (open) $k$-$\star$-cluster of $\zeta$} is a  maximal $k$-$\star$-connected (via $k$-$\star$-adjacency) component of $\{v\in \Lambda^{(k)}: \zeta_v = 1\}$. The $k$-cluster containing a vertex~$v$ is denoted $\cC_v(\zeta)$, and the $k$-$\star $-cluster containing~$v$ is denoted $\cC_v^\star(\zeta)$. 

We write $x\xleftrightarrow[]{\zeta} y$ if $y\in \cC_x(\zeta)$, and  $x\xleftrightarrow[]{\zeta}_\star y$ if $y\in \cC_x^\star(\zeta)$. We will often be interested in the clusters of the {\it complement\/} of $\zeta$, i.e., $\mathbf 1-\zeta$, defined pointwise as $(\mathbf 1-\zeta)_x = 1-\zeta_x$ for all $x\in \mathbb T_m^{(k)}$. 
\end{definition}

\begin{definition}\label{def:separating-surface}
Consider a percolation $\zeta: \Lambda_m^{(k)} \to \{0,1\}$ and let $l<m-2k$. We say that $\zeta$ has an {\it open separating $k$-surface\/} in $\Lambda_m^{(k)}\setminus \Lambda_l^{(k)}$ if  the following event occurs:  
$$\Big\{\zeta: \Lambda_{l+k}^{(k)}\xleftrightarrow[]{\one-\zeta}_\star \partial\Lambda_{m-k}^{(k)}\Big\}^c\,,$$
i.e., there is no $(\one - \zeta)$-open $k$-$\star$-cluster intersecting both $\Lambda_{l+k}^{(k)}$ and $(\Lambda_{m-k}^{(k)})^c$. 
Any (minimal) set $\Gamma$ of open sites of $\Lambda_{m}^{(k)}$ in~$\zeta$ that serves as a witness to this event (i.e., the event holds no matter the values of $\zeta$ outside $\Gamma$) is called an \emph{open separating $k$-surface}. 
\end{definition}

\begin{remark}\label{rem:outermost-separating-surface}
Observe that for any fixed set of vertices $A^{(k)}$ such that $\partial \Lambda_m^{(k)}  \subset A^{(k)} \subset \Lambda_m^{(k)}$, if we let $D = \bigcup_{v\in A^{(k)}} \cC^\star_{v}(\zeta)$ denote its open $\star$-component(s) in $\zeta$, then if $D \cap \Lambda_{l+k}^{(k)} = \emptyset$, the outer $\star$-boundary $$\partial_{\textrm{out}} D = \{u \in \Lambda_m^{(k)} \setminus D : u\sim^{k}_\star D\}$$ of $D$ in $\Lambda_m^{(k)}$ yields an open separating $k$-surface of $\zeta$ in $\Lambda_m^{(k)}\setminus \Lambda_l^{(k)}$.  
\end{remark}

Since $\Gamma_k \subset \Lambda_m^{(k)}$, it necessarily splits $k\mathbb Z^d\setminus \Gamma_k$ into exactly one infinite $k$-$\star$-connected component and some finite ones: call the infinite one $\Ext(\Gamma_k)$, and call the union of the finite ones $\Int(\Gamma_k)$, or in other words $\Int(\Gamma_k) = \Lambda_m^{(k)} \setminus (\Gamma_k \cup \Ext(\Gamma_k))$. One can observe (as established in~\cite{DeuschelPisztora96}) that an open separating $k$-surface $\Gamma_k$ is such that any $k$-$\star$-connected path from $\Int(\Gamma_k)$ to $\Ext(\Gamma_k)$ must intersect $\Gamma_k$. In particular, if $\Gamma_k \subset \Lambda_m^{(k)}\setminus \Lambda_l^{(k)}$ is an open separating $k$-surface, then $\Lambda_l^{(k)} \subset \Int(\Gamma_k)$ and any $k$-$\star$-connected path from $\Lambda_l^{(k)}$ to $\partial \Lambda_m^{(k)}$ must hit~$\Gamma_k$.

%One can observe (as established in~\cite{DeuschelPisztora96}) that if $\Gamma_k \subset \Lambda_m^{(k)} \setminus \Lambda_l^{(k)}$ is an open separating $k$-surface, it is a connected set of open sites of~$\zeta$ such that any $k$-$\star$-connected path from $\Lambda_l^{(k)}$ to $\Lambda_m^{(k)}$ intersects $\Gamma_k$. Thus, $\Gamma_k$  necessarily splits $k\mathbb Z^d \setminus \Gamma_k$ into exactly one infinite $\star$-connected component and some finite ones: call the infinite one $\mathsf{Ext}(\Gamma_k)$ and let $\mathsf{Int}(\Gamma_k) = \Lambda_m^{(k)}\setminus (\Gamma_k \cup \mathsf{Ext}(\Gamma_k))$. 

We make the following key observation, whose proof is deferred to Appendix~\ref{app:deferred-proofs}, that will be central to the applications of the low-temperature coarse-graining, and will be used later.

\begin{observation}\label{obs:connected-component-good-blocks}
	Let $\eta^k(\omega)$ be as in Definition~\ref{def:good-low-temp}. Suppose $A_k$ is a $k$-connected open set of $\Lambda_n^{(k)}$ in $\eta^{k}(\omega)$, and let $A = \bigcup_{x\in A_k} B_x$. Then $\omega(A)$ has exactly one component of size greater than or equal to $k$. 
\end{observation}

Using this observation, one obtains the following Markov-like property for the low-temperature coarse graining.

\begin{figure}
	\centering
	\includegraphics[width = .4\textwidth]{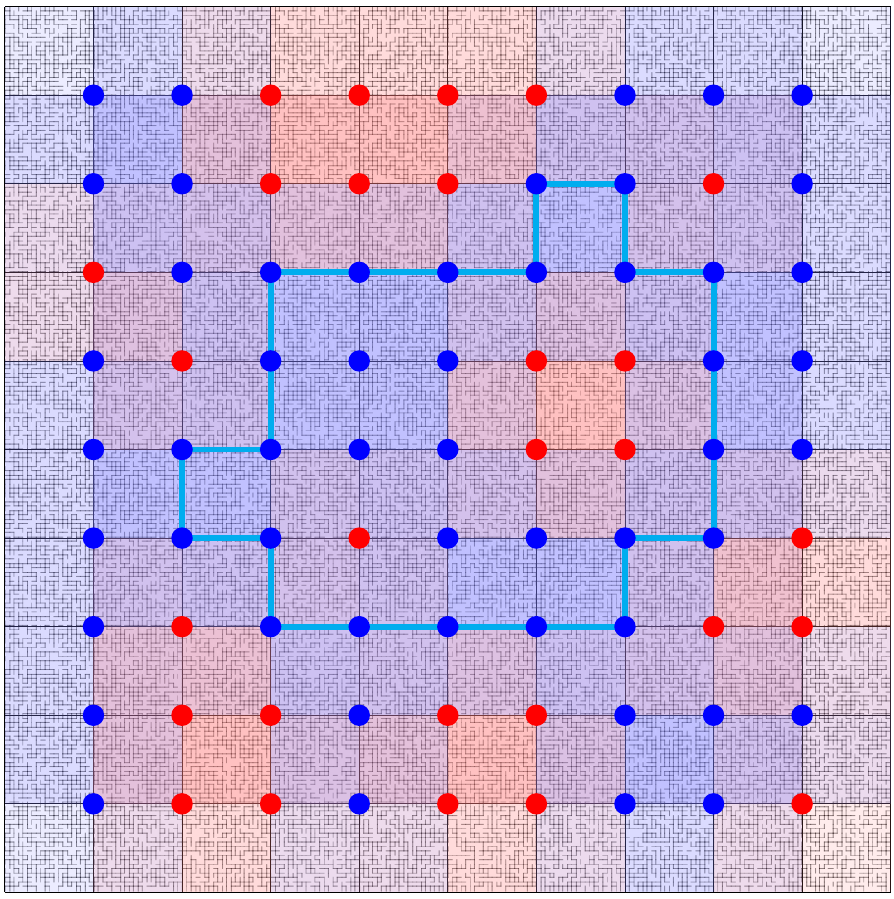}
	\qquad \qquad
	\includegraphics[width = .402\textwidth]{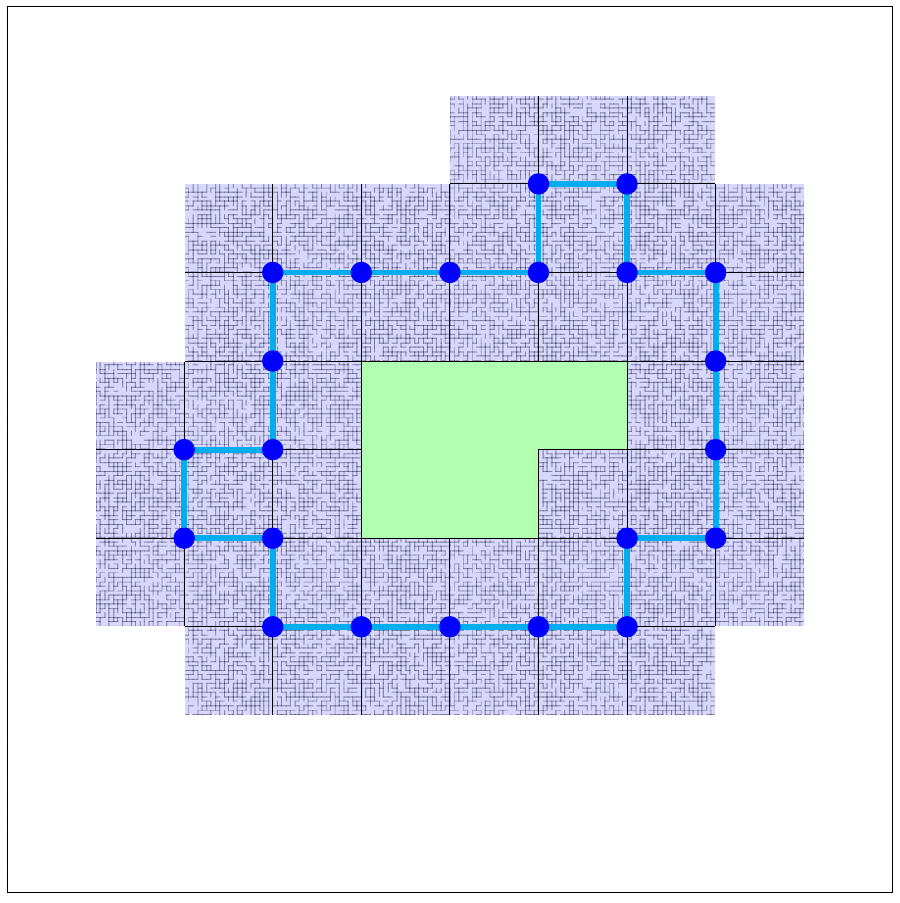}
	\caption{Left: a random-cluster configuration at $p>p_c(q,d)$ with its coarse graining; the coarse graining features a separating surface $\Gamma_k$ (cyan). Right: The separating surface $\Gamma_k$, with its corresponding blocks $\Gamma$ is shown in blue; the distribution in $\Int(\Gamma)$ (shaded green) is conditionally independent of the distribution in $\Ext(\Gamma)\setminus \Gamma$ (white).}
	\label{fig:separating-surface}
\end{figure}

\begin{lemma}\label{lem:separating-surface-disconnects-information}
	Let $\eta(\omega)$ be as in Definition~\ref{def:good-low-temp}. Consider a random-cluster configuration $\omega$ in $\mathbb Z^d$ and let $\Gamma_k(\omega)$ be any open separating $k$-surface of $\eta(\omega)$. If 
	$$\Gamma  = \bigcup_{x\in \Gamma_k} B_x \qquad \mbox{and}\qquad \mathsf{Ext}(\Gamma) = \bigcup_{x\in \mathsf{Ext}(\Gamma)}B_x\,,$$ then the boundary condition induced by $\omega(\Gamma \cup \mathsf{Ext}(\Gamma))$ on $\Lambda_m \setminus (\Gamma \cup \mathsf{Ext}(\Gamma))$ depend only on $\omega(\Gamma)$. 
\end{lemma}

We refer the reader to Figure~\ref{fig:separating-surface} for a visualization of Lemma~\ref{lem:separating-surface-disconnects-information}. An analog of it was established in~\cite[Section 3.2]{DCGR20}. For completeness, we include a proof in Appendix~\ref{app:deferred-proofs}.

\subsection{Coupling random-cluster configurations via coarse-graining}\label{subsec:coupling-rc-configurations}
A key tool in our proof of WSM within a phase for the Ising model up to its critical point is a coupling of random-cluster configurations $\omega(\Lambda_{r/2}),\omega'(\Lambda_{r/2})$ that are drawn from distinct random-cluster measures, e.g., $\pi^{\rc}_{\Lambda_r^\one}$ and $\pi^\rc_{\Lambda_n^p}$. 
Our aim in this subsection is to couple the boundary conditions induced by $\omega(\Lambda_r \setminus \Lambda_{r/2})$ and $\omega'(\Lambda_n \setminus \Lambda_{r/2})$ on $\Lambda_{r/2}$ , as well as, respectively, the component of $\omega$ intersecting $\partial \Lambda_r$ and the largest component of $\omega'$. This stronger coupling will be central to our ability to couple the corresponding Ising configurations from $\pi_{B_r^{\plus}(v)}$ and $\widehat \pi$. 

To this end, we define a different coarse-graining which applies to pairs of random-cluster configurations. This coarse-graining is based on one found in~\cite{DCGR20} and used for proving Theorem~\ref{thm:rc-wsm}. 

\begin{definition}
For random-cluster configurations $(\omega,\omega')$, we say that a block $B_v$ is $k$-\emph{very good} if $\omega(B_v) = \omega'(B_v)$ and that common configuration is $k$-good, per Definition~\ref{def:good-low-temp}. Define the $k$-\emph{very-good coarse graining}, $\zeta^{(k)}(\omega, \omega'): \Lambda_n^k \to \{0,1\}$ as taking the value $1$ if $B_v$ is very good, and $0$ otherwise. 
\end{definition}

\begin{definition}
Let $\cE_{\textsc{vg}}$ be the event that the very good coarse-graining $\zeta^{(k)}(\omega,\omega')$ has a $k$-very good separating surface in the annulus $\Lambda_{m}\setminus \Lambda_{m/2}$. 
\end{definition}

The main result of this subsection is the following, and is very similar to Lemma 3.3 of~\cite{DCGR20}.
\begin{lemma}\label{lem:E-very-good-probability}
Suppose $k$ is sufficiently large. Then there is a monotone coupling $\mathbb P$ such that for any $\xi,\xi'$, if $\omega\sim \pi_{\Lambda_m^{\xi}}$ and $\omega' \sim \pi_{\Lambda_m^{\xi'}}$, 
\begin{align*}
    \mathbb P((\omega,\omega') \notin \cE_{\textsc{vg}}) \le Ce^{ - m/C}\,, 
\end{align*}
and such that, on the event $\cE_{\textsc{vg}}$, if $\Gamma_k$ denotes the outermost separating surface of $\zeta^{(k)}(\omega,\omega')$, 
\begin{enumerate}
    \item The boundary condition induced on $\Int(\Gamma)$ by $\omega(E(\Lambda_m^\xi) \setminus \Int(\Gamma))$ are identical to those induced by $\omega'(E(\Lambda_m^{\xi'}) \setminus \Int(\Gamma))$; and
    \item $\omega(\Int(\Gamma)) = \omega'(\Int(\Gamma))$.
\end{enumerate}
\end{lemma}

The proof of this lemma is very similar to that of~\cite[Lemma 3.3]{DCGR20} where it was applied simply to establish Theorem~\ref{thm:rc-wsm}. 
Due to this similarity, we defer the proof and only include it for completeness in Appendix~\ref{app:deferred-proofs}. Intuitively, the coupling goes by revealing, from the outside in, the components of $k$-very bad blocks, so that when the revealing process terminates one has exposed only the outermost separating surface of $k$-very good blocks and nothing inside it. Lemma~\ref{lem:separating-surface-disconnects-information} then ensures that the two configurations can be drawn from the identity coupling inside the separating surface of $k$-very good blocks.

\subsubsection{Using coarse-graining to match the giant components}
While the very good separating surface is sufficient for coupling the component structure of the random-cluster configurations $(\omega(\Lambda_{r/2}), \omega'(\Lambda_{r/2}))$, we also need to couple their corresponding Ising configurations, including in the presence of (a) $\plusone$ boundary condition or (b) conditioning of the form of $\widehat \pi$. This requires a more refined understanding of the component structures of the random-cluster measures, which is the purpose of this section. 

We define two different events for the random-cluster measure on $\Lambda_m$, used together in the proof of Theorem~\ref{thm:Ising-wsm-within-phase-intro}. We refer the reader to Figure~\ref{fig:E-events} for visualizations. The first such event is the following. 

\begin{figure}[t]
	\begin{subfigure}[b]{.48\textwidth}
	\centering
		\begin{tikzpicture}[scale = .45]
		\node at (0,0) {\includegraphics[width = 4.5cm]{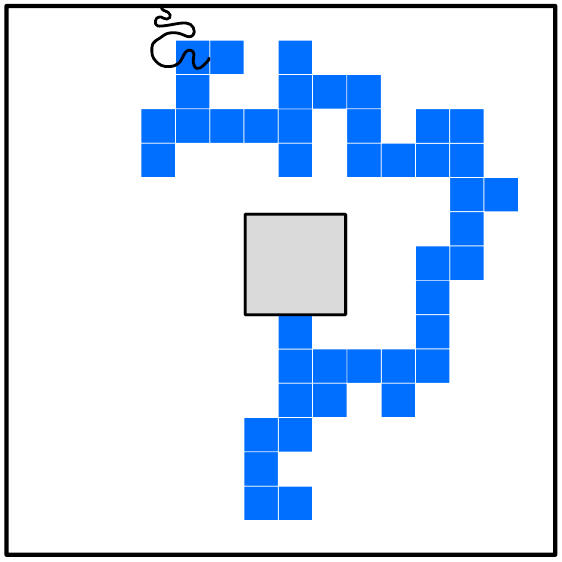}};
		\node at (.25,.25) {$A$};
%		\draw[|-|] (-5,-2.8)--(5,-2.8);
%		\draw[->] (0,-2.8)--(0,3);
%		\shade[ball color = red!40, opacity = 1] (0,-2.525) circle (.25cm);
%		\draw [->] (.3,-2.5) to [out=10,in=220] (1,-2.2);
%		\draw [->] (-.3,-2.5) to [out=180-10,in=180-220] (-1,-2.2);
		\end{tikzpicture}
		\subcaption{$\mathcal E_{m,A}$}
	\end{subfigure}
	\begin{subfigure}[b]{.48\textwidth}
	\centering
		\begin{tikzpicture}[scale = .45]
		\node at (0,0) {\includegraphics[width = 4.5cm]{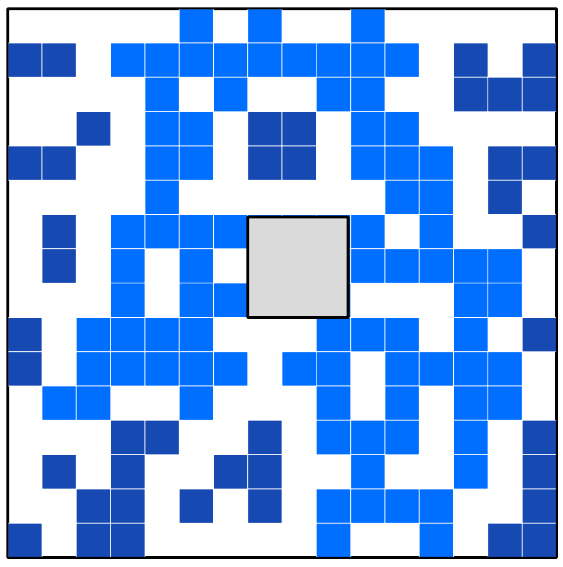}};
		\node at (.25,.25) {$A$};
		\node at (.25,.25 -5.2) {$|$};
		\node at (.25,.25 +4.63) {$|$};
		\node at (.25 + 4.7,.25 ) {$=$};
		\node at (.25 -5.1 ,.25) {$=$};

%		\draw[|-|] (-5,-2.8)--(5,-2.8);
%		\draw[->] (0,-2.8)--(0,3);
%		\shade[ball color = red!40, opacity = 1] (4.525,2.7) circle (.25cm);
%		\shade[ball color = red!40, opacity = 1] (-4.525,2.7) circle (.25cm);
%		\draw [->] (4.525,2.4) to [out=-95,in=85] (4.475,1.7);
%		\draw [->] (-4.525,2.4) to [out=-85,in=95] (-4.475,1.7);
		\end{tikzpicture}
		\subcaption{$\mathcal E_{m,A}^{\theta}$}
	\end{subfigure}
	\caption{(a) The event $\mathcal E_{m,A}$ (on $\Lambda_m$) requires that the good blocks (blue) of the coarse graining connect $\partial A^{(k)}$ to $\partial \Lambda_{m-k}^{(k)}$ and that $\omega$ itself connects that coarse-grained component to the boundary $\partial \Lambda_m$ (black line). (b) The event $\mathcal E_{m,A}^{\theta}$ (on the torus) requires that there is one large coarse-grained component (blue) having positive density at least $\theta$ and connecting $\partial A^{(k)}$ to $\partial \Lambda_{m-k}^{(k)}$, and all $k$-bad components (white) are not too large.}
	\label{fig:E-events}
\end{figure}
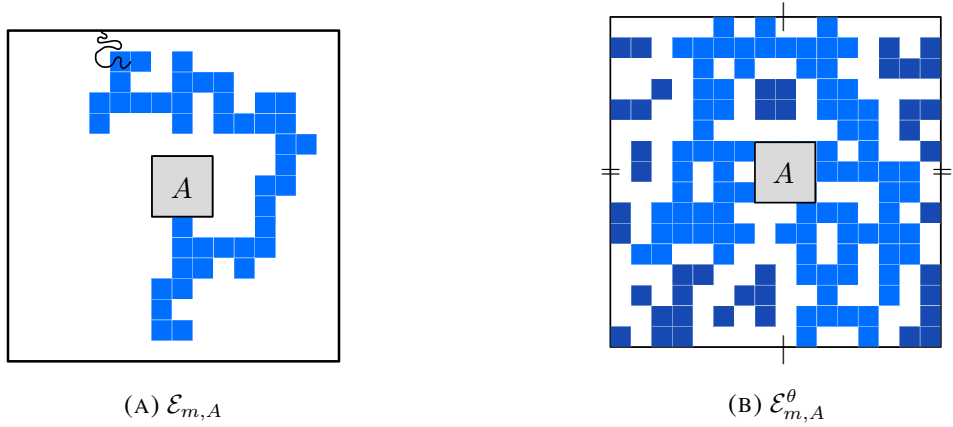

\begin{definition}\label{def:E-m-A}
For $A\subset \Lambda_{m/2}$, define $\cE_{m,A}$ as the set of $\omega$ such that the $k$-good coarse-graining $\eta^{(k)}(\omega)$ has a $k$-open path connecting $\partial A^{(k)}$ to $\partial \Lambda_{m-k}^{(k)}$, and the largest open cluster in that path is connected to $\partial \Lambda_m$ in $\omega$. 
\end{definition}

While the above event avoided discussing the coarse-graining within distance $k$ of the boundary $\partial \Lambda_m$, the next event is to be viewed for coarse-grainings of the torus $\Lambda_m^p$; therefore we note that the notions of cluster and adjacency of blocks are to be viewed as being on the coarse-graining $\mathbb T^{(k)}_m$. 

\begin{definition}\label{def:E-m-p}
Fix $\theta$. For $A \subset \Lambda_{m/2}$, let $\cE_{m,A}^{\theta}$ be the set of $\omega$ such that
\begin{enumerate}
    \item The largest $\star$-cluster of $k$-bad blocks has size at most $m/4k$;
    % There is at most one cluster of $k$-good blocks of $\omega$ having size greater than $m/4k$. 
    \item The largest cluster of $k$-good blocks has size at least $\theta |\Lambda_m^{(k)}|$, and contains a $k$-good path connecting $\partial A^{(k)}$ to $\partial \Lambda_{m-k}^{(k)}$. 
\end{enumerate}
We also write $\cE_{m}^\theta$ to denote $\cE_{m,A}^{\theta}$ with the last requirement of a $k$-good path connecting $\partial A^{(k)}$ to $\partial \Lambda_m^{(k)}$ omitted.  Thus for all $A$, we have $\cE_{m,A}^\theta\subset \cE_{m}^\theta$. 
\end{definition}

\subsubsection{Bounding the probability of the $\cE_{m,A},\cE_{m,A}^{\theta}$ events} 
The goal of this subsection is to prove the following pair of results for $k$ sufficiently large. 
\begin{proposition}\label{prop:E-m-probability}
Let $d\ge 2$ and $p>p_c(d)$, and $v= (0,...,0)$. As long as $k$ is sufficiently large, then, uniformly over all boundary conditions $\xi$, for all $r\le m$, we have  
\begin{align*}
    \pi_{\Lambda_m^\xi}^\rc \big((\cE_{m,B_{r/2}(v)})^c\big) \le Ce^{  - r/C}\,.
\end{align*}
\end{proposition}

\begin{proposition}\label{prop:E-m-theta-probability}
Let $d\ge 2$ and $p>p_c(d)$, and $v= (0,...,0)$. For all $k$ sufficiently large, there exists $\theta_0(k,p,d)$ such that for every $\theta<\theta_0$, for all $r\le m$, 
\begin{align*}
    \pi_{\Lambda_m^p}^\rc\big((\cE_{m,B_{r/2}(v)}^{\theta})^c\big) \le Ce^{ - r/C}\,. 
\end{align*}
\end{proposition}

\begin{proof}[\textbf{\emph{Proof of Proposition~\ref{prop:E-m-probability}}}]
Fix $v,r$ and for ease of notation, let $B= B_{r/2}(v)$. Consider the $k$-good coarse-graining $\eta^{(k)}(\omega)$, and recall from Corollary~\ref{cor:stoch-domination} that it stochastically dominates the $\mbox{Ber}(1-Ce^{ - k/C})$ independent percolation  process $\tilde \eta$ on $\Lambda_{m-2k}^{(k)}$. Let us consider the event $\tilde \cE_{m,B}$ that a coarse-grained percolation~$\eta$ has an open component $\cC_1(\eta)$ connecting $\partial B^{(k)}$ to $\partial \Lambda_{m-2k}^{(k)}$ and moreover, that component satisfies
\begin{align*}
    \cC_1(\eta) \cap \partial \Lambda_{m-2k}^{(k)} \ge |\partial \Lambda_{m-2k}^{(k)}|/2\,.
\end{align*}
Observe that $\tilde \cE_{m,B}$ is increasing, and the event $\{\eta^{(k)}(\omega)\in \tilde \cE_{m,B}\}$ is measurable with respect to $\omega(\Lambda_{m-k})$. By its increasing nature and the above stochastic domination, 
\begin{align*}
    \pi_{\Lambda_m^{\xi}}^{\rc}(\eta^{(k)}(\omega)\in \tilde \cE_{m,B}^c) \le \mathbb P(\tilde \eta \in \tilde \cE_{m,B}^c)\,.
\end{align*}

Then by standard facts about Bernoulli percolation, $\widetilde\eta$ has this property with probability $1-Ce^{ - r^{d-1}/C}$ as long as $k$ is sufficiently large, so that $\widetilde \eta$ is sufficiently super-critical. (This can be seen by looking at the process $\one - \tilde \eta$, and noticing that in order for the complement of the above event to occur, either $\one - \tilde \eta$ needs a $\star$-component of size at least proportional to $|\partial B^{(k)}|$, or $\one - \tilde \eta$ needs to be open on $|\partial \Lambda_m-{2k}^{(k)}|/2$ of the sites of $\partial \Lambda_{m-2k}^{(k)}$. For $k$ large, the first of these has probability $Ce^{ - r^{d-1}/C}$ and the second of these has probability $Ce^{- m^{d-1}/C}$.) Thus we have 
\begin{align}\label{eqn:ajs5}
    \pi_{\Lambda_m^\xi}^\rc\big(\eta^{(k)}(\omega) \in \cE_{m,B}^c\big) \le Ce^{ - r^{d-1}/C} +  \max_{\omega(\Lambda_{m-k}) \in \tilde\cE_{m,B}} \pi^\rc_{\Lambda_m^{\xi}}\big(\cE_{m,B}^c \mid \omega(\Lambda_{m-k})\big)\,.
\end{align}
Now fix any configuration $\omega(\Lambda_{m-k}) \in \tilde \cE_{m,B}$. We consider the probability that there is a connection in $\omega$ between the largest component of the $k$-good path from $\partial B^{(k)}$ to $\partial \Lambda_{m-2k}^{(k)}$ and the boundary $\partial \Lambda_m$. Notice that since $\omega(\Lambda_{m-k})\in \tilde\cE_{m,B}$, there are at least $|\partial \Lambda_{m-2k}^{(k)}|/(4d)$ disjoint blocks $(B_{x_i})$ for $x_i \in \partial \Lambda_{m-2k}^{k}$ that are $k$-good in $\eta^{(k)}(\omega)$ and are part of $\cC_1 (\eta^{(k)}(\omega))$. For each of these blocks $B_{x_i}$, there is a vertex $v_i \in \partial \Lambda_{m-k}$ contained in the largest component of $\omega(B_{x_i})$.

Now notice that there is a collection $(\gamma_i)$ of edge-disjoint paths connecting $v_i$ to $\partial \Lambda_m$, and each of these paths has length $k$. The probability that a fixed such path is open is at least $c^k$ for $c(p)>0$, independently of the configuration on $E(\Lambda_m)\setminus \gamma_i$, so the probability that none of them is open gives the bound: 
\begin{align*}
    \max_{\omega(\Lambda_{m-k})\in \tilde\cE_{m,B}} \pi^\rc_{\Lambda_m^{\xi}} (\cE_{m,B}^c \mid \omega(\Lambda_{m-k})) \le \mathbb P\big(\mbox{Bin}(|\partial \Lambda_{m-k}^{(k)}|/(4d), c^k) = 0\big) \le Ce^{ - m^{d-1}/C}\,.
\end{align*}
Plugging this bound into~\eqref{eqn:ajs5} then gives the desired bound on the probability of $\cE_{m,B}^c$. 
\end{proof}

\begin{proof}[\textbf{\emph{Proof of Proposition~\ref{prop:E-m-theta-probability}}}]
Fix $v,r$ and let $B = B_{r/2}(v)$. Fix $\theta<\theta_0(k,p,d)$ sufficiently small, to be chosen later. For ease of notation, let $\mathcal E_m = \mathcal E_{m,B}^\theta$ and let $\cE_{m} = \cE_{m,1}\cap \cE_{m,2}$ where the index $1$ or $2$ corresponds to the two different items in Definition~\ref{def:E-m-p}. Clearly we have
\begin{align*}
    \pi_{\Lambda_m^p}^{\rc} (\mathcal E_m^c) \le \pi_{\Lambda_m^p}^{\rc}(\cE_{m,1}^c) + \pi_{\Lambda_m^p}^{\rc}(\cE_{m,2}^c)\,.
\end{align*}

Let us begin by bounding the probability of $\cE_{m,1}^c$. By Corollary~\ref{cor:stoch-domination}, $\one - \eta^{(k)}(\omega) \preceq \one - \widetilde \eta$ on $\mathbb T_m = \Lambda_m^p$, where $\widetilde \eta$ is a $\mbox{Ber}(1-Ce^{ - k/C})$ percolation process on $\mathbb T_m^{(k)}$. The probability that $\one - \eta^{(k)}(\omega)$ has a $k$-$\star$-cluster (of $k$-bad blocks) of size at least $m/4k$ is  less than that of $\one - \widetilde \eta$ having a $k$-$\star$-cluster; but the percolation parameter of $\one - \widetilde \eta$ is $Ce^{- k/C}$, which when $k$ is large enough is sub-critical on $\Lambda_{m-k}^{(k)}$ endowed with $\star$-adjacency. In particular, this has probability at most $C\exp( - m/C)$. 

Next we turn to bounding the probability of $\cE_{m,2}^c$. The probability of the giant component not connecting $\partial B^{(k)}$ to $\partial \Lambda_{m-k}^{(k)}$ is bounded as in the proof of Proposition~\ref{prop:E-m-probability} by $Ce^{ - r^{d-1}/C}$. It remains to bound the probability of that giant component not having density at least $\theta$ in $\Lambda_m^{(k)}$. 
By standard properties of super-critical Bernoulli percolation, there exists $\theta_0(k)>0$ (going to $1$ as $k\to \infty$) such that with probability $1-Ce^{ - m^{d-1}/C}$, as long as $k$ is sufficiently large, the largest component of the process $\widetilde \eta$ has a connected component of size at least $\theta_0 |\Lambda_m^{(k)}|$. This also implies the desired property
for $\eta^{(k)}(\omega)$ by Corollary~\ref{cor:stoch-domination}. 
\end{proof}

\subsection{Edwards--Sokal representation of the Ising measure in the plus phase}\label{subsec:edwards-sokal-pi-hat}
The proof of Theorem~\ref{thm:Ising-wsm-within-phase-intro} will couple
$\widehat \pi$ to $\pi_{B_{r}^{\plus}(v)}$ via an intermediate distribution $\widetilde \pi^\es$ defined as follows. This latter measure's coloring stage will be more tractable to couple to that of $\pi_{B_r^{\plus}(v)}^\es$ later. 

\begin{definition}
Let $\widetilde \pi^\es_{\Lambda_n^p}$ be the joint distribution $(\widetilde \omega, \widetilde \sigma)$, obtained by 
\begin{enumerate}
    \item Sampling $\widetilde \omega \sim \pi_{\Lambda_n^p}^{\rc}$.
    \item Sampling $\widetilde \sigma$ by coloring the largest connected component of $\widetilde \omega$ with $+1$ (breaking ties according to some fixed ordering over all connected components), and coloring all other components independently and  uniformly over $\{\pm 1\}$. 
\end{enumerate}
\end{definition}

\begin{remark}\label{rem:pi-hat-edwards-sokal}
We can also extend the distribution $\widehat \pi$ to a joint distribution over Ising and random-cluster configurations. When $|\Lambda_n|$ is odd, it is clear that $\widehat \pi$ can be sampled from by first drawing a sample $\widehat \omega\sim \pi_{\Lambda_n^p}^\rc$ then coloring its connected components independently, uniformly over $\{\pm 1\}$, conditional on the magnetization being positive (because, irrespective of $\widehat \omega$, this has probability exactly $\frac 12$). When $|\Lambda_n|$ is even, we claim that the marginal of this distribution is within $Ce^{ - n^{d-1}/C}$ of $\widehat \pi$: indeed this is evident because it gives the correct relative weights to all configurations with non-zero magnetization, and the entire mass of zero magnetization (i.e., $\widehat \Omega \cap \widecheck \Omega$) is at most $Ce^{ - n^{d-1}/C}$ by~\eqref{eq:surface-order-LDP}. As this error is negligible everywhere, we abuse notation and simply call $\widehat \pi^\es_{\Lambda_n^p}$ this joint distribution. 
\end{remark}

The main result of this subsection will be the following. 

\begin{proposition}\label{prop:tilde-hat-comparison}
Fix $d\ge 2$ and let $\beta>\beta_c(d)$. Then, 
\begin{align*}
    \|\widetilde \pi_{\Lambda_n^p} - \widehat \pi_{\Lambda_n^p}\|_\tv \le Ce^{ - n/C}\,.
\end{align*}
\end{proposition}

The main ingredient in the proof of Proposition~\ref{prop:tilde-hat-comparison} is the following coupling of their coloring stages when the random-cluster draws are in $\cE_{n}^{\theta}$. Recall that $\cE_n^{\theta}$ is the event of Definition~\ref{def:E-m-p} excluding the requirement of a $k$-good path connecting $\partial A^{(k)}$ to $\partial \Lambda_{n-k}^{(k)}$.

\begin{lemma}\label{lem:tilde-hat-ising-coupling}
Consider any configuration $\omega$ in $\cE_{n}^\theta$. Then 
\begin{align*}
    \|\widetilde \pi_{\Lambda_n^p}^\es (\widetilde \sigma \in \cdot \mid \widetilde \omega = \omega) - \widehat\pi_{\Lambda_n^p}^\es (\widehat \sigma \in \cdot \mid \widehat \omega = \omega)\|_\tv \le Ce^{ - n/C}\,.
\end{align*}
\end{lemma}
\begin{proof}
Fix any $\omega \in \cE_{n}^\theta$ and consider any event $A$ for the Ising configuration. Let $\widetilde \Omega = \widetilde \Omega(\omega)$ be the event that the largest component of $\omega$ is colored $+1$; call $\widehat \Omega(\omega)$ the event that the magnetization after the coloring step is (strictly) positive.  Consider
\begin{align*}
    \widetilde \pi^\es_{\Lambda_n^p} (A \mid \omega) - \widehat\pi^\es_{\Lambda_n^p} (A \mid \omega)  & =  \frac{\pi^\es_{\Lambda_n^p} (A, \widetilde \Omega(\omega) \mid \omega)}{\pi^\es_{\Lambda_n^p}(\widetilde \Omega(\omega)\mid \omega)} - \frac{\pi^\es_{\Lambda_n^p} (A , \widehat \Omega(\omega) \mid \omega^p)}{\pi_{\Lambda_n^p}^\es(\widehat \Omega(\omega) \mid \omega)} \\ 
    &  \le 2 \big(\pi_{\Lambda_n^p}^\es(A, \widetilde \Omega(\omega) \mid \omega) - \pi_{\Lambda_n^p}^\es (A , \widehat \Omega(\omega) \mid \omega)\big)  + \big| \pi_{\Lambda_n^p}^\es(\widehat \Omega(\omega) \mid \omega) - \frac{1}{2}\big|\,.
\end{align*}
The difference in the two probabilities can be bounded as
\begin{align*}
    \pi^\es_{\Lambda_n^p}(\widetilde \Omega^c(\omega) ,\widehat \Omega(\omega) \mid \omega) + \pi^\es_{\Lambda_n^p}(\widetilde \Omega(\omega),\widehat \Omega^c(\omega) \mid \omega) \le \pi^\es_{\Lambda_n^p}(\widehat \Omega(\omega) \mid \widetilde \Omega^c(\omega),\omega) + \pi^\es_{\Lambda_n^p}(\widehat \Omega^c(\omega) \mid \widetilde \Omega(\omega),\omega)\,.
\end{align*}
The difference $\pi_{\Lambda_n^p}^\es(\widehat \Omega(\omega) \mid \omega) - \frac 12$ is seen to also be bounded by 
\begin{align*}
    \big|\pi^\es_{\Lambda_n^p}(\widehat \Omega(\omega) \mid \omega) - \pi_{\Lambda_n^p}(\widetilde \Omega(\omega) \mid \omega)\big| \le \pi^\es_{\Lambda_n^p}(\widetilde \Omega^c(\omega) ,\widehat \Omega(\omega) \mid \omega) + \pi^\es_{\Lambda_n^p}(\widetilde \Omega(\omega),\widehat \Omega^c(\omega) \mid \omega)\,.
\end{align*}
Thus it suffices to control these two quantities. They are shown to be small by identical arguments so let us just consider the latter, where we are conditioning on the largest component in $\omega$ being colored $+1$, and considering the probability on that event that the magnetization is non-positive. Let $\cC_1,\cC_2,...$ be the components of $\omega$ ordered in decreasing size, so that because $\omega\in \cE_{n,B_{r/2}(v)}^\theta$, necessarily $|\cC_1|\ge \theta k^{1-d} |\Lambda_n|$, and $|\cC_j|\le k^{d-1} n$ for all $j\ne 1$. Here, the additional factors of $k^{1-d}$ and $k^d$ come from the facts that the minimum volume of the largest component of $\omega(B_x)$ for a block $B_x$ is $k$, while the maximum volume of any component of $\omega(B_x)$ is $k^d$. The bound on the second and subsequent largest components in $\omega$ follows because all such components either have size at most $k$ or are confined to be fully contained in a cluster of $k$-bad blocks, which on $\mathcal E_{n,B_{r/2}(v)}^\theta$ has size at most $n/4k$.

Then since these clusters are colored independently, the probability that the magnetization is non-positive is at most
\begin{align*}
    \pi_{\Lambda_n^p}^\es \big(M(\sigma)\le 0 \mid \widetilde \Omega(\omega), \omega\big)\le \pi_{\Lambda_n^p}^\es\Big(M\Big(\sigma\Big(\bigcup_{j\ge 2} \cC_j\Big)\Big) < - \theta k^{1-d} |\Lambda_n| \,\, \Big\vert \,\, \omega\Big)\,.
\end{align*}
Notice that this magnetization can be expressed as the sum of independent random variables $Z_j$ that take values $\pm |\cC_j|$ with probability $\frac 12$-$\frac 12$. By Hoeffding's inequality, this is at most
\begin{align*}
    \mathbb P\Big(\sum_j Z_j < -\theta k^{1-d} |\Lambda_n|/2\Big) \le \exp \Big( - \frac{\theta^2 k^{2-2d} |\Lambda_n|^2}{n^d |\cC_2|}\Big) \le C\exp( - \theta^2 n^{d-1}/C)\,,
\end{align*}
which for $d\ge 2$ is at most $C\exp( - n/C)$. 
\end{proof}

\begin{proof}[\textbf{\emph{Proof of Proposition~\ref{prop:tilde-hat-comparison}}}]
By using the identity coupling for the random-cluster draws from $\pi_{\Lambda_n^p}^\rc$, one easily obtains the bound 
\begin{align*}
    \|\widetilde \pi_{\Lambda_n^p} - \widehat \pi_{\Lambda_n^p}\|_\tv \le \pi_{\Lambda_n^p}^{\rc}\big((\cE_{n}^\theta)^c\big) + \max_{\omega\in \cE_{n}^\theta}     \|\widetilde \pi_{\Lambda_n^p}^\es (\widetilde \sigma \in \cdot \mid \widetilde \omega = \omega) - \widehat\pi_{\Lambda_n^p}^\es (\widehat \sigma \in \cdot \mid \widehat \omega = \omega)\|_\tv\,.
\end{align*}
The desired bound then follows from Proposition~\ref{prop:E-m-theta-probability} together with Lemma~\ref{lem:tilde-hat-ising-coupling}. 
\end{proof}

\subsection{Proof of WSM within a phase}\label{subsec:proof-of-wsm}
We prove Theorem~\ref{thm:Ising-wsm-within-phase-intro}
using a combination of the above properties for the good and very good coarse-grained percolation processes for $\omega \sim \pi_{B_r^\one(v)}^\rc$ and $\omega' \sim \pi_{\Lambda_n^p}^{\rc}$. Recall the event $\mathcal E_{\textsc{vg}}$ and abusing notation slightly, say $(\omega,\omega')\in \cE_{\textsc{vg}}$ if $(\omega(B_r(v)), \omega'(B_r(v))$ have a $k$-very good separating surface in the annulus $B_{r}(v)\setminus B_{r/2}(v)$. Also recall the events from Definitions~\ref{def:E-m-A}--\ref{def:E-m-p}.

\begin{lemma}\label{lem:coupling-Ising-configurations}
Consider a pair of configurations $(\omega,\omega')$ satisfying $(\cE_{r,B_{r/2}(v)} \times \cE_{n,B_{r/2}(v)}^{\theta}) \cap \cE_{\textsc{vg}}$; let $\Gamma_k$ denote the outermost open separating $k$-surface of very good blocks $\zeta^{(k)}(\omega,\omega')$ in $B_{r}(v)\setminus B_{r/2}(v)$. If  $\omega(\Int(\Gamma)) = \omega'(\Int(\Gamma))$, then there exists a coupling of $$\sigma^{(r)}\sim \pi_{B_r^{\plus}(v)}^{\es} ( \cdot \mid \omega^{(r)} =\omega)\qquad \mbox{and}\qquad \sigma^{(n)}\sim \widetilde \pi_{\Lambda_n^p}^\es(\cdot \mid \widetilde \omega^{(n)} = \omega')$$ such that $\sigma^{(r)}_{B_{r/2}(v)} = \sigma^{(n)}_{B_{r/2}(v)}$ with probability one. 
\end{lemma}

\begin{proof}
We first show that under the conditions of the lemma, the two configurations $\omega,\omega'$ satisfy 
\begin{enumerate}
    \item Their induced component structures agree on $B_{r/2}(v)$, and 
    \item A component in $\omega(B_{r/2}(v))$ is connected to $\partial B_{r}(v)$ in $\omega$, if and only if that component is a subset of the largest component of $\omega'$. 
\end{enumerate}
The fact that their induced component structures agree on $B_{r/2}(v)$ is immediate from the facts that $\Int(\Gamma)\supset B_{r/2}(v)$ on the event $\cE_{\textsc{vg}}$, and Lemma~\ref{lem:separating-surface-disconnects-information}. 

For item (2), it suffices for us to show that on one hand, the largest component of $\omega(\Gamma)$ is connected to $\partial B_r(v)$ in $\omega$, and on the other hand, the largest component of $\omega'(\Gamma)$ is a subset of the largest component of $\omega'$. This is enough since all components of $\omega(\Gamma)$ or $\omega'(\Gamma)$ besides the largest one have size at most $k$, and since all of $\Gamma$ is very good, the largest components of $\omega(\Gamma)$ and $\omega'(\Gamma)$ coincide. 

Consider first the union of $\Gamma$ and the portion of the path of $k$-good blocks connecting $\partial B_{r/2}(v)$ to $\partial B_{r}(v)$ exterior to $\Gamma$ in $\omega$. This must be a $k$-good connected component by definition of $\Gamma_k$ being a $k$-separating surface. Since this includes a block whose side coincides with $\partial B_r(v)$, the large component (of size at least $k$) of $\omega(\Gamma)$ is connected to $\partial B_{r}(v)$ in $\omega$. 
Next, by Definition~\ref{def:E-m-p}, on $\cE_{n,B_{r/2}(v)}^\theta$, the largest component of $\omega'$ coincides with its largest component in the path connecting $\partial B_{r/2}^{(k)}(v)$ to $\partial \Lambda_{n}^{(k)}$. Reasoning as above, this implies that the largest component of $\omega'$ also coincides with the largest component of $\omega'(\Gamma)$, as claimed.  

We now construct a coupling of the corresponding Ising configurations given a pair $(\omega,\omega')$ satisfying (1)--(2) above. Recall that via the Edwards--Sokal coupling, $\pi^\es_{B_r^{\plus}(v)}$ is obtained by drawing a sample $\omega^{(r)} \sim \pi^\rc_{B_r^\one(v)}$, setting the state of the component of the boundary to be $+1$, and coloring all other components independently uniformly among $\{\pm 1\}$. The measure $\widetilde \pi^\es_{\Lambda_n^p}$ is similarly obtained by drawing a sample $\widetilde \omega^{(n)} \sim \pi^\rc_{\Lambda_n^\one}$ and setting its largest component to be $+1$ (and coloring all other components independently uniformly among $\{\pm 1\}$). By item (2), we have 
\begin{align*}
    \cC_{\partial B_r(v)}(\omega) \cap B_{r/2}(v)  = \cC_{1}(\omega') \cap B_{r/2}(v)\,,
\end{align*}
where $\cC_1(\omega')$ is the largest component of $\omega'$. 
Those sites are all colored $+1$ in both~$\sigma^{(r)}$ and~$\sigma^{(n)}$. Consider the remaining vertices, $B_{r/2}(v)\setminus  \cC_{\partial B_r(v)}(\omega)$. Those vertices have the same induced component structure given the configurations $\omega(B_{r}(v)\setminus B_{r/2}(v))$ and $\omega'(\Lambda_n\setminus B_{r/2}(v))$. Therefore, we can assign colors to the clusters of $\omega,\omega'$ in the following way. 
\begin{enumerate}
    \item Fix an enumeration of all the vertices in $\Lambda_n$, but such that the vertices in $B_{r/2}(v)$ are enumerated before any vertices in $\Lambda_n \setminus B_{r/2}(v)$ are. 
    \item For each $v\in \Lambda_n$, let $s_v$ be an i.i.d.\ uniform $\{\pm 1\}$ spin.
    \item For every cluster $\cC$ of $\omega$ (disjoint from $\cC_{\partial B_r(v)}$), for all $w\in \cC$, let $\sigma^{(r)}_w = s_{v_{\min}(\cC)}$ where $v_{\min}(\cC)$ is the smallest (in the enumeration) vertex in $\cC$. 
    \item For every cluster $\cC'$ of $\omega'$ (disjoint from $\cC_1(\omega')$), for all $w\in \cC'$, let $\sigma^{(n)}_w = s_{v_{\min}(\cC')}$ where $v_{\min}(\cC')$ is the smallest (in the enumeration) vertex in $\cC'$. 
\end{enumerate}
Since every cluster that intersects $B_{r/2}(v)$ in either of $\omega,\omega'$ has its smallest vertex in $B_r(v)$, the color assignments will be coupled such that all vertices in $B_{r/2}(v)$ get the same colors under $\sigma^{(r)}$ and $\sigma^{(n)}$. 
\end{proof}

Finally, we are able to prove Theorem~\ref{thm:Ising-wsm-within-phase-intro}
from the introduction, which claims that the Ising model satisfies WSM
within a phase at all low temperatures.  We recall that, combined with the
results of Section~\ref{sec:relaxation-within-phase}, this concludes the proof of our main result, Theorem~\ref{thm:Ising-main}.

\begin{proof}[{\textbf{\emph{Proof of Theorem~\ref{thm:Ising-wsm-within-phase-intro}}}}]
We will construct couplings between the two Ising distributions, $\pi_{B_{r}^{\plus}(v)}$ and $\widetilde \pi_{\Lambda_n^p}$ such that the two agree on $B_{r/2}(v)$ except with probability $Ce^{ - r/C}$. This suffices by a triangle inequality together with Lemma~\ref{lem:tilde-hat-ising-coupling} and the discussion preceding it in Remark~\ref{rem:pi-hat-edwards-sokal}. 

We can upper bound, under any coupling of $\omega\sim \pi_{B_r^\one(v)}^\rc$ and $\omega'\sim \pi_{\Lambda_n^p}^\rc$, 
% from Lemma~\ref{lem:coupling-Ising-configurations} (more precisely, first exposing $\omega'(\Lambda_n^p \setminus B_r(v))$ then using that coupling on $B_r(v)$),  
\begin{align*}
    \|\pi_{B_r^{\plus}(v)} & (\sigma (B_{r/2}(v))\in \cdot)- \widetilde \pi_{\Lambda_n^p}(\sigma_{B_{r/2}(v)} \in \cdot) \|_\tv \\
    & \le \mathbb P \big((\omega,\omega') \notin (\cE_{r,B_{r/2}(v)} \times \cE_{n,B_{r/2}(v)}^\theta) \cap \cE_{\textsc{vg}}\big)   \\ 
    &  \quad + \mathbb E \big[\|\pi_{B_r^{\plus}(v)}^\es(\sigma_{B_{r/2}(v)}\in \cdot\mid \omega) - \widetilde \pi_{\Lambda_n^p}(\sigma_{B_{r/2}(v)}\in \cdot \mid \omega')\|_\tv \mid (\cE_{r,B_{r/2}(v)} \times \cE_{n,B_{r/2}(v)}^\theta) \cap \cE_{\textsc{vg}} \big]\,.
\end{align*}
Using the monotone coupling given to us by Lemma~\ref{lem:E-very-good-probability}, on the event $\cE_{\textsc{vg}}$ we have $\omega(\Int(\Gamma)) = \omega'(\Int(\Gamma))$ and we can apply the coupling of Ising configurations given by Lemma~\ref{lem:coupling-Ising-configurations} to see that the second term on the right-hand side above is zero. It therefore suffices for us to bound the probability of $(\cE_{r,B_{r/2}(v)} \times \cE_{n,B_{r/2}(v)}^\theta) \cap \cE_{\textsc{vg}})^c$ under the coupling of Lemma~\ref{lem:E-very-good-probability}. This bound follows a union bound  combining Proposition~\ref{prop:E-m-probability} applied at $m = r$, Proposition~\ref{prop:E-m-theta-probability} applied at $m=n$, and Lemma~\ref{lem:E-very-good-probability}. 
\end{proof}

\section{Applications to other geometries: random regular graphs}\label{sec:random-graphs}
In this section, we demonstrate how the proof approach of the paper can be applied to more general geometries to establish fast mixing of the low-temperature Ising dynamics from a random initialization. We demonstrate this generalization in the well-studied setting of random regular graphs, and prove Theorem~\ref{thm:random-graph-mixing} stated in the introduction.  

% \subsection{The Ising model on random regular graphs}
Throughout this section, let $\Delta \ge 3$ and let $\Prrg$ be the uniform distribution over random $\Delta$-regular graphs on $N$ vertices. 
In this section, we will use $\cG$ to denote a $\Delta$-regular graph drawn uniformly according to $\Prrg$. Overwriting earlier notation, in this section we will use $d := \Delta -1$ as the branching factor of the graph, $\pi := \pi_{\cG}$ to denote the Ising measure on $\cG$ and the notation $X_{t}^{x_0}$ to denote the Ising Glauber dynamics on $\cG$ with initial configuration~$x_0$. Furthermore, we use $\widehat \Omega$ and $\widecheck \Omega$ to denote the configurations with non-negative and non-positive magnetizations respectively, $\widehat \pi = \widehat \pi_{\cG}$ and $\widecheck \pi = \widecheck \pi_{\cG}$ the corresponding conditional measures, and  $\widehat X_t^{x_0}, \widecheck X_t^{x_0}$ the corresponding restricted Ising Glauber dynamics.

\subsection{Proof ingredients}
Towards proving Theorem~\ref{thm:random-graph-mixing}, we need to generalize the notion of WSM within a phase to graphs beyond $(\mathbb Z/n\mathbb Z)^d$. For one thing, we need to restrict the radii we consider, which in Definition~\ref{def:wsm-within-a-phase} were at most $n = N^{1/d}$, so that they are smaller than the diameter of the underlying graph: in the context of a random $\Delta$-regular graph, we therefore only consider spatial mixing at scales $r$ smaller than (say) half the typical diameter $\log_d N$. The other modification we make is to assume a temperature dependence in the rate of WSM within a phase: this is because, in random $\Delta$-regular graphs, the competing volume term grows exponentially with~$r$, so the temperature dependence is needed to ensure that, at low enough temperatures, the spatial mixing decay dominates the volume growth. Naturally, this means that our results on random $\Delta$-regular graphs hold only at low enough temperatures, not all the way up to the critical point as in $(\mathbb Z/n\mathbb Z)^d$. 

Recall that $B_r^{\plus}(v)$ denotes the $\ell_\infty$-ball of radius $r$ about the vertex $v$, with $\plusone$ boundary condition on $\partial B_r(v) = \{w\in V(\cG)\setminus B_r(v): d(w,B_r(v)) = 1\}$.  

\begin{definition}\label{def:wsm-random-graph}
For a graph $\cG$ of maximal degree $\Delta = d+1$ on $N$ vertices, we say the Ising model has \emph{WSM within a phase at inverse temperatures $\beta>\beta_0$} if, for every $r<(1/2) \log_d N$ and every $v\in V(\cG)$,
\begin{align*}
    \|\pi_{B_r^{\plus}(v)} (\sigma_v\in \cdot) - \widehat \pi(\sigma_v\in \cdot) \|_\tv \le Ce^{ - \beta r/C}\,,
\end{align*}
for some constant $C= C(\Delta)>0$ (independent of $\beta$). 
\end{definition}

Given Definition~\ref{def:wsm-random-graph}, we will reduce the proof of  Theorem~\ref{thm:random-graph-mixing} to the following two lemmas, each an analogue of one earlier in the paper for the Ising model on $(\mathbb Z/n\mathbb Z)^d$. The first is the analogue of Theorem~\ref{thm:Ising-wsm-within-phase-intro} and establishes the WSM within a phase property at sufficiently low temperatures. 

\begin{lemma}\label{lem:wsm-within-phase-random-graph}
Fix $\Delta\ge 3$ and suppose $\beta_0$ is sufficiently large. With probability $1-o(1)$, $\cG\sim \Prrg$ is such that the Ising model on $\cG$ satisfies WSM within a phase at inverse temperatures $\beta>\beta_0$. 
\end{lemma}

The other lemma replaces the assumption~\eqref{eq:assumption-scale-m} on the rate of convergence to stationarity for Glauber dynamics at local scales with $\plusone$ boundary condition. Unlike in Section~\ref{sec:relaxation-within-phase}, here we can take an optimal, exponential rate of relaxation for the local-scale Markov chain with $\plusone$ boundary condition due to available sharp bounds on the log-Sobolev constant of the Ising dynamics on $\Delta$-regular trees with $\plusone$ boundary condition from~\cite{MSW-trees-bc}, combined with the tree-like nature of local balls in a random $\Delta$-regular graph. 

\begin{lemma}\label{lem:local-mixing-random-graph}
With probability $1-o(1)$, $\cG \sim \Prrg$ is such that, for every $\beta>0$, there exists $c_\star(\beta,\Delta)>0$ such that  for every $v\in V(\cG)$ and every $r\le (1/3)\log_d N$, for all $t>0$,
\begin{align*}
   \|\mathbb P(X_{B_r^{\plus}(v),t}^{\plus} (v) \in \cdot) - \pi_{B_r^{\plus}(v)} (\sigma_v\in \cdot) \|_\tv \le Cd^r e^{ - c_\star t}\,.
\end{align*}
\end{lemma}
% 
% on the mixing times of boxes at local scales with $\plusone$ boundary conditions. This input is essentially given to us by the sharp bound on the mixing time of the Ising dynamics on $\Delta$-regular trees with $\plusone$ boundary conditions from~\cite{MSW-trees-bc}, together with the treelike nature of local balls in a random $\Delta$-regular graph. 

\subsection{Proof of Theorem~\ref{thm:random-graph-mixing} given Lemmas~\ref{lem:wsm-within-phase-random-graph}--\ref{lem:local-mixing-random-graph}}
We will for now defer the proofs of Lemmas~\ref{lem:wsm-within-phase-random-graph} and~\ref{lem:local-mixing-random-graph} to conclude the proof of Theorem~\ref{thm:random-graph-mixing}; this will go via the following proposition, which we view as an analogue of Theorems~\ref{thm:Ising-torus-restatement} and~\ref{thm:Ising-mixing-restricted-chain}. This organization is chosen to demonstrate that the proof of Theorem~\ref{thm:random-graph-mixing} is essentially the same as that of Theorem~\ref{thm:Ising-main}, once one has the appropriate analogous WSM within a phase property and local mixing time bound.  

\begin{proposition}\label{prop:mixing-rate-random-graph}
With probability $1-o(1)$, $\cG \sim \Prrg$ is such that, for all $\beta$ sufficiently large, the Ising Glauber dynamics on $\cG$ restricted to the plus phase satisfies 
\begin{align}\label{eq:mixing-rate-restricted-random-graph}
        \|\mathbb P(\widehat X_{t}^{\plus} \in \cdot) - \widehat \pi\|_\tv \le C N e^{ - (t\wedge \beta \log N)/C}\,,
\end{align}
and the (un-restricted) Ising Glauber dynamics on $\cG$ satisfies 
\begin{align}\label{eq:mixing-rate-unrestricted-random-graph}
    \|\mathbb P(X_{t}^{\nu^{\pmb{\pm}}} \in \cdot) - \pi\|_\tv \le C N e^{ - (t\wedge \beta \log N)/C}\,.
\end{align}
\end{proposition}

\begin{remark}
Notice that this exponential decay fits directly into the framework of proofs from Section~\ref{sec:relaxation-within-phase} as follows. Lemma~\ref{lem:local-mixing-random-graph} ensures that we can take $f(m)$ in~\eqref{eq:assumption-scale-m} to be constant. With that choice, the function~$g(t)$ from Section~\ref{sec:relaxation-within-phase} would be proportional to~$t$, yielding the above exponential decay. The maximum with~$n=N^{1/d}$ in Section~\ref{sec:relaxation-within-phase} is replaced by a maximum with $\beta \log N$, as the radius of a random graph on $N$ vertices is logarithmic. 
\end{remark}

\subsubsection*{Hitting time of zero magnetization}
We begin by establishing the analogues of Claim~\ref{clm:monotonicity-relations} and Lemma~\ref{lem:zero-magnetization-hitting-time}, to circumvent any non-monotonicity issues that may arise in the analysis of the dynamics. We first recall that the low-temperature Ising model on a random $\Delta$-regular graph has a magnetization bottleneck separating the plus and minus phases. 
Recall that $M(\sigma) = \sum_{v\in V(\cG)} \sigma_v$ is the total magnetization of the configuration $\sigma$. The following result is implicit in the work~\cite{CHK19}, and for completeness we include a proof in Appendix~\ref{app:random-graph-magnetization-ld}:
with probability $1-o(1)$, $\cG\sim \Prrg$ is such that for all $\beta$ sufficiently large, there exists $\epsilon(\beta)>0$ such that
\begin{align}\label{eq:magnetization-ld-random-graph}
    \pi(N^{-1} |M(\sigma)|\notin [-\epsilon,\epsilon]) \le Ce^{ - N/C}\,.
\end{align}

We can now deduce that the hitting time to $M(\sigma) = 0$ for the chains under consideration is typically exponentially large. For the restricted Glauber dynamics $\widehat X_{t}^{x_0}$, let $\widehat \tau^{x_0}$ denote the hitting time of $\partial \widehat \Omega$. 

\begin{lemma}\label{lem:zero-magnetization-hitting-time-random-graph}
For every $\cG$, for every $0\le t\le \widehat \tau^{\widehat \pi}$, we have $\widehat X_{t}^{\widehat \pi} = X_{t}^{\widehat \pi}\le X_{t}^{\plus} = \widehat X_{t}^{\plus}$ under the grand coupling.
Moreover, for every $\beta$ large and every $\cG$ such that~\eqref{eq:magnetization-ld-random-graph} holds, for all $t\ge 0$,  $\mathbb P(\widehat \tau^{\widehat \pi}\le t) \le C(t\vee 1)e^{ - N/C}$. 
\end{lemma}

\begin{proof}
The first claim follows immediately from the grand coupling of the Ising Glauber dynamics and the restricted Glauber dynamics from different initializations. To prove the second claim, we reason exactly as in the proof of Lemma~\ref{lem:zero-magnetization-hitting-time}. Namely, either the number of clock rings by time $t$ is more than $N^2 (t\vee 1)$ (this has probability $Ce^{ - N/C}$), or there are at most $N^2 (t\vee 1)$ many clock rings and we can take a union bound over these rings to find that 
$$\mathbb P(\widehat \tau^{\widehat \pi}\le t) \le C e^{ - N/C} + N^2 (t\vee 1) \widehat \pi (\partial \widehat \Omega)\le C'e^{ - N/C'}\,,$$
where the last inequality used~\eqref{eq:magnetization-ld-random-graph}  to bound $\widehat \pi(\partial \widehat \Omega)$.
\end{proof}

We are now in a position to prove Proposition~\ref{prop:mixing-rate-random-graph} and in turn conclude Theorem~\ref{thm:random-graph-mixing}. 

\begin{proof}[\textbf{\emph{Proof of Proposition~\ref{prop:mixing-rate-random-graph}}}]
Our proof of Proposition~\ref{prop:mixing-rate-random-graph} will follow the proofs of Theorems~\ref{thm:Ising-torus-restatement}--\ref{thm:Ising-mixing-restricted-chain} very closely. Indeed, essentially all the changes in this proof are simply due to the differences in the radius and the non-uniformity of balls in the random graph; however, we include the details for completeness.  

To begin, fix $\beta>0$ sufficiently large, $c_\star$ sufficiently small per Lemmas~\ref{lem:wsm-within-phase-random-graph}--\ref{lem:local-mixing-random-graph}  and let $\Gamma$ denote the set of simple $\Delta$-regular graphs on $N$ vertices such that the properties in Lemmas~\ref{lem:wsm-within-phase-random-graph}--\ref{lem:local-mixing-random-graph} and~\eqref{eq:magnetization-ld-random-graph} hold. By those results, we have 
\begin{align*}
    \Prrg(\cG \in \Gamma) = 1-o(1)\,.
\end{align*}
We will show that the claims of Proposition~\ref{prop:mixing-rate-random-graph} hold uniformly over all $\cG \in \Gamma$. To that end, fix any $\cG \in \Gamma$. 

We first consider the single site marginals of the restricted chain $\widehat X_t^{\plus}$, analogously to Proposition~\ref{prop:single-site-relaxation}. Fix $t>0$ and any $v\in V(\cG)$ and let 
\begin{align*}
    \widehat P_t f_v(\plus) := \mathbb P(\widehat X_{t}^{\plus}(v) = +1) - \widehat\pi(\sigma_v = +1)\,.
\end{align*}
For any $r = r(t)\le \frac{1}{3}\log_d N$, by monotonicity of the Glauber dynamics we can write
\begin{align*}
    \widehat P_t f_v(\plus) \le \mathbb P(\widehat X_{t}^{\plus}(v) = +1) - \mathbb P(X_{t}^{\plus}(v) = +1) + \mathbb P(X_{B_r^{\plus}(v),t}^{\plus}(v) = +1) - \widehat \pi(\sigma_v = +1)\,,
\end{align*}
where we recall that $X_{B_r^{\plus}(v)}^{\plus}$ is the Glauber dynamics on $B_r(v)$ with $\plusone$ boundary condition. By Lemma~\ref{lem:zero-magnetization-hitting-time-random-graph}, the difference of the first two terms is at most $\mathbb P(\widehat \tau^{\widehat \pi} \le t)$ so a triangle inequality gives
\begin{equation}
\begin{aligned}
    \widehat P_t f_v(\plus) \le \mathbb P(\widehat \tau^{\widehat \pi} \le t) & + |\pi_{B_r^{\plus}(v)}(\sigma_v = +1) - \widehat \pi(\sigma_v = +1)| \\ %\nonumber\\ 
    & + |\mathbb P(X_{B_r^{\plus}(v)}^{\plus}(v) = +1) - \pi_{B_r^{\plus}(v)}(\sigma_v = +1)|\,. \label{eqn:asnew1}
\end{aligned}
\end{equation}
The first term on the r.h.s.~of~\eqref{eqn:asnew1} is at most $Ce^{ - N/C}$ while $t\le e^{N/K}$ for a sufficiently large $K$, by Lemma~\ref{lem:zero-magnetization-hitting-time-random-graph}. The second term is exactly the quantity bounded by WSM within a phase, so that by Lemma~\ref{lem:wsm-within-phase-random-graph} it is at most $Ce^{ - \beta r/C}$. 
By Lemma~\ref{lem:local-mixing-random-graph}, the third term is at most $C d^r e^{ - c_\star t}$.
Putting these estimates together, and taking $r = t/\beta \wedge (1/3) \log_d N$, we find that for $\beta$ sufficiently large,
\begin{align}\label{eq:random-graph-Ptfv-bound}
    \widehat P_t f_v (\plus) \le C e^{ - t/C} \qquad \mbox{for all $t\le \frac{\beta}{3} \log_d N$}\,.
\end{align}
Let us now use this to derive the desired mixing time conclusions. We first consider the mixing time of the restricted chain. Bounding the total variation distance by the probability of disagreement under an optimal coupling, and proceeding as in the first two displays of the proof of Theorem~\ref{thm:Ising-mixing-restricted-chain}, we have 
\begin{align}
    \|\mathbb P(\widehat X_{t}^{\plus}\in \cdot) - \widehat \pi\|_\tv \le \mathbb P(X_t^{\plus} \ne X_t^{\widehat \pi}) + \mathbb P(\widehat \tau^{\widehat \pi}\le t)\,,  \label{eqn:asnew2}
\end{align}
where the latter probability is under the grand monotone coupling. By Lemma~\ref{lem:zero-magnetization-hitting-time-random-graph}, the second term on the on the r.h.s.\ of~\eqref{eqn:asnew2} is at most $Ce^{ - N/C}$ while $t \le e^{N/K}$. By monotonicity of the coupling, the first term is at most 
\begin{align}
    \sum_{v\in V(\cG)} \Big( \mathbb P(X_{t}^{\plus}(v) = +1) &  - \mathbb P(X_t^{\widehat \pi}(v) = +1)\Big) \nonumber \\
    & \le \sum_{v\in V(\cG)} \Big(\mathbb P(\widehat X_{t}^{\plus}(v) = +1)  - \widehat \pi(\sigma_v = +1) + 2\mathbb P(\widehat \tau^{\widehat \pi}\le t)\Big)\,. \label{eqn:asnew3}
\end{align}
The first difference in the parentheses in~\eqref{eqn:asnew3} is exactly $\widehat P_t f_v(\plus)$, while the last term is again bounded by Lemma~\ref{lem:zero-magnetization-hitting-time-random-graph}. Plugging all this back into~\eqref{eqn:asnew2}, we obtain 
\begin{align*}
    \|\mathbb P(\widehat X_{t}^{\plus} \in \cdot) - \widehat \pi\|_\tv \le  N \max_{v\in V(\cG)} \widehat P_t f_v({\plus}) + C N e^{ - N/C}\,. 
\end{align*}
Finally, by~\eqref{eq:random-graph-Ptfv-bound}, we obtain
\begin{align}\label{eq:random-graph-within-phase-mixing}
    \|\mathbb P(\widehat X_{t}^{\plus}\in \cdot) - \widehat \pi\|_\tv \le C N e^{ - t/C}\,,\qquad \mbox{for all $t\le  (\beta/3)\log_d N$},
\end{align}
completing the verification of the first part~\eqref{eq:mixing-rate-restricted-random-graph} of the proposition.

We now lift this to a bound on the distance to stationarity for the (un-restricted) Glauber dynamics on $\cG$ when initialized from $\nu^{\pmb{\pm}}$. Towards this, we write 
\begin{align*}
   \|\mathbb P(X_{t}^{\nu^{\pmb{\pm}}}\in \cdot) - \pi\|_\tv &  \le  \|\frac 12 \mathbb P(X_{t}^{\plus}\in \cdot) + \frac 12 \mathbb P(X_{t}^{\minus}\in \cdot) - \frac 12 \widehat \pi - \frac 12 \widecheck \pi\|_\tv + \|\pi - \frac 12 (\widehat \pi + \widecheck \pi)\|_\tv \\ 
   & \le \frac 12 \|\mathbb P(X_{t}^{\plus}\in \cdot) - \widehat \pi\|_\tv + \frac 12 \|\mathbb P(X_{t}^{\minus}\in \cdot) - \widecheck \pi\|_\tv  + \pi(\partial \widehat \Omega) + \pi(\partial \widecheck \Omega)\,.
\end{align*}
By~\eqref{eq:random-graph-within-phase-mixing}, each of the first two quantities is bounded by $C N e^{- t/C}$ while $t\le (\beta/3)\log_d N$. By~\eqref{eq:magnetization-ld-random-graph}, the third and fourth terms are at most $Ce^{ - N /C}$ and can thus be absorbed into the first and second terms. Altogether, this gives the desired bound on the total variation distance $\|\mathbb P(X_{t}^{\nu^{\pmb{\pm}}}\in \cdot) - \pi\|_\tv$, concluding the proof of the second part~\eqref{eq:mixing-rate-unrestricted-random-graph} of the proposition. 
\end{proof}

\begin{proof}[\textbf{\emph{Proof of Theorem~\ref{thm:random-graph-mixing}}}]
Theorem~\ref{thm:random-graph-mixing} is an immediate corollary of the second result of Proposition~\ref{prop:mixing-rate-random-graph}. Namely, in order to make the right-hand side of~\eqref{eq:mixing-rate-unrestricted-random-graph} less than $\epsilon$ for any $\epsilon = \epsilon_N>0$ that decays as $\Omega(1/\mbox{poly}(N))$, it suffices to take continuous time $t = O(\log N)$, which translates to $O(N\log N)$ in discrete time per~\eqref{eq:continuous-time-discrete-time-comparison}.   
\end{proof}

\subsection{Random graph preliminaries}
It now remains to establish Lemmas~\ref{lem:wsm-within-phase-random-graph}--\ref{lem:local-mixing-random-graph}. Towards this, we recall some preliminary facts about random regular graphs that we will need. 
% We defer the proofs of these random graph lemmas to Appendix~\ref{app:random-graph-facts} as they are standard in the random graph literature (see e.g.,~\cite{BollobasBook,FriezeBook}).

The first such fact is their locally tree-like nature. There are various ways of expressing this property; we shall use the following.

\begin{definition}\label{def:treelike}
We say that the ball $B_r(v)$ in a $\Delta$-regular graph $\cG$ is {\it $K$-tree-like\/} if the removal of at most $K$ edges makes it a tree (or forest). \end{definition}

The following lemma is easy to verify using the configuration model for random $\Delta$-regular graphs (see e.g.,~\cite[Fact 2.3]{BG20}).

\begin{lemma}\label{lem:random-graph-treelike}
There exists $K(\Delta)>0$ such that, with probability $1-o(1)$, $\cG \sim \Prrg$ is such that $B_r(v)$ is $K$-tree-like for all $v\in V(\cG)$ and all $r\le (1/3)\log_d N$. 
\end{lemma}

The second fact we use about random $\Delta$-regular graphs is that they are typically expanders. Specifically, we will use the following measure of edge expansion to establish the needed equilibrium estimates at sufficiently low temperatures. 

\begin{definition}\label{def:edge-expansion}
A graph $\cG$ has {\it $\zeta$-edge expansion\/} if $\min_{S\subset V(\cG): |S|\le N/2} \frac{|E(S,S^c)|}{|S|}\ge \zeta$. Here $E(S,S^c)$ denotes the set of edges of $E(\cG)$ with one endpoint in $S$ and one endpoint in $V(\cG)\setminus S$.
\end{definition}

In~\cite[Theorem 1]{Bollobas-isoperimetric} an explicit $\zeta(\Delta)$ was found such that the random $\Delta$-regular graph has $\zeta$-edge expansion with high probability. We only need the following weaker existential statement.

\begin{lemma}\label{lem:random-graph-expansion}
For every $\Delta$, there exists $\zeta(\Delta)>0$ such that, with probability $1-o(1)$, $\cG \sim \Prrg$ has $\zeta$-edge expansion.
\end{lemma}

\subsection{WSM within a phase at low temperatures}
In this section we prove Lemma~\ref{lem:wsm-within-phase-random-graph}, establishing that WSM within a phase holds for the Ising model on $\cG \sim \Prrg$ at sufficiently low temperatures. 

We begin by using Peierls--type arguments to establish exponential tail bounds on the size of all but one cluster: the largest cluster will therefore dictate which phase the configuration is in. Towards that, let us define a polymer formalism for the Ising model on~$\cG$. 

\begin{definition}\label{def:polymer-representation}
A \emph{polymer} $\gamma$ of the Ising model with $\plusone$ boundary condition (respectively, in the plus phase) is a (maximal) connected set of vertices all of whose spins are~$-1$, also known as a \emph{minus cluster}. Let $\partial_e \gamma$ denote the outer edge boundary of $\gamma$, i.e., $\partial_e \gamma:=\{e=(v,w)\in E(\cG): v\in \gamma, w\notin \gamma\}$. The \emph{weight} of such a polymer is $w(\gamma) = e^{ - \beta |\partial_e\gamma|}$. 

The polymer representation for the Ising model on $B_r^{\plus}(v)$
is defined as follows. Let $V^-(\sigma)$ be the set of vertices that have spin $-1$ in $\sigma$, and let $\Gamma(\sigma)$ be the set of (maximal) connected components of the subgraph induced by $V^-(\sigma)$. This gives a 1-to-1 correspondence between configurations $\sigma$ on $B_r(v)$ and collections of polymers $\Gamma(\sigma)$. With this, we can express the Ising measure~\eqref{eqn:gibbs} as  
\begin{align}\label{eq:pi-B-r-low-temp-rep}
    \pi_{B_r^{\plus}(v)}(\sigma)\,\propto \prod_{\gamma\in \Gamma(\sigma)} w(\gamma)\,.
\end{align}

Similarly the polymer representation for the Ising model in the plus-phase on $\cG$ gives a 1-to-1 correspondence between configurations on $V(\cG)$ in $\widehat \Omega$ and collections of polymers $\Gamma$ such that $\sum_{\gamma \in \Gamma} |\gamma|\le N/2$. With  this, we can express the plus-phase Ising measure as
\begin{align}\label{eq:pi-hat-low-temp-rep}
\widehat \pi (\sigma)\,\propto\prod_{\gamma \in \Gamma(\sigma)} w(\gamma) \qquad \mbox{for $\sigma \in \widehat \Omega$}\,.
\end{align}
Given a configuration $\sigma$ with polymer representation $\Gamma(\sigma)$, let $\gamma_w= \gamma_w(\sigma)$ be the polymer in $\Gamma(\sigma)$ containing $w$ if such exists, and let $\gamma_w$ be empty otherwise. By convention, if $\gamma_w = \emptyset$ then $\partial_e \gamma_w = \emptyset$. 
\end{definition}

\begin{lemma}\label{lem:random-graph-clusters-exponential-tails}
Suppose $\cG$ has $\zeta$-edge expansion for $\zeta>0$.  Then, for any $v\in V(\cG)$ and $r\le (1/2)\log_d N$, any $w\in B_r(v)$, and any fixed collection of polymers $\Gamma_*$ in $B_r(v)$ such that $d(w,\Gamma_*)>1$, 
\begin{align*}
    \pi_{B_r^{\plus}(v)}\big(|\partial_e \gamma_w|\ge \ell \mid \Gamma_* \subset \Gamma(\sigma)\big) \le Ce^{ - \beta \ell/C}\,.
\end{align*}
Likewise, for any $w\in V(\cG)$ and any fixed collection of polymers $\Gamma_*$ in $V(\cG)$ such that $|\Gamma_*|\le N/2$ and such that $d(w,\Gamma_*)>1$,
\begin{align*}
    \widehat \pi\big(|\partial_e \gamma_w|\ge \ell \mid \Gamma_* \subset \Gamma(\sigma)\big)\le Ce^{ - \beta \ell/C}\,.
\end{align*}
\end{lemma}

\begin{proof}
We use a Peierls map together with the edge expansion properties guaranteed by Lemma~\ref{lem:random-graph-expansion} to argue the bounds of the lemma.  Let us begin by establishing the first bound of the lemma. For ease of notation, let $B^{\plus} = B_r^{\plus}(v)$. For each $l\ge \ell$, let $\Omega_{*,w,l}$ denote
the set of  configurations $\sigma$ having $\Gamma_* \subset \Gamma(\sigma)$ and $|\partial_e \gamma_w(\sigma)|=  l$.  Let $\Phi$ be the following map: for every $\sigma\in \Omega_{*,w,l}$, let $\Phi(\sigma)$ be the Ising configuration whose polymer representation is given by $\Gamma(\sigma)\setminus \gamma_w(\sigma)$. 
 Observe that for every $\sigma \in \Omega_{*,w,l}$, the configuration $\Phi(\sigma)$ has $\Gamma_* \subset \Gamma(\Phi(\sigma))$. 

On the one hand, the energy gain from the map $\Phi$ is $\beta l$, i.e., by~\eqref{eq:pi-B-r-low-temp-rep}, for every $\sigma \in \Omega_{*,w,l}$, 
\begin{align}\label{eq:peierls-energy-gain}
    \frac{\pi_{B^{\plus}}(\sigma)}{\pi_{B^{\plus}}(\Phi(\sigma))} = e^{ - \beta |\partial_e \gamma_w(\sigma)|}  = e^{ - \beta l}\,.
\end{align}
At the same time, we claim that the multiplicity of the map is bounded by an exponential in $l$. Observe that the pre-image of any configuration $\sigma'\in \Phi(\Omega_{*,w,l})$, i.e., every configuration $\sigma \in \Phi^{-1}(\sigma')$, is uniquely identified by the polymer $\gamma_w(\sigma)$. We can therefore enumerate the set of all pre-images $\sigma\in \Omega_{*,w,l}$ by the number of connected vertex sets $\gamma$ containing $w$ and having $|\partial_e\gamma| = l$. Any such polymer must have $|\gamma|\le l/\zeta$ since $\cG$ has $\zeta$-edge expansion and $|B_r(v)|\le N/2$. Since $\cG$ has degree at most $\Delta = d+1$, the number of connected vertex subsets of size $l/\zeta$ containing a given vertex~$w$ is easily bounded by $d^{l/\zeta}$. 

We can therefore write
\begin{align}
     \pi_{B^{\plus}}\big(|\partial_e \gamma_w|\ge\ell , \Gamma_* \subset \Gamma(\sigma)\big) =  \sum_{l\ge \ell} \sum_{\sigma \in \Omega_{*,w,l}} \pi_{B^{\plus}}(\sigma) & \le \sum_{l \ge \ell} \sum_{\sigma \in \Omega_{*,w,l}} \pi_{B^{\plus}}(\Phi(\sigma)) e^{ - \beta l} \nonumber \\ 
     & \le \sum_{l\ge \ell} \sum_{\sigma' \in \Phi(\Omega_{*,w,l})} \pi_{B^{\plus}}(\sigma') d^{l/\zeta} e^{ - \beta l}\,. \label{eqn:asnew4}
 \end{align}
 Here, the first inequality used~\eqref{eq:peierls-energy-gain} and the second inequality used the multiplicity bound. Since every element $\sigma'$ of $\Phi(\Omega_{*,w,l})$ has $\Gamma_* \subset \Gamma(\sigma')$, \eqref{eqn:asnew4} yields
 \begin{align*}
     \pi_{B^{\plus}}\big(|\partial_e \gamma_w|\ge\ell , \Gamma_* \subset \Gamma(\sigma)\big) \le \pi_{B^{\plus}} (\Gamma_* \subset \Gamma(\sigma)) \sum_{l\ge \ell} d^{l/\zeta} e^{-\beta l} \le Ce^{ - \beta \ell/C}\pi_{B^{\plus}} (\Gamma_* \subset \Gamma(\sigma))\,.
 \end{align*}
 Dividing both sides by $\pi_{B^{\plus}}(\Gamma_* \subset \Gamma(\sigma))$, we obtain the first bound of the lemma. 

The proof of the second bound of the lemma follows by an identical argument using~\eqref{eq:pi-hat-low-temp-rep} together with the observation that, for any $\sigma \in  \widehat \Omega_{*,w,l}$ (where the addition of the $\widehat \cdot$ indicates that $\sum_{\gamma \in \Gamma(\sigma)} |\gamma|\le N/2$, so that $M(\sigma) \ge 0$), the application of the map $\Phi$ preserves the property that the magnetization is non-negative as it only increases the number of $+1$ spins. 
\end{proof}

\begin{proof}[\textbf{\emph{Proof of Lemma~\ref{lem:wsm-within-phase-random-graph}}}]
Suppose $\cG$ is a $\Delta$-regular graph with $\zeta$-edge expansion, for $\zeta(\Delta)>0$ given by Lemma~\ref{lem:random-graph-expansion}; by Lemma~\ref{lem:random-graph-expansion}, $\cG\sim \Prrg$ satisfies this property with probability $1-o(1)$. 
Fix $v\in V(\cG)$ and $r\le (1/2)\log_d N$. For ease of notation, let $B= B_r(v)$. 
We will prove the WSM within a phase property by constructing a coupling of $\sigma^+\sim \pi_{B^{\plus}}$ and $\widehat \sigma \sim \widehat \pi$ such that $\sigma_v^+ = \widehat \sigma_v$ except with probability $Ce^{ - r/C}$. 

Let $\widehat \cE$ be the set of configurations on $V(\cG)\setminus B$ such that if they are concatenated with the all-minus configuration on $B$, they belong to $\widehat \Omega$. The coupling $(\sigma^+,\widehat \sigma)\sim \mathbb P$ is constructed as follows: 
\begin{enumerate}
    \item Sample the configuration $\widehat \sigma$ on $B^c := V(\cG) \setminus B$. 
    \item If $\widehat \sigma(B^c)\notin \widehat \cE$, then independently sample $\sigma^+(B)\sim \pi_{B^{\plus}}$ and $\widehat \sigma(B)\sim \widehat\pi (\cdot \mid \widehat \sigma(B^c))$.  
    \item If $\widehat \sigma(B^c)\in \widehat \cE$, then independently, from $\partial B$ inwards, sample the spins of $(\sigma^+,\widehat \sigma)$ in the minus component of $\partial B$ (i.e., the connected components of minus clusters containing some vertex in $\partial B$) in the configuration $\widetilde \sigma := \sigma^+ \wedge \widehat \sigma$. This can be done by iteratively revealing the vertex sets $(V_i^-)_i$ as follows: let $V_0^- := \partial B$; for $i\ge 1$, let $v_{i-1}$ be an arbitrary vertex in $V_{i-1}^-$, and sample (independently) the values of $\sigma^+_{v_{i-1}}$ and $\widehat \sigma_{v_{i-1}}$ conditionally on their values on $V_{i-1}^-$; set 
    \begin{align*}
        V_{i}^- := \begin{cases} V_{i-1}^{-}\setminus \{v_{i-1}\} & {\rm if\ } \widetilde \sigma_{v_{i-1}} = +1; \\ (V_{i-1}^{-} \setminus \{v_{i-1}\}) \cup (\bigcup \{w\sim v_{i-1}: w\notin \bigcup_{j<i} V_{j}^-\}) & {\rm if\ }\widetilde \sigma_{v_{i-1}} = -1.
        \end{cases}
    \end{align*}
    This iterative procedure terminates when $V_i^-$ is empty; call $S = \bigcup_{j} V_j^-$ the set of all spins that have then been revealed by the above procedure. Necessarily then, its inner boundary has been revealed to consist only of spins that are $+1$ in both $\sigma, \widetilde \sigma$, and no spin values of $B\setminus S$ have been revealed. 
    \item Draw the values of $\sigma^+(B\setminus S)$ conditionally on the revealed configuration (i.e., from the measure $\pi_{B\setminus S}^{\plus}$) and let $\widehat \sigma$ take the same values on $B\setminus S$.
\end{enumerate}
The fact that this is a valid coupling can be easily checked, with the main technicality to note being that the Ising measure $\widehat \pi$, when conditioned on an exterior configuration that belongs to $\widehat \cE$, satisfies the domain Markov property, since the conditioning on $\widehat \Omega$ can be dropped. This ensures that in step~(4) the distributions of $\sigma^+(B\setminus S)$ and $\widehat \sigma (B\setminus S)$ are the same conditional on their respective revealed configurations, namely conditional on all spins of $\partial (B\setminus S)$ being $+1$, allowing us to use the identity coupling. 

It is clear that, on the intersection of the event $\widehat \cE$ and the  event $\widehat \cE_{v,r}$ that $\widetilde \sigma$ does not have a minus path connecting $\partial B$ to $v$ so that $v\notin S$, we will have with probability one under the above coupling that $\sigma^+_v = \widehat \sigma_v$. 
It therefore suffices to bound the probability under the above coupling that this intersection of events does not occur. More precisely, by the coupling definition of total variation distance,
\begin{align}
    \|\pi_{B^{\plus}}(\sigma_v \in \cdot) - \widehat \pi(\sigma_v \in \cdot)\|_\tv \le \mathbb P((\widehat \cE \cap \widehat \cE_{v,r})^c) \le  \widehat \pi(\widehat \cE^c) + \mathbb P(\widehat \cE_{v,r}^c)\,.  \label{eqn:asnew5}
\end{align}
We begin with the first probability on the right-hand side of~\eqref{eqn:asnew5}. In order for $\widehat \cE^c$ to occur, it must be the case that the overall magnetization is within $|B_r(v)|$ of $0$; since $r\le (1/2)\log_d N$, $|B_r(v)|= o(N)$ and therefore, this probability is at most $Ce^{ - N/C}$ by~\eqref{eq:magnetization-ld-random-graph}. 

Turning now to the second event in~\eqref{eqn:asnew5}, we can bound its probability using the following claim that provides exponential tails on the clusters of $\widetilde \sigma$. 

\begin{lemma}\label{lem:sigma-tilde-exponential-tails}
Suppose $\beta_0$ is sufficiently large. Independently draw $\sigma^+\sim \pi_{B^{\plus}}$ and $\widehat \sigma\sim \widehat \pi$, and let $\widetilde \sigma := \sigma^+ \wedge \widehat \sigma$. For every $w\in B$ and all $\beta>\beta_0$, 
\begin{align*}
    \mathbb P(|\gamma_w (\widetilde \sigma)| \ge \ell) \le Ce^{ - \beta \ell/C}\,,
\end{align*}
for a constant $C(\Delta)>0$ independent of $\beta$.
\end{lemma}

Since the content of Lemma~\ref{lem:sigma-tilde-exponential-tails} is not far beyond that of Lemma~\ref{lem:random-graph-clusters-exponential-tails}, let us defer its proof and conclude the proof of Lemma~\ref{lem:wsm-within-phase-random-graph} assuming Lemma~\ref{lem:sigma-tilde-exponential-tails}. The event $\widehat \cE_{v,r}^c$ requires that the component $\gamma_v(\widetilde \sigma)$ intersects $\partial B$, which in turn requires that $|\gamma_v(\widetilde \sigma)|\ge r$. Plugging in $\ell =r$ in Lemma~\ref{lem:sigma-tilde-exponential-tails}, we see that $\mathbb P(\widehat \cE_{v,r}^c)\le Ce^{ - \beta r/C}$, which combined with~\eqref{eqn:asnew5} and our earlier bound on $\widehat \pi(\widehat \cE^c)$ implies  Lemma~\ref{lem:wsm-within-phase-random-graph}. 
\end{proof}

It now remains to deduce Lemma~\ref{lem:sigma-tilde-exponential-tails} from Lemma~\ref{lem:random-graph-clusters-exponential-tails}. 

\begin{proof}[\textbf{\emph{Proof of Lemma~\ref{lem:sigma-tilde-exponential-tails}}}]
Our goal is to establish an exponential tail bound on the minus clusters of
$\widetilde \sigma$. Any minus cluster of $\widetilde \sigma$ containing $v$ can be expressed via the following procedure: 
\begin{enumerate}
    \item The zero'th generation minus cluster given by $G_0 = \gamma_v(\sigma^+)$ if it is non-empty, and by $\gamma_v(\widehat \sigma)$ otherwise. Suppose it is given by $\gamma_v(\sigma^+)$ (otherwise swap the roles of $\sigma^+$ and $\widehat \sigma$ in what follows). 
    \item Iteratively, for each odd $i$, define the $i$'th generation of minus clusters, denoted $G_i$, as the union of all minus clusters of $\widehat \sigma$ that contain vertices of $\partial(\bigcup_{j<i} G_j) = \{w: d(w,\bigcup_{j<i} G_j) = 1\}$. 
    \item For each even $i$, define the $i$'th generation of minus clusters, denoted $G_i$, as the union of all minus clusters of $\sigma^+$ that contain vertices of $\partial(\bigcup_{j<i} G_j) = \{w: d(w,\bigcup_{j<i} G_j) = 1\}$. 
\end{enumerate}
In this manner, $\gamma_v(\widetilde \sigma)$ is expressed as the following disjoint union: 
\begin{align*}
    \gamma_v(\widetilde \sigma) = \bigcup G_j = \bigcup_{i\in 2\mathbb N} \bigcup_j \gamma_{v_{i,j}}(\sigma^+) \cup \bigcup_{i-1 \in 2\mathbb N} \bigcup_j \gamma_{v_{i,j}}(\widehat \sigma)\,.
\end{align*}
Here, $v_{i,j}$ are representative vertices of the constituent components of $G_i$, having $v_{i,j} \in \partial (\bigcup_{k<i} G_k)$ and 
$d(v_{i,j}, G_{i-1})=1$. (These representative vertices are uniquely determined by $G_i$ using, say, an arbitrary ordering of the vertices of $B$.)

Then, the probability that $\gamma_v(\widetilde \sigma)$ is given by any fixed set is at most 
\begin{align*}
    \mathbb P\Big(\gamma_v(\widetilde \sigma) = \bigcup_{i,j}\gamma_{v_{i,j}}\Big) \le \pi_{B^{\plus}}\Big(\bigcup_{i\in 2\mathbb N} \gamma_{v_{i,j}}\in \Gamma(\sigma)\Big) \widehat \pi \Big(\bigcup_{i-1\in 2\mathbb N} \gamma_{v_{i,j}}\in \Gamma(\sigma)\Big)\,.
\end{align*}
In order for the cluster containing $v_{i,j}$ to be $\gamma_{v_{i,j}}$ it must have $|\partial_e \gamma_{v_{i,j}}| \ge \zeta |\gamma_{v_{i,j}}|$ by the $\zeta$-edge expansion assumption on $\cG$. Therefore, by repeated application of Lemma~\ref{lem:random-graph-clusters-exponential-tails} (iteratively conditioning on $\Gamma_{i,j} := \bigcup_{l<i}\bigcup_{k<j} \{\gamma_{v_{l,k}}\}$ being a subset of $\Gamma(\sigma)$),  
\begin{align*}
   \mathbb P\Big(\gamma_v(\widetilde \sigma) = \bigcup_{i,j}\gamma_{v_{i,j}}\Big) \le  C\exp\Big( - \beta \zeta \sum_{i,j} |\gamma_{v_{i,j}}| /C\Big)\,.
\end{align*}
Let us now enumerate the number of possible choices of $\gamma_{v_{i,j}}$ having $\sum_{i,j} |\gamma_{v_{i,j}}| = L$. These can be enumerated by first choosing $(L_i)_{i}$ having $\sum_{i} L_i = L$ and for each $i$, $L_{i,j}$ having $\sum_{j} L_{i,j} = L_i$ such that $|\gamma_{v_{i,j}}| = L_{i,j}$ for each $i,j$. There are, crudely, at most $2^{2L}$ many choices for the sequence $(L_{i,j})_{i,j}$. Now, for each $i \ge 1$, enumerate first over the at most $2^{d L_{i-1}}$ many subsets $(v_{i,j})_j \in \partial G_{i-1}$ (in the $i=0$ case, there is no such choice as $(v_{0,j})_j$ is the singleton $v$) and then for each $v_{i,j}$ enumerate over the at most $d^{L_{i,j}}$ many choices of vertex subsets containing $v_{i,j}$ of size $L_{i,j}$. Putting all the above together, we obtain
\begin{align*}
    \mathbb P(|\gamma_v(\widetilde \sigma)|\ge \ell) \le \sum_{L \ge \ell}  Ce^{ - \beta \zeta L/C} 2^{2L} 2^{d L} d^{L} \le C' e^{ - \beta \ell /C'}\,,
\end{align*}
for some other constant $C'$ possibly depending on $\Delta$, exactly as desired. 
\end{proof}

\subsection{Local mixing with plus boundary condition} 
In this subsection, we establish that the log-Sobolev constant of the Ising Glauber dynamics on a $K$-tree-like ball of radius $r$ of $\cG$, with $\plusone$ boundary condition, is $\Omega(1)$. This will directly imply Lemma~\ref{lem:local-mixing-random-graph}. Recall that, 
for a Markov chain on a finite state space $\Omega$ with transition matrix $P$, reversible with respect to a distribution $\mu$, the {\it Dirichlet form\/} of any function $f:\Omega \to \mathbb R$ is defined by
\begin{align}\label{eq:Dirichlet-form}
    \mathcal E(f,f) := \frac{1}{2}\sum_{\omega,\omega'\in \Omega} \mu(\omega) P (\omega, \omega') (f(\omega) - f(\omega'))^2\,,
\end{align}
and its \emph{log-Sobolev constant} is given by 
\begin{align}\label{eq:lsi-constant}
    \alpha(P) := \inf_{f: \mbox{Ent}_\mu[f^2]\ne 0} \frac{\mathcal E(f,f)}{\mbox{Ent}_\mu[f^2]}\,, \qquad \mbox{where} \qquad \mbox{Ent}_{\mu} [f^2] = \mathbb E_{\mu}\Big[f^2 \log \frac{f^2}{\mathbb E_{\mu}[f^2]}\Big]\,.
\end{align}
The log-Sobolev constant of a Markov chain governs its rate of convergence to equilibrium in total-variation distance; namely (see, e.g.,~\cite{SClecture-notes}) for any $\alpha<\alpha(P)$, 
    \begin{align}\label{eq:LS-implies-tv-mixing}
        \max_{x_0} \|\mathbb P(X_t^{x_0}\in \cdot) - \mu\|_\tv\le \sqrt{2} e^{-\alpha t} \Big(\log \frac{1}{\min_{x} \mu(x)}\Big)^{1/2}\,.
    \end{align}

Combining~\eqref{eq:LS-implies-tv-mixing} with the trivial bound $\min_{\sigma} \pi_{B_r^{\plus}(v)}(\sigma) \ge e^{ - \Omega(d^r)}$, we have reduced proving Lemma~\ref{lem:local-mixing-random-graph} to establishing the following fact about
the corresponding log-Sobolev constants. 
\begin{lemma}\label{lem:local-LS-constants}
There exists $\alpha_\star(\beta,\Delta)>0$ such that with probability $1-o(1)$, $\cG \sim \Prrg$ is such that 
\begin{align*}
    \min_{v\in V(\cG)} \min_{r\le (1/3) \log_d N} \alpha\big(P_{B_r^{\plus}(v)}\big) \ge \alpha_\star >0\,,
\end{align*}
where $P_{B_r^{\plus}(v)}$ is the transition matrix of the Ising Glauber dynamics on $B_r^{\plus}(v)$. 
\end{lemma}

\begin{proof}
Let $e_1, e_2,...,e_K$ be edges whose removal from $B_r(v)$ results in a forest. In the breadth-first exploration of $B_r(v)$, each of these edges goes from a vertex at some height $h_1(e)$ to $h_2(e) \in \{h_1(e),h_1(e)+1\}$.  For each such edge~$e_k$, let $\cT_{e_k,1}$ and $\cT_{e_k,2}$ be two full $\Delta$-regular trees of depths $h_1(e_k)-1$ and $h_2(e_k)-1$ respectively, with $\plusone$ boundary condition. Let $H_{r,v}^{\plus}$  be the graph given by the disjoint union of $B_r(v)$ and $\{\cT_{e_k,1},\cT_{e_k,2}\}_{k}$, each with $\plusone$ boundary condition. We claim that the following sequence of graph modifications will change $H_{r,v}^{\plus}$ into a full $\Delta$-regular tree having $\plusone$ boundary condition at depth~$r$. 
\begin{enumerate}
    \item Take $B_r(v)$, and delete all edges $e_1,...,e_K$. 
    \item For each $i$, let $v_1$ and $v_2$ respectively denote the endpoints of~$e_i$ at heights $h_1(e_i)$ and $h_2(e_i) \in \{h_1(e_i), h_1(e_i)+1\}$. 
    \item Add an edge connecting $v_1$ to the root of $\cT_{e_i,1}$ and an edge connecting $v_2$ to the root of $\cT_{e_i,2}$. 
\end{enumerate}
In this manner, the vertices of $H_{r,v}$ are identified with the vertices of a full $\Delta$-regular tree of depth $r$,  and the symmetric difference of their edge sets has size $3K$. In particular, since the vertices are identified with one another, the corresponding sets of Ising model configurations are the same, and we have the following bound on the Radon--Nikodym derivative between the respective probability distributions:
\begin{align*}
    \max_{\sigma}\Bigg|\frac{\pi_{H_{r,v}^{\plus}}(\sigma)}{\pi_{T_r^{\plus}}(\sigma)} \vee \frac{\pi_{T_r^{\plus}}(\sigma)}{\pi_{H_{r,v}^{\plus}}(\sigma)}\Bigg| \le e^{ 6 \beta K}\,.
\end{align*}
A similar bound applies to the ratio of the transition probabilities for the respective Ising Glauber dynamics on $T_r^{\plus}$ and $H_{r,v}^{\plus}$, and
by~\cite[Theorem 4.1.1]{SClecture-notes}, the above also bounds the ratio of the entropies of the two distributions. In particular, we deduce that the log-Sobolev constants of the Ising Glauber dynamics on $T_r^{\plus}$ and $H_{r,v}^{\plus}$ are within a factor of $e^{C\beta K}$ of one another. By~\cite[Theorem 1.2]{MSW-trees-bc}, the log-Sobolev constant on $T_r^{\plus}$ is $\Omega(1)$, so the log-Sobolev constant on $H_{r,v}^{\plus}$ is $\Omega(e^{ - C\beta K})$, which is $\Omega(1)$ when $K = O(1)$. 

Now observe that the Ising Glauber dynamics on $H_{r,v}^{\plus}$ decomposes into a product of Ising Glauber dynamics on $B_r^{\plus}(v)$, $(\cT_{e_k,1})_k$ and $(\cT_{e_k,2})_k$. As such, letting $\alpha_H,\alpha_{B}, (\alpha_{k,1})_k$ and $(\alpha_{k,2})_k$ denote the log-Sobolev constants of each of these Glauber dynamics, we obtain the desired inequality 
\begin{align*}
    \alpha_B \ge \min\Big\{\alpha_B, \min_k \alpha_{k,1}, \min_k \alpha_{k,2}\Big\} = \alpha_H  = \Omega(e^{-C\beta K})\,.
\end{align*}
Here, the first equality used the classical fact that the log-Sobolev constant of a product Markov chain is the minimum of the log-Sobolev constants of each of the constituent chains. 
\end{proof}

\appendix

\section{Deferred proofs using the random-cluster coarse-graining}\label{app:deferred-proofs}
In this section, we provide standard proofs that were deferred from Section~\ref{sec:Ising-wsm-within-phase}. 

\subsection{Properties of the $k$-good coarse-graining} 
\begin{proof}[\textbf{\emph{Proof of Observation~\ref{obs:connected-component-good-blocks}}}]
This follows from the fact that if $x\sim_k y$ (say, without loss of generality, that $y = x+ (k,0,...,0)$) then any subset of $B_x$ intersecting its two opposite sides in the $e_1$ direction must have intersection of size at least $k$ with $B_y$. From there, the definition of the $k$-good coarse graining $\eta(\omega)$ implies that if $\eta_x(\omega) = \eta_y(\omega) = 1$, the largest component of $B_x$ must have intersection of size at least $k$ with $B_y$ and therefore must also be connected to the largest component of $B_y$. 
\end{proof}

\begin{proof}[\textbf{\emph{Proof of Lemma~\ref{lem:separating-surface-disconnects-information}}}]
We begin with the preliminary observation that 
\begin{align*}
    d(\Lambda_m\setminus (\Gamma \cup \mathsf{Ext}(\Gamma)), \mathsf{Ext}(\Gamma))\ge 2k\,.
\end{align*}
To see this, notice by construction that every vertex in $\mathsf{Int}(\Gamma_k)$ must be at $\sim_k^\star$ distance at least $2$ from $\mathsf{Ext}(\Gamma_k)$, so that their distance in $\Lambda_m$ is at least $2k$. It is then a simple geometric observation that the distance between the above two sets is attained by the distance between their coarse-grained vertex sets $\Lambda_m^{(k)}\setminus (\Gamma_k \cup \mathsf{Ext}(\Gamma_k))$ and $\mathsf{Ext}(\Gamma_k)$, since all ``corners" of $\Gamma$ must be at vertices of $k\mathbb Z^d$.

Now let $\omega^\one$ (respectively, $\omega^\zero$) be the random-cluster configuration that modifies $\omega$ by setting all edges in $\mathsf{Ext}(\Gamma)\setminus \Gamma$ to open (respectively, closed). It suffices to show that $\omega^\one, \omega^\zero$ induce the same boundary conditions on $\partial (\Lambda_m \setminus (\Gamma \cup \mathsf{Ext}(\Gamma)))$. Suppose that two vertices on the boundary $\partial (\Lambda_m \setminus \Gamma \cup \mathsf{Ext}(\Gamma))$ are connected to one another in $\omega^\one$, but not in $\omega^\zero$. Since $\Gamma_k$ is an open separating $k$-surface, they must both be connected in $\omega(\Gamma)$ to $\mathsf{Ext}(\Gamma)$ without being part of the same cluster of $\omega(\Gamma)$. That necessitates two distinct components of size at least $2k$ in $\omega(\Gamma)$. We claim that this contradicts the fact that $\eta_x(\omega) =1$ for all $x\in \Gamma_k$. This follows immediately from Observation~\ref{obs:connected-component-good-blocks} if we establish that the open separating $k$-surface is a $k$-connected set. 
This connectedness is a consequence of the duality between $\star$-connectivity and connectivity in $\mathbb Z^d$, whereby any witness to the absence of a $\star$-connection between two sets must be connected, and vice-versa (see~\cite{DeuschelPisztora96} where this duality was proved). 
\end{proof}

\subsection{Using coarse-graining to couple random-cluster configurations} 
Our aim in this section is to prove Lemma~\ref{lem:E-very-good-probability}. Our approach is to construct an explicit coupling that ensures that items (1)--(2) in the lemma hold on $\cE_{\textsc{vg}}$, and to bound the probability of $\cE_{\textsc{vg}}^c$ directly through this coupling. The coupling consists of an algorithmic procedure for revealing the outermost very good separating surface of $(\omega,\omega')$, and is defined as follows. 

\begin{definition}\label{def:very-good-coupling}
Let $E_j$ be the set of edges $e$ for which we have sampled $(\omega_e,\omega_e')$ by step $j$ of the revealing process, and initialize it as $E_0 = \emptyset$. Initialize $F_0 = \partial \Lambda_{m}^{(k)}$ and update $F_j$ as follows: 
\begin{itemize}
    \item If $F_{j-1}\ne \emptyset$, select some $v_j\in F_{j-1}$ and sample 
  \begin{align*}
      \omega(E(B_{v_j}) \setminus E_{j-1}) &  \sim \pi^\rc_{\Lambda_m^\xi}(\omega(E(B_{v_j})\setminus E_{j-1}) \in \cdot \mid \omega(E_{j-1})) \qquad \mbox{and} \\ 
      \omega'(E(B_{v_j})\setminus E_{j-1}) & \sim \pi^{\rc}_{\Lambda_m^{\xi'}}(\omega(E(B_{v_j})\setminus E_{j-1}) \in \cdot \mid \omega'(E_{j-1}))
  \end{align*}
  using the standard monotone coupling (using the same uniform random variables to reveal the states of edges one at a time) on edges of $E(B_v)\setminus E_j$. Then, letting $N(v_{j}) = \{w\in \Lambda_{n-2k}^{(k)} : w\sim_\star^k v_j\}$, update 
    \begin{itemize}
        \item $E_j  = E_{j-1}\cup E(B_{v_j})$
        \item $F_j = F_{j-1} \setminus \{v_j\}$ if $B_{v_j}$ is \emph{very good}, and $F_j = F_{j-1} \cup N(v_j) \setminus \{v_1,...,v_j\}$ otherwise.
    \end{itemize}
    \item If $F_{j-1} = \emptyset$, set $\tau:= j-1$ and sample $$\omega(E(\Lambda_m)\setminus E_{j-1}) = \omega'(E(\Lambda_m)\setminus E_{j-1})\quad \sim \quad \pi^\rc_{\Lambda_m^\xi} \big(\omega(E(\Lambda_m) \setminus E_{j-1}) \in \cdot \mid \omega(E_{j-1})\big)\,.$$ 
\end{itemize} 
\end{definition}

\begin{lemma}\label{lem:very-good-coupling-well-defined}
The construction in Definition~\ref{def:very-good-coupling} gives a valid monotone coupling of $(\omega, \omega')$. On the event $\cE_{\textsc{vg}}$, if $\Gamma_k$ is the outermost such separating $k$-surface of very good blocks, then items (1)--(2) of Lemma~\ref{lem:E-very-good-probability} hold. 
\end{lemma}

\begin{proof}
The coupling in Definition~\ref{def:very-good-coupling} can straightforwardly be seen to be a valid, monotone coupling with $\omega(E(\Lambda_m)\setminus E_\tau) = \omega'(E(\Lambda_m)\setminus E_\tau)$: we refer to Lemma 3.3 of~\cite{DCGR20} for more details. Item (1) of Lemma~\ref{lem:E-very-good-probability} follows from Lemma~\ref{lem:separating-surface-disconnects-information}. Item (2), that $\omega(\Int(\Gamma))$ equals $\omega'(\Int(\Gamma))$, follows because, crucially, upon stopping, the (inner) boundary of $\bigcup_{j\le \tau} F_j$ is a \emph{very good} separating surface of $\Lambda_m^{(k)}\setminus \Lambda_{m/2}^{(k)}$ in $(\omega (E_{\tau}), \omega'(E_{\tau}))$ and therefore according to Lemma~\ref{lem:separating-surface-disconnects-information}, the distributions 
\begin{align*}
    \pi^\rc_{\Lambda_m^{\xi}} (\omega(E(\Lambda_m) \setminus E_{\tau}) \in \cdot \mid \omega(E_{\tau})) \qquad \mbox{and}\qquad \pi^\rc_{\Lambda_m^{\xi'}} (\omega'(E(\Lambda_m) \setminus E_{\tau}) \in \cdot \mid \omega'(E_{\tau}))
\end{align*}
coincide. 
\end{proof}

\begin{proof}[\textbf{\emph{Proof of Lemma~\ref{lem:E-very-good-probability}}}]
On the event $\cE_{\textsc{vg}}$, items (1)--(2) of Lemma~\ref{lem:E-very-good-probability} hold by Lemma~\ref{lem:very-good-coupling-well-defined}. 
It thus remains to bound the probability that $E_\tau \cap \Lambda_{\frac{m}{2}+k} \ne \emptyset$ as
\begin{align*}
    \mathbb P(\cE_{\textsc{vg}}^c) \le \mathbb P(E_\tau \cap \Lambda_{\frac{m}{2} + k} \ne \emptyset)\,.
\end{align*}
Observe, first of all, that if $E_\tau \cap \Lambda_{m/2 +k} \ne\emptyset$, there must exist a sequence of steps $j_1< j_2< ... <j_s$ with $s\ge m/16k$ such that $v_{j_i}$ is at distance exactly $3$ in the metric induced by $\sim_\star^k$ from $v_{j_{i-1}}$, and such that under the coupling, the $(B_{v_{j_i}})_i$ are not very good. We also claim that, uniformly over the configurations 
$$\omega\big(E(B_{v_{j_1}}) \cup \cdots \cup E(B_{v_{j_{i-1}}})\big), \omega'\big(E(B_{v_{j_1}})\cup \cdots \cup E(B_{v_{j_{i-1}}})\big)\,,$$
the probability that $B_{v_{j_i}}$ is very good is at least $1-Ce^{-k/C}$ for $k$ large enough. On the one hand, the probability of the block being $k$-bad is at most $Ce^{ - k/c}$ by Lemma~\ref{lem:good-whp-low-temp}, because $B_{2k}(v_{j_i})$ is disjoint from all the blocks $B_{v_{j_{1}}} \cup \cdots \cup B_{v_{j_{i-1}}}$.  On the other hand, since the above coupling is monotone, we have 
\begin{align*}
    \mathbb P\Big(\omega\big(E(B_{v_{j_i}})) \ne \omega'(E(B_{v_{j_{i}}})) \mid & \omega(E(B_{v_{j_1}})\cup\cdots\cup E(B_{v_{j_{i-1}}})\big),  \omega'\big(E(B_{v_{j_1}})\cup\cdots\cup E(B_{v_{j_{i-1}}})\big)\Big) \\
    & \le \sum_{e\in E(B_{v_{j_i}})} \pi^\rc_{B_{2k}^\one(v_{j_i})}(\omega(e) =1) - \pi^\rc_{B_{2k}^\zero(v_{j_i})}(\omega(e) =1)\,.
\end{align*}
The quantity on the right-hand side above is at most $Ce^{- k/C}$ for $k$ sufficiently large, by the random-cluster WSM property. Altogether then, the probability of $B_{v_{j_i}}$ being very good, conditional on both $\omega$ and $\omega'$ on $E(B_{v_{j_{1}}}),...,E(B_{v_{j_{i-1}}})$, is at least $1-Ce^{ -k/C}$ when $k$ is sufficiently large. 

Taking a union bound over all possible paths of $\sim_\star^k$ distance $3$ separated vertices of $\Lambda_{m-2k}^{(k)}$, of the event that those blocks are not very good in the coupling, we deduce that 
\begin{align*}
    \mathbb P(E_\tau \cap \Lambda_{m/2 +k}\ne \emptyset) \le (3^{3d} k^{-d})^{m/(16k)}\,,
\end{align*}
which, as long as $k$ is a sufficiently large constant that $k^{-d} <3^{-3d}$, is $\exp( - \Omega(m))$ as claimed. 
\end{proof}

\section{Bottleneck for the magnetization on random regular graphs}\label{app:random-graph-magnetization-ld}

In this section, we include a short proof of the magnetization bottleneck for the Ising model on the random $\Delta$-regular graph at sufficiently low temperatures. 

\begin{proof}[\textbf{\emph{Proof of~\eqref{eq:magnetization-ld-random-graph}}}]
Fix $\zeta>0$ from Lemma~\ref{lem:random-graph-expansion}, and suppose $\cG$ is a $\Delta$-regular random graph having $\zeta$-edge-expansion. By Lemma~\ref{lem:random-graph-expansion}, this has probability $1-o(1)$ under $\Prrg$. Let $\widehat \Omega_{\epsilon}$ be the set of all Ising configurations having total magnetization $M(\sigma) \in [0,\epsilon N]$, and analogously let $\widecheck \Omega_\epsilon$ be the set of all Ising configurations having $M(\sigma) \in [-\epsilon N, 0]$. Evidently, by symmetry it suffices for us to bound $\pi(\widehat \Omega_\epsilon)$ for $\epsilon(\beta)$ sufficiently small. 

We can clearly bound  
\begin{align*}
    \pi(\widehat \Omega_\epsilon) \le \frac{\pi(\widehat \Omega_\epsilon)}{\pi(\plusone)} \le |\Omega_\epsilon| \max_{\sigma \in \widehat \Omega_\epsilon} e^{ - \beta |C(\sigma)|}
\end{align*}
where we recall that $C(\sigma)$ is the set of edges in $E(\cG)$ connecting vertices of different spins. This in turn yields the inequality 
\begin{align*}
    \pi(\widehat \Omega_\epsilon) \le 2^{N} \max_{S \subset \cG: \frac N2 - \frac{\epsilon N}{2} \le |S| \le \frac N2} e^{ -\beta |\partial_e S|} \le 2^N e^{ - \beta \zeta (\frac 12- \frac \epsilon 2) N}\,, 
\end{align*}
where the second inequality used the $\zeta$-expansion of the graph $\cG$. This then yields the desired $Ce^{- N/C}$ bound on $\pi (\widehat \Omega_\epsilon \cup \widecheck  \Omega_\epsilon)$ by taking $\epsilon<1$ and $\beta$ sufficiently large depending on $\zeta$. 
\end{proof}

\bibliographystyle{abbrv}
\bibliography{references}

\end{document}